%% file: AB2b-v36-arxiv.tex
\documentclass[11pt, sort&compress]{elsarticle} 

\usepackage{fullpage}
\usepackage{amsfonts} 
\usepackage{graphics}

\usepackage{tikz}
\usetikzlibrary{arrows}
\usetikzlibrary{shapes}
\usetikzlibrary{positioning}

\usepackage{xr} 

\usepackage{refcheck} 
\norefnames	
\nocitenames 	
\ignoreunlbld 
\setoffmsgs 

\usepackage{mathtools} 

\usepackage{arydshln} 

\usepackage{url}      

\usepackage{graphicx}
\usepackage{amsmath, amsthm, amssymb}
\usepackage{amssymb,amsmath,graphics,color}

\usepackage{epsfig}
\usepackage{epstopdf}

\usepackage{multirow}

\usepackage{relsize} 

\usepackage{tocloft}
\renewcommand{\cftsecfont}{}

\usepackage{here} 

\allowdisplaybreaks[4]

\usepackage{xcolor}
\def\red #1{\textcolor{red}{#1}}
\def\blue #1{\textcolor{blue}{#1}}

\def\brown #1{\textcolor{brown}{#1}}

\def\black #1{\textcolor{black}{#1}}

\definecolor{darkbrown}{rgb}{.5,.1,.1} 
\definecolor{darkgreen}{rgb}{0,.6,0}

\def\darkgreen #1{\textcolor{darkgreen}{#1}}

\def\coltab{\black}            
\def\ptt #1{{\small\tt #1}}    

\def\EME #1{}

\def\EMEvquatorze #1{}

\def\EMEvder #1{}

\def\futur #1{}

\usepackage{tikz}

\def\boxit [#1]#2{\vbox{\hrule\hbox{\vrule
     \vbox spread #1{\vss\hbox spread#1{\hss #2\hss}\vss}%
        \vrule}\hrule}}

\newbox\algo

\font\algofont=cmtt10 scaled 1100

\newenvironment {algorithme}{\smallskip \smallskip
\bgroup\algofont\parindent= 0mm}{\egroup
\smallskip\smallskip}

\def\itemlabel#1#2{%
\global\expandafter\edef\csname #1\endcsname{#2}%
}

\def\itemref#1{%
\expandafter\csname #1\endcsname%
}

\newtheorem{thm}{Theorem}[section]
\newtheorem{cor}[thm]{Corollary} 
\newtheorem{prop}[thm]{Proposition} 
\newtheorem{lemma}[thm]{Lemma}
\newtheorem{remark}[thm]{Remark}

\newtheorem{observation}[thm]{Observation}
\newtheorem{definition}[thm]{Definition}

\def\newline{\hfill\break}

\def\newpage{\vfill\break}

\def\square{\hbox{\vrule\vbox{\hrule\phantom{o}\hrule}\vrule}}
\def\square{\hbox{\vrule\vbox{\hrule\phantom{O}\hrule}\vrule}}

\definecolor{shadecolor}{RGB}{0,0,150}

\def\Int{\mathrm{Int}}

\def\Ext{\mathrm{Ext}}

\def\min{\mathrm{min}}
\def\Min{\mathrm{min}}

\def\max{\mathrm{max}}

\def\ass{\mathrm{Part}}

\def\F{{\cal F}}
\def\X{\bullet}

\def\ep{\varepsilon}
\def\io{\iota}

\def\ov{\overrightarrow}
\def\s{\setminus}
\def\bk{\backslash}
\def\ovl{\overline}

\def\bs{\bigskip}
\def\ms{\medskip}
\def\ss{\smallskip}

\def\finsi{}

\def\plus{\uplus} 

\def\gbullet{{\times}}

\def\littlefullsquare{\tikz\draw[blue,fill=blue] (0,0) rectangle (0.13,0.13);}

\def\taille{1.15}

\def\bgbullet{\blue{\littlefullsquare}}

\def\plussquare{
	\tikz{
	\draw[blue, fill=blue!15] (0,0) rectangle (0.30,0.30);
	\draw[black, line width=0.5mm] (0.05,0.15) -- (0.25,0.15);
	\draw[black, line width=0.5mm] (0.15,0.05) -- (0.15,0.25);
	}
	}
\def\minussquare{\tikz{\draw[blue, fill=blue!15] (0,0) rectangle (0.30,0.30);\draw[black, line width=0.5mm] (0.05,0.15) -- (0.25,0.15);}}
\def\pbbullet{{{\plussquare}}}
\def\mbbullet{{{\minussquare}}}

\def\plussquarevide{
	\tikz{
	\draw[blue] (0,0) rectangle (0.30,0.30);
	\draw[black, line width=0.5mm] (0.05,0.15) -- (0.25,0.15);
	\draw[black, line width=0.5mm] (0.15,0.05) -- (0.15,0.25);
	}
	}
\def\minussquarevide{\tikz{\draw[blue] (0,0) rectangle (0.30,0.30);\draw[black, line width=0.5mm] (0.05,0.15) -- (0.25,0.15);}}
\def\pbvbullet{{{\plussquarevide}}}
\def\mbvbullet{{{\minussquarevide}}}

\def\plussquarenoir{
	\tikz{
	\draw[black, fill=black!15] (0,0) rectangle (0.30,0.30);
	\draw[black, line width=0.5mm] (0.05,0.15) -- (0.25,0.15);
	\draw[black, line width=0.5mm] (0.15,0.05) -- (0.15,0.25);
	}
	}
\def\minussquarenoir{\tikz{\draw[black, fill=black!15] (0,0) rectangle (0.30,0.30);\draw[black, line width=0.5mm] (0.05,0.15) -- (0.25,0.15);}}
\def\pbbulletn{{{\plussquarenoir}}}
\def\mbbulletn{{{\minussquarenoir}}}

\def\plussquarevidenoir{
	\tikz{
	\draw[black] (0,0) rectangle (0.30,0.30);
	\draw[black, line width=0.5mm] (0.05,0.15) -- (0.25,0.15);
	\draw[black, line width=0.5mm] (0.15,0.05) -- (0.15,0.25);
	}
	}
\def\minussquarevidenoir{\tikz{\draw[black] (0,0) rectangle (0.30,0.30);\draw[black, line width=0.5mm] (0.05,0.15) -- (0.25,0.15);}}
\def\pbvbulletn{{{\plussquarevidenoir}}}
\def\mbvbulletn{{{\minussquarevidenoir}}}

\def\pluscircle{
	\tikz{
	\draw[red,fill=red!15] (0.15,0.15) circle (0.17);
	\draw[black, line width=0.5mm] (0.05,0.15) -- (0.25,0.15);
	\draw[black, line width=0.5mm] (0.15,0.05) -- (0.15,0.25);
	}
	}
\def\minuscircle{
	\tikz{
	\draw[red,fill=red!15] (0.15,0.15) circle (0.17);
	\draw[black, line width=0.5mm] (0.05,0.15) -- (0.25,0.15);
	}
	}
\def\prbullet{{{\pluscircle}}}
\def\mrbullet{{{\minuscircle}}}

\def\pluscirclevide{
	\tikz{
	\draw[red] (0.15,0.15) circle (0.17);
	\draw[black, line width=0.5mm] (0.05,0.15) -- (0.25,0.15);
	\draw[black, line width=0.5mm] (0.15,0.05) -- (0.15,0.25);
	}
	}
\def\minuscirclevide{
	\tikz{
	\draw[red] (0.15,0.15) circle (0.17);
	\draw[black, line width=0.5mm] (0.05,0.15) -- (0.25,0.15);
	}
	}
\def\prvbullet{{{\pluscirclevide}}}
\def\mrvbullet{{{\minuscirclevide}}}

\font\ptirm=cmr10 scaled 900

\def\bul{$\bullet$ }

\def\G{{\ov G}}
\def\M{M}

\def\bysame{\rule{1cm}{0.1mm}}

\newcommand{\overbar}[1]{\mkern 1.5mu\overline{\mkern-1.5mu#1\mkern-1.5mu}\mkern 1.5mu}
\def\bar{\overbar}    

\def\eme#1{}
\def\emeref#1{}
\def\emenew#1{}
\def\emevder#1{}

\title{The Active Bijection \\ 
 2.b - Decomposition of activities for oriented matroids,\\ and general definitions of the active bijection}

\author[lirmm]{Emeric Gioan\corref{cor1}}
\ead{emeric.gioan@lirmm.fr}

\author[paris]{Michel Las Vergnas\corref{cor2}}

\cortext[cor1]{Email: emeric.gioan@lirmm.fr}
\cortext[cor2]{R.I.P.}
\address[lirmm]{CNRS, LIRMM, Universit\'e de Montpellier, France}
\address[paris]{CNRS, Paris, France}

\date{}

\begin{document}

\begin{abstract} 
\noindent
The active bijection for oriented matroids 
(and real hyperplane arrangements, and graphs, as particular cases)
is introduced and investigated  by the authors in a series of papers.
Given any oriented matroid defined on a linearly ordered ground set, we exhibit one particular of its bases, which we call its
 active basis,
 with remarkable properties. It preserves activities (for oriented matroids in the sense of Las Vergnas, for matroid bases in the sense of Tutte), as well as some active partitions of the ground set associated with oriented matroids and matroid bases. 
It yields a canonical bijection between classes of reorientations and bases 
(this bijection depends only on the reorientation class of the oriented matroid, that is on the non-signed pseudosphere arrangement in terms of a topological representation). It also yields 
a refined bijection between all 
reorientations and subsets of the ground set.
Those bijections are related to various Tutte polynomial expressions (in terms of usual and refined  activities for bases/subsets or reorientations, in terms of beta invariants of minors). They contain various noticeable bijections
involving orientations/signatures/reorientations and spanning trees/simplices/bases of a graph/real hyperplane arrangement/oriented matroid.
For instance, we obtain 
an activity preserving  bijection
between acyclic reorientations and no-broken-circuit subsets.
\EMEvder{laisser cette dernier phrase ?}%


In previous papers of this series, we defined the active bijection between bounded regions and uniactive internal bases by means of fully optimal bases (No. 1), and we defined a decomposition of activities for matroid bases by means of  active filtrations  (or active partitions) yielding particular sequences of minors (companion paper, No. 2.a).
The present paper is central in the series. 
First,
we define a decomposition of activities for oriented matroids, using the same sequences of minors, yielding a decomposition of an oriented matroid into bounded regions of minors.
Second, we use the previous results together to 
provide
the canonical and refined 
active bijections
alluded to above. 
%
%
We also give an overview and examples of the various results of independent interest involved in the construction. They arise as soon as the ground set of an oriented matroid is linearly ordered.

\end{abstract}

\maketitle   


\newpage

\hrule 
\vspace{-3mm}
\setcounter{tocdepth}{1}  
    \renewcommand*\contentsname{\normalsize Contents\vspace{-3mm}}
  \renewcommand\cftsecfont{\small\rm}
 \renewcommand{\cftsecpagefont}{\small\normalfont}
 \pagenumbering{arabic}
    \setlength{\cftbeforesecskip}{0cm}
    \setlength{\cftparskip}{0cm}

\tableofcontents

\setcounter{page}{2}

\ms
\hrule


\futur{faire corrections d'anglais, comme dans chapter, voir fichier anglais.tex  passage "SUR CHAPTER"}%

\section{Introduction and overview}
\label{sec:intro}



\EMEvder{comparer struture et intro de cet article a ABG2 \cite{ABG2} ou les three-levels sont beauocup plus acentues, ainsi que la flexibilite de la cosntuction}


\EMEvder{??? REMPLACER LE PLUS POSSIBLE on a ienarly ordered edge set PAR ordered, MEME DANS INTRO}%





\EMEvder{enlever italique de observation ? non car resultats cit?s}


\eme{a mettre ? Erratum : phd, FPSAC 03 + mention : supersolv, LP}%

\eme{REMPLACER LE PLUS POSSIBLE on a ienarly ordered element set PAR ordered, MEME DANS INTRO}%




The active bijection for oriented matroids 
(and real hyperplane arrangements, and graphs, as particular cases)
%
 is the subject of several papers by the present authors 
\cite{GiLV04, GiLV05, GiLV06, GiLV07, GiLV09, AB1, AB2-a, AB3, AB4, ABG2, ABG2LP}.
%
The general setting of this set of papers is to relate 
orientations/signatures/reorientations and spanning trees/simplices/bases of graphs/real hyperplane arrangements/oriented matroids,
and, more generally, to study oriented matroids as soon as they are defined on a linearly ordered set, in terms of structural, constructive, enumerative or bijective canonical properties.
%

In general, 
\EMEvder{Our main...}%
we map any oriented matroid on a linearly ordered set onto one particular of its  bases, which we call its \emph{active basis}. 
This allows us to define a canonical activity preserving correspondence between reorientations and bases of an oriented matroid,
with various related bijections, constructions and characterizations. 
%
%
%
The original motivation was to provide a bijective interpretation
and a structural understanding
of the equality of two classical expressions of the Tutte polynomial (detailed in Section \ref{sec:prelim}). The first is in terms of basis activities by Tutte \cite{Tu54} (extended to matroids by Crapo \cite{Cr69}):

\centerline{$\displaystyle t(M;x,y)=\sum_{\io,\ep}b_{\io,\ep}x^\io y^\ep$}
\noindent where $b_{\io,\ep}$ is the number of bases of $M$ 
with internal activity $\io$ and external activity $\ep$.
The second is in terms of reorientation activities by Las Vergnas \cite{LV84a}:
 
 \centerline{$\displaystyle t(M;x,y)=\sum_{\io,\ep}o_{\io,\ep}\ \Bigl({x\over 2}\Bigr)^\io \ \Bigl({y\over 2}\Bigr)^\ep$}
\noindent where $o_{\io,\ep}$ is the number of reorientations of $M$ with  dual-activity $\io$ and activity $\ep$.
This second expression contains various famous enumerative results from the literature, such as counting regions or acyclic reorientations, 
mainly by Winder \cite{Wi66}, Stanley \cite{St73}, Zaslavsky \cite{Za75}, and Las Vergnas \cite{LV77}.
%
\EMEvder{autres refs vite fait ?}%
Roughly, one can think of activities as situating bases and reorientations with respect to the minimal and maximal basis.
Much more details, either on related results or on specific references to the literature, are given in the introduction of each section of the paper, as they deal with separate aspects of the construction.
The rest of the introduction aims at presenting a practical and global 
 overview of these various features.
%

\bs

Let us first situate the authors' works 
on the subject.
The 
main papers are \cite{AB1, AB2-a, AB3, AB4} along with the present central one.
They deal with general oriented matroids and form a consistent series (details will be given  further on): \cite{AB1} (No. 1) deals with 
the bounded/uniactive case (that is, the case where $\io=1$ and $\ep=0$, or $\io=0$ and $\ep=1$); 
\cite{AB2-a} (No. 2.a) is a companion paper of the present one and deals with a decomposition of matroid bases
into uniactive bases of minors; 
the present paper (No. 2.b) deals with a decomposition of oriented matroids 
into bounded minors
and is central in the series as it uses the previous papers to define the active bijection in general; to be continued with \cite{AB3} (No. 3) that deals with elaborations on linear programming yielding an inverse to the construction of \cite{AB1} ; and with \cite{AB4} (No. 4) that deals with  deletion/contraction constructions and broader characterizations.\EMEvder{broader charcterizations???}
These papers are written in oriented matroid terms and 
mainly 
illustrated in pseudosphere/hyperplane arrangements from a topological/geometrical viewpoint (the reader may read the 
preliminary section 
of \cite{AB1} for a summary on this viewpoint).
\EMEvder{() a mettre dans prelim plutot? non bien ici pour lecteur}

These main papers are completed with papers dealing with particular cases: \cite{GiLV04} yields an easy case introduction to the subject by studying uniform and rank 3 oriented matroids; \cite{GiLV06} studies the case of  supersolvable real hyperplane arrangements and of particular Coxeter arrangements; and \cite{GiLV05, ABG2, ABG2LP} deal with graphs. The reader specifically interested in graphs is encouraged to read \cite{ABG2} which summarises 
the main results of the whole series formulated in this structure.
These papers are completed with short conference notes:
\cite{GiLV07} presents briefly the whole constructions and main results, and \cite{GiLV09} presents briefly the elaborations on linear programming in the real case.
A summary about the active bijection and related notions can also be found in \cite{GiChapterOriented}.
\EMEvder{FPSAC03???}


Let us mention that the question of relating basis and orientation activities came from Las Vergnas  in \cite{LV84a}, following on from which, in \cite{LV84b}, a definition for a correspondence between spanning trees and orientations of graphs was proposed. 
 It was based on an algorithm, given with no proof%
   \footnote{
   \label{footnote:LV84}
 Besides the fact that no proof exist, 
 the 
 authors 
 suspect that 
  this algorithm would not yield a proper correspondence anyway if its formulation was extended beyond regular matroids. Notably, its technicalities and its non-natural behaviour with respect to duality, in contrast with the active bijection, made 
  the authors 
  abandon this algorithm.}, 
  %
 which 
 inspired the general decompositions of activities developed for the active bijection, but which does not yield the correspondence given by the active bijection (not for general activities, nor for the restriction to $(1,0)$ activities, and nor with respect to duality).
\EMEvquatorze{, and may most likely not be generalized beyond regular matroids.}%
%
%
%
Also, let us 
mention that a different notion of activities for graph orientations had been introduced even earlier by  Berman in \cite{Be77}, 
along with incorrect constructions according to \cite{LV84a}\footnote{
The construction in \cite{Be77} consisted in defining some active directed cycles/cocycles in a complex way, instead of active edges, and in enumerating those cycles/cocycles. It claimed to yield a Tutte polynomial formula which was formally similar to that of Las Vergnas \cite{LV84a} using those different activities,
and  a correspondence between orientations and spanning trees.
According to \cite[footnote page 370]{LV84a}, those constructions were not correct.
}%
.
Finally, the active bijection for oriented matroids, 
as addressed in this series of papers, has been introduced and developed in 
the
Ph.D. thesis \cite{Gi02}, where most of the results from the series
were given, at least in a preliminary form.
\bs



Let us now get into the substance.
First, without requiring any preliminary knowledge, let us give to the reader one of the shortest  definitions of the active basis
(condensed yet complete, combining Definitions \ref{def:acyc-alpha} and 
\ref{def:om-alpha-var},
among various equivalent definitions addressed in this paper).
%
For any 
oriented matroid $M$ on a linearly ordered set $E$, 
\emph{the active basis} $\alpha(M)$ of $M$ is determined~by:%
\vspace{-2mm}
\begin{itemize}
\itemsep=-0.5mm

\item 
\emph{Fully optimal basis of a bounded region.} If $M$ is 
acyclic and every positive cocircuit of $M$ contains $\min(E)$, then $\alpha(M)$ is the unique 
basis $B$ of $M$ such~that:

\vspace{-2mm}
\begin{itemize}
\itemsep=-1mm
\partopsep=0mm 
\topsep=0mm 
\parsep=0mm
\vspace{-2mm}
\item 
 for all $b\in B\setminus p$, the signs of $b$ and $\min(C^*(B;b))$ are opposite in $C^*(B;b)$;

\item 
for all $e\in E\s B$, the signs of $e$ and $\min(C(B;e))$ are opposite in $C(B;e)$.
\end{itemize}
\vspace{-1mm}

\item 
\emph{Duality.} 
$\alpha(M)=E\s \alpha(M^*).$

\item 
\emph{Decomposition.} 
$\alpha(M)=\alpha(M/ F)\ \uplus\ \alpha(M(F))$
where $F$ is the union of all positive circuits of $M$ whose smallest element is the greatest possible smallest element of a positive circuit of~$M$.
\end{itemize}

%
Though the central concept can be  shortly defined  as above, the paper deals 
with separate topics and constructions of independent interest, and unifies them in a consistent framework.
What we call \emph{the active bijection} is actually a three-level  construction, built from the mapping $M\mapsto \alpha(M)$ applied to reorientations of $M$. 
It is summarized in the diagram of Figure \ref{fig:diagram} below.
In general, we get structural and bijective interpretations of several Tutte polynomial expressions, and we get various bijections involving reorientations and bases of an oriented matroid on a linearly ordered set (which apply in particular to graphs and real hyperplane arrangements). 
See Table \ref{table:intro} below for a list. 
Now, let us present more precisely some features and relate them to the paper sections.

\def\Mref{{\M_{\hbox{\small ref}}}}
\def\Mref{{\M}}

\def\hdistance{10.5cm}
\def\vdistance{6.5cm}

\def\hsdistance{3.1cm}
\def\hsldistance{4.2cm}
\def\hsmdistance{4cm}
\def\vsdistance{1.5cm}

\def\diagrdistance{2.6cm}
\def\diagldistance{2.5cm}

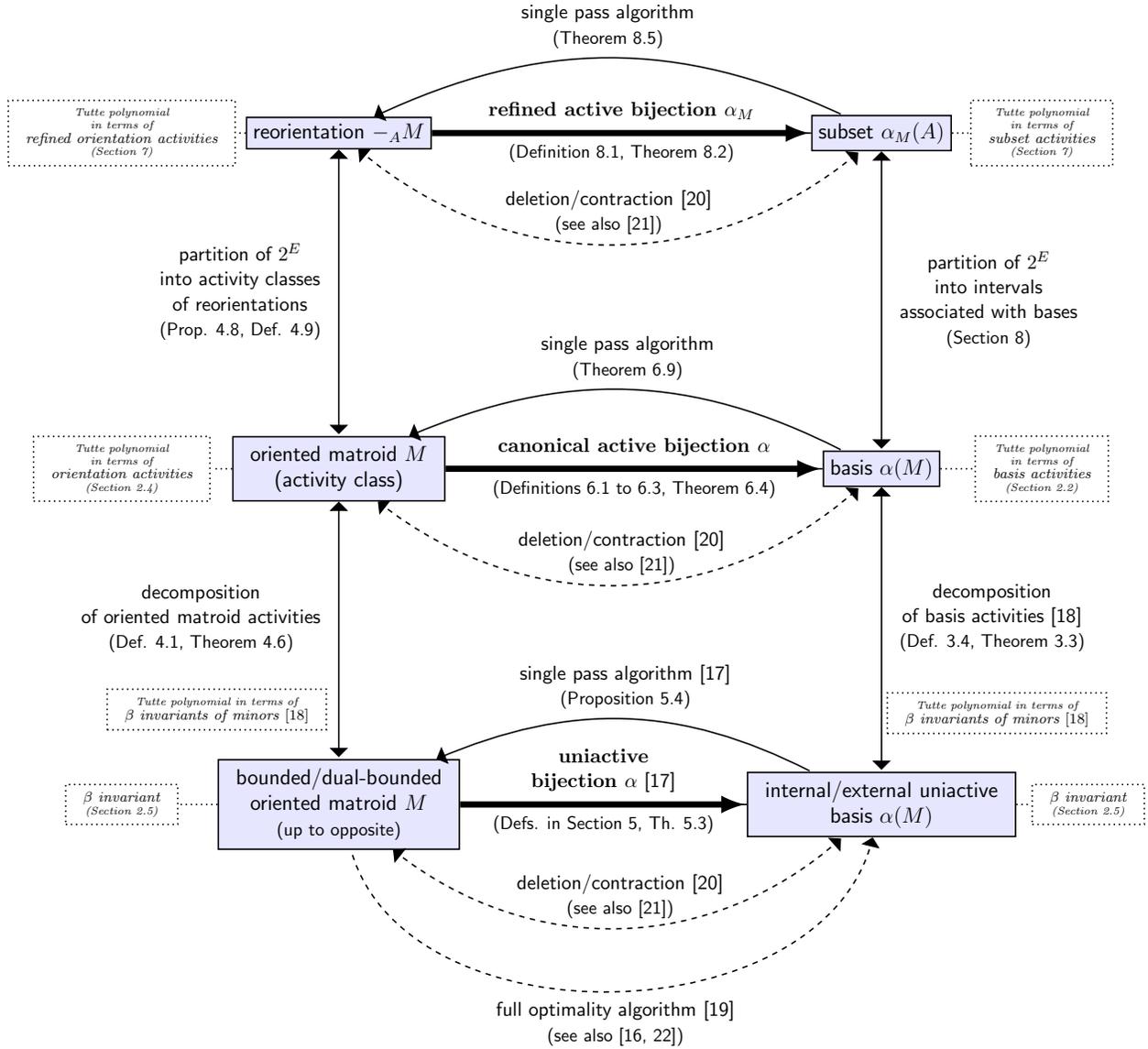
\begin{figure}[]
\centering

\scalebox{0.75}{
\begin{tikzpicture}[->,>=triangle 90, thick,shorten >=1pt,auto, node distance=\hdistance,  thick,  
main node/.style={rectangle,fill=blue!10,draw,font=\sffamily\large},
   TP node/.style={rectangle,dotted,draw,font=\sffamily\it\tiny}]

  \node[main node] (1a) {reorientation $-_A\Mref$};
  \node[main node] (1b) [right of=1a] {subset $\alpha_\Mref(A)$};
   \node[main node] (2a) [below of=1a, node distance=\vdistance] {\begin{tabular}{c}oriented matroid $M$\\ (activity class)\end{tabular}};
  \node[main node] (2b) [right of=2a] {basis $\alpha(M)$};
  \node[main node] (3a) [below of=2a, node distance=\vdistance] {\begin{tabular}{c} bounded/dual-bounded \\ oriented matroid $M$ \\ {\small(up to opposite)}\end{tabular}};
  \node[main node] (3b) [right of=3a] {\begin{tabular}{c}internal/external uniactive \\ basis $\alpha(M)$\end{tabular}};
  
  \node[TP node] (TP1a) [left of=1a, node distance=\hsldistance] {\begin{tabular}{c} Tutte polynomial\\ in terms of \\ \scriptsize refined orientation activities\\ \tiny (Section \ref{sec:partitions})\end{tabular}};
  \node[TP node] (TP1b) [right of=1b, node distance=\hsdistance] {\begin{tabular}{c} Tutte polynomial\\ in terms of \\ \scriptsize subset activities \\ \tiny (Section \ref{sec:partitions})\end{tabular}};
  \node[TP node] (TP2a) [left of=2a, node distance=\hsldistance] {\begin{tabular}{c} Tutte polynomial\\ in terms of \\ \scriptsize orientation activities\\ \tiny (Section \ref{subsec:prelim-orient-activities})\end{tabular}};
  \node[TP node] (TP2b) [right of=2b, node distance=\hsdistance] {\begin{tabular}{c} Tutte polynomial\\ in terms of \\ \scriptsize basis activities\\ \tiny (Section \ref{subsec:prelim-basis-activities})\end{tabular}};
  \node[TP node] (TP3a) [left of=3a, node distance=\hsldistance] {\begin{tabular}{c} \scriptsize $\beta$ invariant\\ \tiny (Section \ref{subsec:prelim-beta})\end{tabular}};
  \node[TP node] (TP3b) [right of=3b, node distance=\hsmdistance] {\begin{tabular}{c} \scriptsize $\beta$ invariant\\ \tiny (Section \ref{subsec:prelim-beta})\end{tabular}};
  \node[TP node] (TP4a) [above right of=TP3a, node distance=\diagrdistance]  {\begin{tabular}{c} Tutte polynomial in terms of \\ \scriptsize  $\beta$ invariants of minors \rm \cite{AB2-a}\end{tabular}};
  \node[TP node] (TP4b) [above left of=TP3b, node distance=\diagldistance] {\begin{tabular}{c} Tutte polynomial in terms of \\ \scriptsize  $\beta$ invariants of minors \rm \cite{AB2-a}\end{tabular}};

\path[every node/.style={font=\sffamily}]
	
	(1a) edge [->, >=latex, line width=1.2mm] node [bend right] 
		{\begin{tabular}{c}
			\bf refined active bijection $\alpha_\Mref$
		\end{tabular}} 
	(1b)
	
	(1b) edge [-, >=latex] node [bend right] 
		{\begin{tabular}{c}
			\small (Definition \ref{def:act-bij-ext}, Theorem \ref{th:ext-act-bij})
		\end{tabular}} 
	(1a)
		
	(2a) edge [->, >=latex, line width=1.2mm] node [bend right] 
		{\begin{tabular}{c}
			\bf canonical active bijection $\alpha$
		\end{tabular}}  
	(2b)
	
	(2b) edge [-, >=latex] node [bend right] 
		{\begin{tabular}{c}
			\small (Definitions \ref{def:om-alpha} to \ref{def:alpha-seq-decomp}, Theorem \ref{th:alpha})
		\end{tabular}}  
	(2a)
		
	(3a) edge [->, >=latex, line width=1.2mm] node [bend right] 
		{\begin{tabular}{c}
		{\bf uniactive	 }\\
		{\bf bijection $\alpha$} \cite{AB1}
		\end{tabular}}   
	(3b)

	(3b) edge [-, >=latex] node [bend left] 
		{\begin{tabular}{c}
			\small (Defs. in Section \ref{sec:bij-10}, Th. \ref{thm:bij-10})
		\end{tabular}}   
	(3a)

	(1a) edge [<-, bend left=25] node [bend left] 
		{\begin{tabular}{c}
			single pass algorithm\\
			\small (Theorem \ref{th:basori-refined})
		\end{tabular}}    
	(1b)
	
	(2a) edge [<-, bend left=25] node [bend left] 
		{\begin{tabular}{c}
			single pass algorithm\\
			\small (Theorem \ref{th:basori})
		\end{tabular}} 
	(2b)	
	
	(3a) edge [<-, bend left=25] node [bend left] 
		{\begin{tabular}{c}
			single pass algorithm \cite{AB1}\\
			\small (Proposition \ref{prop:alpha-10-inverse}) 
		\end{tabular}}  
	(3b)

	(1a)	 edge [<->,bend right=40, dashed]  node [bend right] 
		{\begin{tabular}{c} 
			deletion/contraction \cite{AB4}\\ 
			\small (see also \cite{ABG2})
		\end{tabular} } 
	(1b)
	
	(2a)	 edge [<->,bend right=40,dashed]  node [bend right] 
		{\begin{tabular}{c} 
			deletion/contraction \cite{AB4}\\ 
			\small (see also \cite{ABG2})
		\end{tabular} } 
	(2b)
	
	(3a)	 edge [<->,bend right=40, dashed]  node [bend right] 
		{\begin{tabular}{c} 
			deletion/contraction \cite{AB4}\\ 
			\small (see also \cite{ABG2})
		\end{tabular} } 
	(3b)

	
	(3b)	 edge [<-,bend left=75, dashed]  node [bend right] 
		{\begin{tabular}{c} 
			full optimality algorithm \cite{AB3}\\ 
			\small (see also \cite{GiLV09, ABG2LP})
		\end{tabular} }
	 (3a)
	 
	(1a) edge [<->] node [left] 	
		{\begin{tabular}{c} 
			partition of $2^E$ \\ 
			into activity classes\\
			of reorientations\\
			\small (Prop. \ref{prop:act-classes}, Def. \ref{def:act-class})
		\end{tabular}} 
	(2a)

	(1b) edge [<->] node [right] 	
		{\begin{tabular}{c} 
			partition of $2^E$ \\ 
			into  intervals\\
 			associated with bases\\
			\small (Section \ref{sec:refined})
		\end{tabular}} 
	(2b)
	
	(2a) edge [<->] node [left] 	
		{\begin{tabular}{c} 
			decomposition  \\ 
			of oriented matroid activities\\
			\small (Def. \ref{EG:def:ori-act-part}, Theorem \ref{th:dec-ori})\\ 
			\null
		\end{tabular}} 
	(3a)

	(2b) edge [<->] node [right] 	
		{\begin{tabular}{c} 
			decomposition \\ 
			of basis activities \cite{AB2-a}\\
			\small (Def. \ref{def:act-seq-dec-base}, Theorem \ref{th:dec_base})\\
			\null			
		\end{tabular}} 
	(3b)

	
	(1a) edge [-, dotted]  
	(TP1a)
	
	(1b) edge [-, dotted]  
	(TP1b)
	
	(2a) edge [-, dotted]  
	(TP2a)

	(2b) edge [-, dotted]  
	(TP2b)		
	
	(3a) edge [-, dotted]  
	(TP3a)

	(3b) edge [-, dotted]  
	(TP3b)
			
	;		
\end{tikzpicture}
}
\label{fig:diagram}
\caption{Diagram of results and constructions for the active bijection.  
Horizontal arrows indicate in which ways the constructions or definitions apply. 
Vertical arrows indicate how objects are related. 
Dotted rectangles indicate how the Tutte polynomial is involved or transforms through the constructions. Dashed arrows concern results detailed in forthcoming papers.
}
\end{figure}

\begin{table}
\centering

{\ptirm
\renewcommand{\arraystretch}{1}
\noindent
\begin{tabular}{|l|l|l |}

\hline
\multicolumn{1}{|c|}{\rm REORIENTATIONS} & \multicolumn{1}{|c|}{\rm BASES/SUBSETS} & \multicolumn{1}{|c|}{} \\
\hline
\multicolumn{3}{|c|}{{\it \small canonical active bijection for ordered oriented matroids}}\\
\hline
 activity classes of reorientations  \ 		&  bases 			&  $t(M;1,1)$ \\ 

 activity classes of acyclic reorientations	& internal bases		& $t(M;1,0)$ \\ 
 activity classes of totally cyclic reorientations   & external  bases			& $t(M;0,1)$ \\ 
 	 			 bounded reorientations {\scriptsize (up to opposite)}	&   uniactive internal bases	  		&  
 	 			$\beta(M)\hphantom{^*}=t_{1,0}$   \\ 
 	 			 dual bounded reorientations {\scriptsize (up to opposite)}	&   uniactive external bases	  		&  $\beta(M^*)=t_{0,1}$   \\ 



 \hline
 \multicolumn{3}{|c|}{{\it \small refined active bijection w.r.t. a given reference reorientation}}\\
 \hline
 \eme{dessous en commentaire act classes <-> sp. tree intervals}%
  				 reorientations	& subsets of the ground set							& $t(M;2,2)$ \\ 
%
%

 
 \begin{tabular}{@{}l@{}}
 reorientations with fixed orientation  \\
 \hskip 1cm  for active elements
\end{tabular}
 & 	independents	&  $t(M;2,1)$  \\ 

 \begin{tabular}{@{}l@{}}
 reorientations with fixed orientation \\
\hskip 1cm  for dual-active elements
\end{tabular} 
& spanning subsets			&  $t(M;1,2)$\\ 

 				 acyclic reorientations	& no-broken-circuit subsets			& $t(M;2,0)$ \\ 
 				 totally cyclic reorientations & supersets of external bases & $t(M;0,2)$\\


 \begin{tabular}{@{}l@{}}
 reorientations with fixed orientation  \\
\hskip 1cm for active and dual-active elements
\end{tabular}  
&  bases 			&   $t(M;1,1)$\\ 

 \begin{tabular}{@{}l@{}}
acyclic  reorientations with fixed orientation  \\
\hskip 1cm for dual-active elements
\end{tabular} 
& internal bases			&  $t(M;1,0)$ \\ 

 \begin{tabular}{@{}l@{}}
 totally cyclic  reorientations with fixed orientation  \\
\hskip 1cm for active elements
\end{tabular} 
  				 & external  bases			& $t(M;0,1)$  \\

 \hline
 \multicolumn{3}{|c|}{{\it \small translation for the case of real hyperplane arrangements}}\\
 \hline
 
 reorientations $\sim$ signatures & bases $\sim$ simplices 		  		&   \\ 
 acyclic reorientations  $\sim$ regions		& 					&  \\ 
 totally cyclic reorientations  $\sim$ dual regions		& 					&  \\ 
  bounded reorientations  $\sim$ bounded regions		& 					&  \\ 
 \hline
 \multicolumn{3}{|c|}{{\it \small translation for the case of (connected) graphs}}\\
 \hline
 
reorientations  $\sim$ orientations & bases  $\sim$ spanning trees			  		& \cite{GiLV05, ABG2, ABG2LP}  \\ 
totally cyclic  $\sim$ strongly connected 	& independents $\sim$ forests & \\ 
  	 	 bounded  $\sim$ (acyclic) bipolar 	& spanning subsets & \\ 
 	  	 dual-bounded  $\sim$ cyclic-bipolar 	& \hspace{0.5cm} $\sim$ connected spanning subgraphs & \\

  \hline
 \multicolumn{3}{|c|}{{\it \small even more particular cases}}\\
 \hline
\begin{tabular}{@{}l@{}}
{\small\it (in uniform case / general position arrangements)}
\\
 bounded regions
\end{tabular}
& 
\begin{tabular}{@{}l@{}}
pseudo/real linear programming 
\\
  \hspace{2.9cm}  optimal vertices 
\end{tabular}
&	\cite{GiLV04}	 \\


\begin{tabular}{@{}l@{}}
  {\small \it (in graphs, for suitable orderings)}	\\
   unique sink acyclic orientations 
\end{tabular}
  & internal  bases			&  \cite{GiLV05}  \\

\begin{tabular}{@{}l@{}}
{\it\small (in the braid arrangement, or the complete graph)}\\
permutations 
\end{tabular}
 & increasing trees							&\cite{GiLV06}\\ 


\begin{tabular}{@{}l@{}}
{\it \small (in the hyperoctahedral arrangement)}\\
signed permutations
\end{tabular}
 & signed increasing trees					& \cite{GiLV06}\\

\hline
\end{tabular}


}
\caption{
The two first blocks of lines list the canonical and refined active bijections  along with their notable restrictions
(Theorems \ref{th:alpha} and \ref{th:ext-act-bij}).  
The third column gives the number of involved objects.
As defined in Section~\ref{sec:prelim}: internal, resp. external, bases are those with external, resp. internal, activity equal to zero;  uniactive bases are those with only one externally or internally active element; active, resp. dual-active, elements are smallest elements of a positive circuit, resp. cocircuit.
As detailed in Definition \ref{def:act-class}: activity classes are obtained by reorienting arbitrarily unions of positive circuits or cocircuits  with the same fixed smallest element.
In the next blocks of lines, the character $\sim$ stands for a translation. 
The last column gives references where these particular cases are specifically studied. 
}
\label{table:intro}
\end{table}


\EME{dessous apres separation en deux tables}%

\ss

We introducce some notions of \emph{filtrations} for ordered (oriented) matroids, which are particular sequences of nested subsets (intuitively: the subsets involved in the decomposition of the above recursive definition). 
In a matroid, this notion yields (independently of the active bijection itself) an expression of the Tutte polynomial 
in terms of beta invariants of minors induced by these sequences, and a canonical decomposition of bases into uniactive internal/external bases of minors, both detailed in \cite{AB2-a} (No. 2.a) and recalled in Section \ref{sec:dec-seq}. On the other hand, the same sequences yield
a canonical decomposition of oriented matroids 
 into bounded regions of minors of the primal/dual (acyclic/cyclic bipolar directed graph minors in the graph case), as detailed in Section \ref{sec:dec-mo}. We also naturally obtain a partition of the set of reorientations into some \emph{activity (equivalence) classes}.


Independently, the active bijection provides a canonical \emph{uniactive bijection} between bounded regions and uniactive internal bases (their  \emph{fully optimal bases}, as in the above definition). This is a deep and difficult combinatorial result from 
\cite{AB1} (No. 1), that can be seen from different manners, notably as an elaboration of linear programming optimality, as detailed in \cite{AB3} (No. 3). See more details in the introduction of Section \ref{sec:bij-10}. This section surveys, recalls and reformulates these results in an appropriate way for this paper (see also Figure \ref{fig:bounded-duality-diagram} for a diagram on involved duality properties).
%

Putting together the two previous decompositions and the previous bijection, 
we obtain in Section~\ref{sec:can-act-bij}
the  \emph{canonical active bijection} of an ordered oriented matroid (a recursive definiton of which is given above, it can be also thought of as a sign criterion that fundamental circuits and cocircuits of one and only one basis satisfy in a given oriented matroid).
It yields a canonical activity preserving (and active filtration preserving) bijection between reorientation activity classes and bases of any ordered oriented matroid, giving a bijective passage between the two Tutte polynomial expressions mentioned at the beginning of the introduction.
An important feature is that the canonical active bijection depends only on the reorientation class of the oriented matroid, that is on the non-signed pseudosphere/hyperplane arrangement in terms of a topological representation (or the underlying undirected graph in the graph case). 

Furthermore, each one of these two aforementioned Tutte polynomial expressions can be refined into an expansion involving 
\emph{four subset activity variables}, one from bases again, the other from reorientations again, independently of each other.
The result in terms of bases is known  from the literature, and the result in terms of reorientation can be deduced from the above partition into acivity classes. These results are recalled and synthesized in Section~\ref{sec:partitions}.
In our setting, the underlying structural  construction is 
to partition  the power set of the ground set, one into classical boolean intervals associated with bases, the other into reorientation activity classes, respectively. 





Building on what precedes, in Section \ref{sec:refined}, the \emph{refined active bijection} is defined  with respect to any given reference reorientation (or signature of the pseudosphere/hyperplane arrangement, or orientation in the graph case).
It maps reorientation activity classes onto intervals associated with bases, consistently with the canonical active bijection.
 In each  class/interval couple, the reference reorientation is used to naturally fix a boolean lattice isomorphism. By this manner, the global bijection involves all reorientations/subsets, preserves the four refined activity parameters, 
 and allows us to derive various bijections.
 
 The paper ends with developed examples in Section \ref{sec:example}, completing the running example and the illustrations given along the paper.
We suggest that the reader could already  have a glimpse at these final examples,
in particular Figure \ref{fig:ex-arrgt} and its caption are intended to give a first geometrical intuition of various aspects of the construction on a simple but meaningful example.
Let us also mention, as a remarkable example, that the active bijection 
can be seen as a far reaching generalization of the well-known bijection between permutations and increasing trees (a particular case obtained from  complete graphs 
or from the Coxeter arrangement $A_n$ as detailed in~\cite[Section~5]{GiLV06}).

 \ms
 
Finally, let us end this overview by pointing out further features of the active bijection (the uniactive, the canonical, and the refined, as well).
First, from the computational complexity viewpoint,
in general, from bases to reorientations, it can be simply built by a single pass  algorithm over the ground set. From reorientations to bases, the construction is more complicated, noticeably as its restriction to bounded regions of a pseudosphere/real hyperplane arrangement contains the pseudo/real linear programming problem (see \cite{AB1, AB3, ABG2LP}).
Thus, it can be thought of, in general, as a sort of ``one way function''.
Second, the active bijection can also be built by deletion/contraction of the greatest element. Notably, in the bounded case, this construction can be seen as a refinement of the linear programming solving by variable/constraint deletion. More generally, one can describe a deletion/contraction framework for activity preserving correspondences among which the active bijection is uniquely and canonically determined. These constructions are detailed in \cite{AB4} (No. 4) (see also \cite{ABG2} for a condensed version in graphs). Third, the constructions used at each of the three levels of the active bijection are independent of each other to a certain extent. One can thus use the decomposition of activities addressed in the paper to define a decomposition framework for activity preserving correspondences among which the active bijection is uniquely and canonically determined (see Remarks \ref{rk:preserv-act-bij-class} and \ref{rk:refined-bij}, see also \cite{AB4, ABG2}).
\EMEvder{ce similarly A METTRE OU PAS ? a chagner, a detailler en sous section comme dans ABG2?}%
\EMEvder{auter brouillon la dessus, amettre peut eter ailleurs : the first part provides a bijection for $1,0$, th second part extends it, but those two constructions are independent, one could use other bijections for $1,0$ and extend them the same way. But here we obtain something canoncial at each level of the cosntuction (notably intrinsic to the given orietned matrod, indeednet of a reference orietnation).}
%
Fourth, at every level of its construction, the active bijection 
behaves nicely with respect to duality, 
and involves important duality properties,
as witness various remarks and results in this paper and others of the series.%
\EMEvder{lister resultats sur dualite ?}



\EMEvder{choix de donner recursion enterems de circuits car plus courte}%

\def\titretrois{Filtrations of an ordered matroid and decomposition of matroid bases}
\def\titrequatre{Decomposition of an oriented matroid into bounded primal/dual regions of minors (and definition of reorientation activity classes)}
\def\titrecinq{The active bijection between bounded/dual-bounded reorientations and fully optimal uniactive internal/external bases}
\def\titresix{The canonical active mapping of ordered oriented matroids, and the canonical active bijection between reorientation activity classes and matroid bases}
\def\titresept{Partitions of $2^E$ into basis intervals and reorientation activity classes, and Tutte polynomial in terms of four activity parameters for subsets and for reorientations}
\def\titrehuit{The refined active bijection between reorientations and subsets}
\def\titreneuf{Examples}

\section{Preliminaries}
\label{sec:prelim}

\EMEvder{ordered matroid et remplacer le plus possible ?}


\EME{dire pour graphes et arrangements, et dire bipolar OK A ETE VITE FAIT en une phrase au debut de section}


\subsection{Generalities}

\EME{section deja reecrites adaptees ent erms de m.o.}

\EME{AU FINAL : comparer cette section avec AB2-a}

\eme{ordered partout ? raccorucirait !}%

In the paper, $\subset$ denotes the strict inclusion, and $\uplus$, or $+$, denotes the disjoint union.
If $\mathcal F$ is a set of subsets of $E$, then $\cup \mathcal F$ denotes the subset of $E$ obtained by taking the union of all elements~of~$\mathcal F$.

In the paper $M$ denotes a matroid or an oriented matroid on a finite set $E$. 
See \cite{Ox92} and \cite{OM99}  for a complete background on matroid theory and oriented matroid theory, respectively. See in particular \cite[Sections 1.1 and 1.2]{OM99} for the relation with graphs and hyperplane arrangements 
(see also Table \ref{table:intro} of the present paper for some translations in these two settings).
See \cite[Section 2]{AB1} for a summary of oriented matroid combinatorial and geometrical aspects that we specifically use, and see \cite[Section 2]{ABG2} for preliminaries similar to those below specifically using graph terminology.

A matroid or an oriented matroid $M$ on $E$ is called \emph{ordered} when the set $E$ is linearly ordered. Then, the dual $M^*$ of $M$ is ordered by the same ordering on $E$. Any minor of $M$ is ordered the natural way, its ground set ordering being induced by that of $E$.
A minor $M/\{e\}$, resp. $M\bk\{e\}$, for $e\in E$, can be denoted for short $M/e$, resp. $M\bk e$.
A matroid might be called \emph{loop}, or \emph{isthmus}, if it has a unique element and this unique element is a loop ($M=U_{1,0}$), or an isthmus ($M=U_{1,1}$), respectively.
An isthmus is also called a coloop in the literature.
\EMEvder{changer partout isthmus en coloop ?}

Let us recall some classical properties of minors in matroids.
For $F\subseteq E$, it is known that:
circuits of $M(F)$ are circuits of $M$ contained in $F$; cocircuits of $M(F)$ are non-empty inclusion-minimal intersections of $F$ and cocircuits of $M$; circuits of $M/F$ are non-empty inclusion-minimal intersections of $E\s F$ and circuits of $M$ (that is inclusion-minimal subsets obtained by removing $F$ from circuits of $M$); cocircuits of $M/F$ are cocircuits of $M$ contained in $E\s F$.

Let us also recall some usual matroid notions.
A \emph{flat} $F$ of $M$ is a subset of $E$ such that $E\setminus F$ is a union of cocircuits; equivalently: if $C\setminus \{e\}\subseteq F$ for some circuit $C$ and element $e$, then $e\in F$; and equivalently: $M/F$ has no loop.
We call \emph{dual-flat} $F$ of $M$ a subset of $E$ which is a union of circuits; equivalently: its complement is a flat of the dual matroid $M^*$;  equivalently: if $D\setminus \{e\}\subseteq E\s F$ for some cocircuit $D$ and element $e$, then $e\in E\s F$; and equivalently: $M(F)$ has no isthmus.
A \emph{cyclic-flat} $F$ of $M$ is both a flat and a dual-flat of $M$;
 equivalently: $F$ is a flat and $M(F)$ has no isthmus; or equivalently: $M/F$ has no loop and $M(F)$ has no isthmus.
 \ss


\EME{****???????  prendre cyclic-bounded ou dual-bounded ????****}

As far as oriented matroids are concerned, 
given an oriented matroid $M$, the underlying matroid is denoted $\underline M$ when a distinction is important, but it may be denoted also by the same manner $M$.
Similarly, we will often make the abuse of using the same notation $C$ either for the \emph{signed element subset} $(C^+,C^-)$ (oriented matroid circuit) or for its \emph{support} $\underline C=C^+\uplus C^-$ (matroid circuit).
Also, we will use some typical oriented matroid technique, notably orthogonality and compositions of circuits and cocircuits, see \cite[Section 2]{AB1} or \cite{OM99}.

\EME{PEUT ETRE MIEUX DE DEFINIR REORIENTATIONS $-_AM$ EN DSITINGUANT DEUX OPPOSEES... plutot que de s'emboruiller avc des sous-ensembles....}

\EMEvder{def de om obtained by reorientation enlevee ci-dessous}
%
The set of all $-_AM$ for $A\subseteq E$ is called the set of \emph{reorientations of $M$}.
It is very important to point out that we consider this set as isomorphic to $2^E$ (as a reorientation of $M$, $-_AM$ is identified by the subset $A$). By this way, we distinguish for instance between $-_AM$ and $-_{E\setminus A}M$ as reorientations of $M$, even if 
the two resulting 
oriented matroids are equal. 
%
%
This is consistent with signed (real central) hyperplane arrangements and with the topological representation of oriented matroids as signed pseudosphere arrangements: the $2^E$ signatures of the underlying non-signed arrangement are obtained by reorienting any subset of $E$ from a given signature.
And this is consistent with graphs: given a directed graph $G=(V,E)$, the $2^E$ orientations of the underlying undirected graph are obtained by reorienting any subset of edges from $G$.
\EMEvder{attention a usage de acyclique etc, pas pour subset mais pour reorientation... attention car c'etait annonce differement dans texte mis en commentare ci-dessous}%
 %
%
 Given a reorientation  $-_AM$ of $M$, we call  $-_{E\s A}M$ its \emph{opposite} reorientation.
%
%
We say that $e\in E$ has a \emph{fixed orientation (with respect to $M$)} in a set of reorientations of $M$ if, for every reorientation  $-_AM$ in this set, we have $e\not\in A$.
%

%
%


\subsection{Matroid basis activities}
\label{subsec:prelim-basis-activities}

\EMEvder{AU FINAL comparer avec AB2a et recupererer modifs de AB2a}%
Let $M$ be an ordered matroid on $E$, and let $B$ be  a basis of $M$.
For $b\in B$, the \emph{fundamental cocircuit} of $b$ with respect to $B$, denoted $C_M^*(B;b)$, or $C^*(B;b)$ for short,  is the  unique cocircuit contained in $(E\s B)\cup\{b\}$.
For $e\not\in B$, the \emph{fundamental circuit} of $e$ with respect to $B$, denoted $C_M(B;e)$, or $C(B;e)$ for short,
 is the unique circuit contained in $B\cup\{e\}$.
When the matroid $M$ is oriented, then, by convention, 
 $b$ is positive in the fundamental cocircuit $C^*(B;b)$, and $e$ is positive in the fundamental circuit $C(B;e)$.

Let $$\Int(B)=\Bigl\{\ b\in B \mid b=\ \min \ \bigl(\ C^*(B;b)\ \bigr)\ \ \Bigr\},$$
$$\Ext(B)=\Bigl\{\ e\in E\setminus B \mid e=\ \min \ \bigl(\ C(B;e)\ \bigr)\ \ \Bigr\}.$$
We might add a subscript as $\Int_M(B)$ or $\Ext_M(B)$ when necessary.
The elements of $\Int(B)$, resp. $\Ext(B)$, are called \emph{internally active}, resp. \emph{externally active}, with respect to $B$. The cardinality of $\Int(B)$, resp. $\Ext(B)$ is called \emph{internal activity}, resp. \emph{external activity}, of $B$. 
We might write that a basis is \emph{$(i,j)$-active} when its internal and external activities equal $i$ and $j$, respectively.

Observe that $\Int(B)\cap \Ext(B)=\emptyset$ and that, for $p=\min (E)$,  we have $p\in \Int(B)\cup \Ext(B)$. If $\Int(B)=\emptyset$, resp.
$\Ext(B)=\emptyset$, then $B$ is called \emph{external}, resp. \emph {internal}.
If 
$\Int(B)\cup \Ext(B)=\{p\}$
then $B$ is called \emph{uniactive}.
Hence, a base with internal activity $1$ and external activity $0$, or $(1,0)$-active for short, is called uniactive internal, and a base with internal activity $0$ and external activity $1$, or $(0,1)$-active for short, is called uniactive external.
Let us mention that internal uniactive bases can be characterized by several ways, see \cite{GiLV05, AB1, AB2-a}.
\EMEvder{attention meme phrase et meme commentaires persos plus loin dans section 10}%
\EMEvder{redonner caracteristions de internal uniactive (ne pas priviliegeier que les roeitnations) donner en ref dans AB2a ?}%
\EMEvder{autre caracterisation de internal uniactive: GiVL05 Prop 2, bien a remettre quelque aprt}%
Let us mention that exchanging the two smallest elements of $E$ yields a canonical bijection between uniactive internal and uniactive external bases, see \cite[Proposition 5.1]{AB1} up to a typing error%
\footnote{\label{footnote:typo}Let us correct here an unfortunate typing error in \cite[Proposition 5.1 and Theorem 5.3]{AB1}. The statement has been given under the wrong hypothesis 
$B_{\min}=\{p<p'<\dots\}$ 
instead of the correct one $E=\{p<p'<\dots\}$. Proofs are unchanged
(independent typo: in line 10 of the proof of \cite[Proposition 5.1]{AB1}, instead of $B'-f$, read $(E\setminus  B')\setminus\{f\}$).
In \cite[Section 4]{GiLV05}, the statement of the same properties in graphs is correct.
}%
, see also \cite[Section 4]{GiLV05} in graphs.
\eme{remplacer partout ?}%
Also, let $B_{\min}$ be the smallest (lexicographic) base of $M$. Then, as well-known and easy to prove, we have $\Int(B_{\min})=B_{\min}$, $\Ext(B_{\min})=\emptyset$ and $\Int(B)\subseteq B_{\min}$ for every base $B$.
Also, let $B_{\max}$ be the greatest (lexicographic) base of $M$. Then $\Int(B_{\max})=\emptyset$, $\Ext(B_{\max})=E\s B_{\max}$, and $\Ext(B)\subseteq E\s B_{\max}$ for every base $B$.
Thus, roughly, internal/external activities can be thought of as situating a basis with respect to $B_{\min}$ and $B_{\max}$.
Finally, we recall that internal and external activities are dual notions:
$$\Int_M(B)=\Ext_{M^*}(E\s B) \ \ \text{ and }\ \  \Ext_M(B)=\Int_{M^*}(E\s B).$$

By \cite{Tu54, Cr69}, the Tutte polynomial of $M$ is
\begin{equation*}
\tag{``enumeration of basis activities''}
\label{eq:basis-activities}
t(M;x,y)=\sum_{\io,\ep}b_{\io,\ep}x^\io y^\ep
\end{equation*}
where $b_{\io,\ep}$ is the number of bases of $M$ 
with internal activity $\io$ and external activity $\ep$.


\begin{figure}[htbh]
\centering
\includegraphics[scale=1.2]{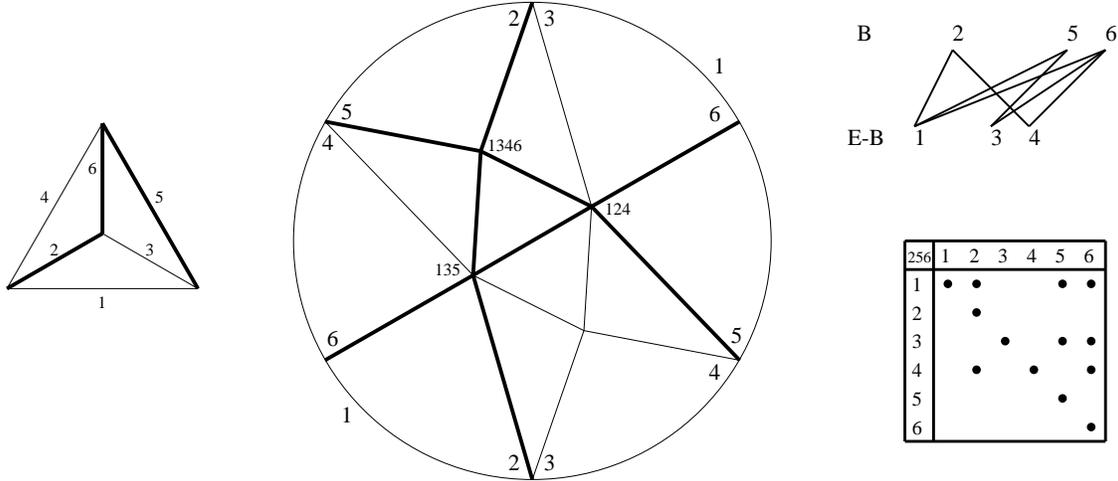}
\caption{This example of $K_4$ with elements $1<\dots<6$ will serve as a running example throughout the paper.  On the left: a graph representation. In the middle: a hyperplane arrangement representation (we always represent $\min(E)$ as a hyperplane at infinity, and we only represent one half of the arrangement, on a given side of $\min(E)$, see \cite[Section 2]{AB1} for more details on such representations). For the basis $256$, we have $\Int(256)=\emptyset$ and $\Ext(256)=\{1,3\}$.
 Fundamental cocircuits of the basis are written in the arrangement next to the vertices of the coresponding simplex. On the upper right and the bottom right, respectively: the fundamental bipartite graph and the fundamental tableau of the basis (see Section \ref{subsec:prelim-fund}).
\EMEvder{verif avec AB2a, voir ausi modif  beta dans AB2a, etc.........}
}
\label{fig:K4exbase256}
\end{figure}


\subsection{Fundamental graph/tableau of a basis}
\label{subsec:prelim-fund}

Observe that the above definitions for a basis $B$ of an ordered matroid $M$ only rely upon the fundamental circuits/cocircuits of the basis, not on the whole structure $M$. In fact all algorithms from bases to reorientations developed in the paper only rely on this local data. 
In particular, in \cite{AB2-a}, we give combinatorial constructions that also only depend on this local data, and thus that can be naturally expressed in terms of general bipartite graphs on a linearly ordered set of vertices. 
In this paper, we do not need this setting, but, for the sake of illustrations, we  introduce the following definitions.

Given a basis $B$ of a matroid $M$ on $E$, the \emph{fundamental graph} of $B$ in $M$, denoted $\F_M(B)$ is the usual graph 
with set of vertices $E$, bipartite w.r.t. the couple of subsets $(B,E\s B)$ forming a bipartition $E=B\uplus E\s B$, and with edges such that for every $b\in B$, $b$ is adjacent to elements of $C^*(B;b)\s \{b\}$, and  for every $e\in E\s B$, $e$ is adjacent to elements of $C(B;e)\s \{e\}$.
Recall that $$e\in C^*(B;b)\text{ if and only if }b\in C(B;e).$$

We call \emph{fundamental tableau} $\F_M(B)$ the matrix whose rows and columns are indexed by $E$, with entries in $\{\X,0\}$, and such that each diagonal element indexed by $(e,e)$, $e\in E$, is non-zero and is the only non-zero entry of its row if $e\in B$, and the only non-zero entry of its column if $e\in E\s B$.
We use the same notation for the fundamental graph and fundamental tableau since,
obviously, they are equivalent structures: each non-diagonal entry of the tableau represents an edge of the corresponding bipartite graph.
We choose to define both because graphs are the underlying compact combinatorial structure, whereas tableaux are better for visualization, notably for signs of the fundamental circuits/cocircuits in the oriented matroid case, and tableaux are consistent with the matrix representation used in the linear programming setting of the active bijection developed in \cite{AB1,AB3}. In what follows, we illustrate examples on both representations.

By the convention stated above, in an oriented matroid $M$, 
given a base $B$, for an element $b\in B$, 
 $b$ is positive in the fundamental cocircuit $C^*(B;b)$, and for an element $e\not\in B$,  $e$ is positive in the fundamental circuit $C(B;e)$.
%
Then, when we represent the tableau of a basis of an oriented matroid, we give signs to the entries in order to represent the fundamental cocircuits as columns and the opposites of fundamental circuits as rows (consistently with the above convention and with circuit/cocircuit orthogonality). 
\EMEvquatorze{A MTRE A JOUR, a deplacer ? OUI DEPLADCER METTRE AU NIVEAU DE LA FIGURE CONCERNEE}

An example of a matroid basis, its activities, its fundamental graph and its fundamental tableau is given here in Figure \ref{fig:K4exbase256}.
An example of a signed fundamental tableau is given later in Figure \ref{fig:ex-fob}.




\subsection{Oriented matroid activities}
\label{subsec:prelim-orient-activities}

Let $M$ be an oriented matroid on $E$.
A \emph{positive circuit}, resp. \emph{positive cocircuit}, of $\M$ is a circuit, resp. cocircuit, of $M$ such that all signs of its elements are positive.
The oriented matroid $\M$ is \emph{acyclic} if it has no positive circuit, or, equivalently, if every element belongs to a positive cocircuit.
The oriented matroid $\M$ is \emph{totally cyclic}, if every element belongs to a positive circuit, or, equivalently, if it has no positive cocircuit.

Let $\M$ be an ordered oriented matroid on a linearly ordered set $E$.
Let $$O^*(\M)=\Bigl\{\ a\in E \mid a=\ \min \ \bigl(\ D\ \bigr)\ \hbox{for a positive cocircuit } D\ \Bigr\},$$
$$O(\M)=\Bigl\{\ a\in E \mid a=\ \min \ \bigl(\ C\ \bigr)\ \hbox{for a positive circuit } C\ \Bigr\}.$$
The elements of $O^*(\M)$, resp. $O(\M)$, are called \emph{dual-active}, resp. \emph{active}, with respect to $\M$. The cardinality of $O^*(\M)$, resp. $O(\M)$, is called \emph{dual-activity}, resp. \emph{activity}, of $\M$.
We might write that an ordered oriented matroid is \emph{$(i,j)$-active} when its dual-activity and its activity equal $i$ and $j$, respectively.
Observe that $O^*(\M)\cap O(\M)=\emptyset$ and that, for $p=\min (E)$,  we have $p\in O^*(\M)\cup O(\M)$. Observe also that we have $O^*(\M)=\emptyset$, resp.
$O(\M)=\emptyset$, if and only if  $\M$ is totally cyclic, resp. acyclic.
Finally, observe that those two activities are dual notions:
$$O(M^*)=O^*(M) \ \ \text{ and }\ \  O^*(M^*)=O(M).$$
An illustration is shown in Figure \ref{fig:K4act}, along with a geometrical interpretation.
\EMEvquatorze{Detailler l'itnerpretation geomtrique sous forme d'observation/easy lemme ? }

\begin{figure}[h]
\centering
\includegraphics[scale=0.8]{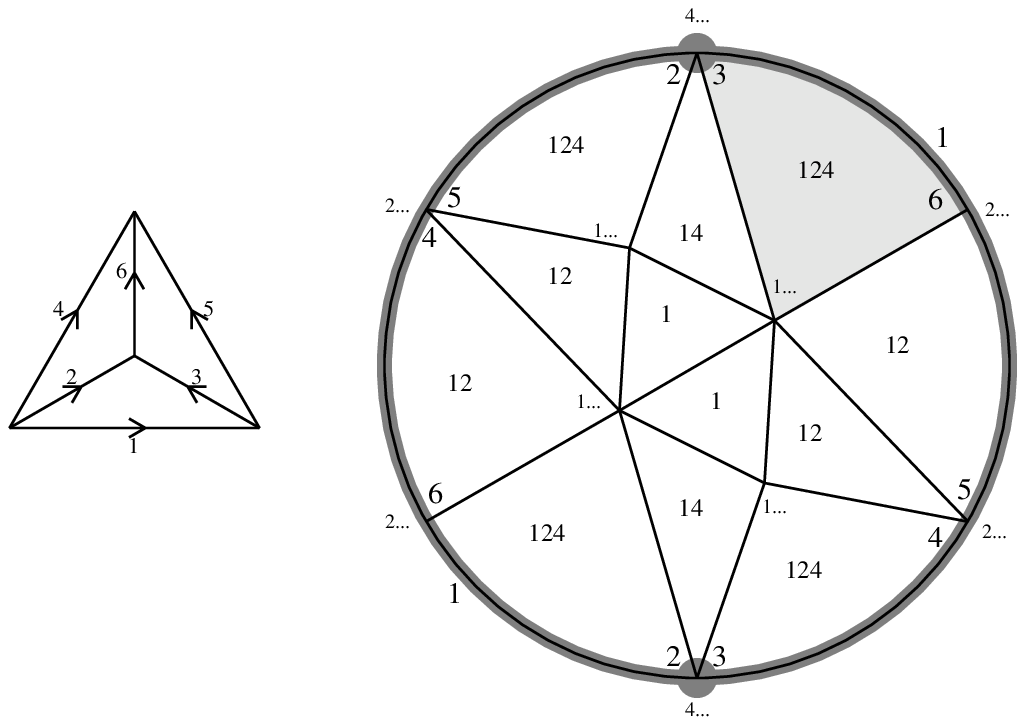}
\hfil
\includegraphics[scale=0.5]{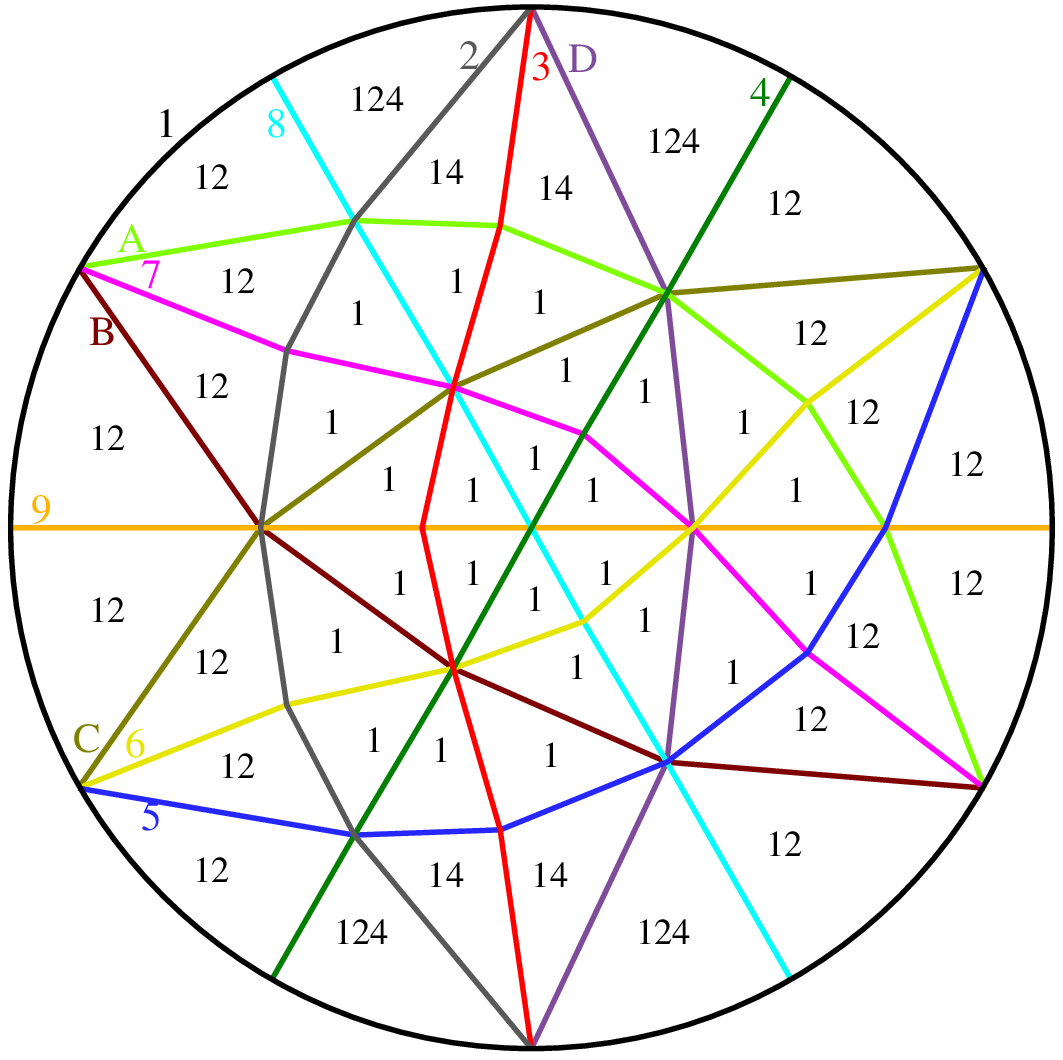}
\caption[]{Let us detail the left part: dual-activities for acyclic reorientations (regions) of $K_4$, continuing the running example from Figure \ref{fig:K4exbase256}.
The smallest element of each cocircuit is written at each corresponding vertex of the arrangement (as $1...$, $2...$, or~$4...$). The dual-activity of each region is written in the region (it is given by the elements written at the vertices of its border).
We get that
$t(K_4;x,0)= 8.({x\over 2})^3 + 12.({x\over 2})^2 + 4.({x\over 2})$. In particular
$t(K_4;2,0)=24$ counts the regions of $K_4$.
Observe, in general, that dual-activities indicate the positions of the regions w.r.t. the sequence of nested faces (here $1\cap 2   \subset 1$, depicted in bold) induced by the minimal basis (here $124$). 
Observe that the dual-activity of a region depends only on the minimal basis and on the unsigned underlying arrangement.
\EMEvder{insister sur le fait que region grise et graphe oriente inutiles sauf pour etablr lien avec graphe ? detailler cette region ?}%
The grey region corresponds to the directed graph shown on the left. Its dual-active elements are $\{1,2,4\}$, given by the positive cocircuits $124$, $2345$ and $456$  (directed cocycles of the directed graph).
\EMEvder{seconde figure ajoutee apres soumission, avant arxiv --- reprendre dimensions}%
The right part shows dual active elements of regions of
 an arrangement on $13$ elements, with minimal basis $1<2<4$ (the rest of the ordering is not used).
This second figure is exhaustively completed at the very end of the paper in Figure \ref{fig:D13refined}.
}%

\label{fig:K4act}
\end{figure}

By  \cite{LV84a}, we have the following theorem enumerating  reorientation activities: 
\begin{equation*}
\tag{``enumeration of reorientation activities''}
\label{eq:reorientation-activities}
t(M;x,y)=\sum_{\io,\ep}o_{\io,\ep}\ \Bigl({x\over 2}\Bigr)^\io \ \Bigl({y\over 2}\Bigr)^\ep
\end{equation*}
where $o_{\io,\ep}$ is the number of reorientations of $M$ with  dual-activity $\io$ and activity $\ep$.
This last formula generalizes various classical results from the literature, such as counting regions or acyclic (re)orientations \cite{Wi66, St73, Za75,LV77}
(see \cite{GiChapterOriented} for a survey on this result and further references).
%
%
%
\EMEvder{ai enleve dessous en commentaire les applications de ce resultat}%
%
%
%

Comparing the above two expressions for $t(M;x,y)$ we get, for all $\io,\ep$:
$$o_{\io,\ep}=2^{\io+\ep}b_{\io,\ep}.$$

\subsection{Beta invariant}
\label{subsec:prelim-beta}


\EMEvquatorze{reordonner 1,0 tout apres ?}

In particular, by the above formula, we have that  $$b_{1,0}={o_{1,0}\over 2}$$
counts the number of uniactive internal bases, and counts half the number of reorientations with orientation activity $1$ and dual orientation activity $0$, or $(1,0)$-active reorientations for short. 
This number  does not depend on the linear ordering of the element set $E$. 
This value $$\beta(M)=b_{1,0}$$ is known as the \emph{beta invariant} of $M$, introduced by Crapo \cite{Cr67}. Assuming $\mid E\mid >1$, it is known that $\beta(M)\not=0$ if and only if $M$ is connected. Let us recall that, 
for a loopless graph $G$ with at least three vertices, the associated matroid $M(G)$ is connected if and only if $G$ is 2-connected.
Also, we have $\beta(M)=b_{1,0}=b_{0,1}=\beta(M^*)$ as soon as   $\mid E\mid >1$.
Note that, if $\mid E\mid=1$, we have $\beta(M)=1$ if the single element is an isthmus of $M$,
and $\beta(M)=0$ if the single element is a loop of $M$.


Finally, for our constructions, we need to define the following dual slight variation $\beta^*$ by:
$$
 \beta^*(M)=\beta(M^*)=b_{0,1}={o_{0,1}\over 2}= \
\Biggr\{
\begin{array}{ll}
       \beta(M) &\text{ if }|E|>1 \\
       0 &\text{ if $M$ is an isthmus} \\
       1 &\text{ if $M$ is a loop.} 
\end{array}
$$

\subsection{Bounded reorientations (or bounded regions) in oriented matroids}
\label{subsec:prelim-bounded}

Let us characterize $(1,0)$-active reorientations of an ordered oriented matroid.
We say that an oriented matroid $\M$  on  $E$ is {\it bounded with respect to $p\in E$} if  $\M$ is acyclic and every positive cocircuit contains $p$.
In particular, if $\M$ consists in a single element $p$ which is an isthmus, then $\M$ is bounded with respect to $p$.
In terms of a topological representation or in terms of an affine real hyperplane arrangement, $M$ is bounded w.r.t. $p$ if and only if it corresponds to a region of the arrangement that does not touch $p$. Since $p=\min(E)$ is considered as an element ``at infinity'', such a region is therefore a ``bounded'' region in the usual sense, see details in \cite[Section 2]{AB1}.
In terms of a directed graph $G=(V,E)$, the associated oriented matroid $M(G)$ is bounded w.r.t. $p$ if and only if $G$ is \emph{bipolar} w.r.t. $p$, meaning that it is acyclic with a unique source and a unique sink which are the extremities of $p$, see details in \cite{GiLV05} or \cite[Section 2]{ABG2}.
Bounded regions with respect to a given element are counted by twice the $\beta$-invariant, as initially shown in \cite{Za75, LV77}.

We say that $\M$ is {\it dual-bounded with respect to $p\in E$} if $M$ is totally cyclic and every positive circuit contains $p$.
In terms of a directed graph $G=(V,E)$, dual-bounded is called \emph{cyclic-bipolar} in \cite[Section 2]{ABG2}.
In particular, if $\M$ consists in a single element $p$ which is a loop, then $M$ is dual-bounded with respect to $p$.
Equivalently, for an oriented matroid $M$ with at least two elements, $M$ is dual-bounded w.r.t. $p$ if and only if $-_pM$ is bounded w.r.t. $p$.

Therefore, for matroids with at least two elements, reorienting $p$ provides a canonical bijection between bounded reorientations with respect to $p$ and dual-bounded reorientations with respect to $p$, see \cite[Proposition 5.2]{AB1} (or also \cite[Proposition 5]{GiLV05} in graphs).  
Observe also that $M$ is bounded w.r.t. $p$ if and only if $M^*$ is dual-bounded w.r.t. $p$. Therefore bounded reorientations of $\M$ with respect to $p$ correspond to dual-bounded reorientations of $M^*$ with respect to~$p$.

Assuming $M$ is ordered, we get by definitions that:  
$\M$ is bounded with respect to $p=\min(E)$ if and only if 
$O(\M)=\emptyset$ (i.e. $\M$ is acyclic, i.e. $\M$ has an activity equal to zero) and $O^*(\M)=\{p\}$ (i.e. it has exactly one dual-active element, i.e. $\M$ has a dual-activity equal to one).
Similarly, $\M$ is dual-bounded if and only if 
$O^*(\M)=\emptyset$ (i.e. $\M$ is totally cyclic, i.e. $\M$ has a dual-activity equal to zero) and $O(\M)=\{p\}$  (i.e. it has exactly one active element, i.e. $\M$ has an activity equal to one).

\section{Filtrations of an ordered matroid and decomposition of matroid bases}

\label{sec:dec-seq}




\EMEvder{verifier avec AB2-a si nouvelles modifs}%


In this section, we recall definitions and the main result from the companion paper \cite{AB2-a}, No. 2.a of the main series.
Briefly, filtrations of an ordered matroid are particular sequences of nested sets (equivalent to particular partitions of the ground set).
A given basis can be decomposed by its active filtration (or active partition) into a uniquely defined sequence of bases of minors, such that these bases are $(1,0)$ or $(0,1)$-active
(in the sense of Tutte polynomial internal/external activities). 
And all this yields a decomposition theorem for the set of all bases using all possible filtrations (Theorem \ref{th:dec_base}). 
 Let us recall that those bases with internal/external activities equal to $1/0$ are counted by the beta invariant, as called after Crapo \cite{Cr67}, which is equal to the coefficient of $x$ of the Tutte polynomial $t(M;x,y)$ of the matroid.
 Let us mention that the bijection provided by this decomposition theorem
yields, numerically, an expression of the Tutte polynomial of a matroid in terms of beta invariants of minors \cite[Theorem \ref{a-th:tutte}]{AB2-a}, 
that refines at the same time the classical expressions in terms of basis activities and orientation activities (if the matroid is oriented), and the convolution formula for the Tutte polynomial. See \cite{AB2-a} for details and references.

\begin{definition}[{\cite[Definition \ref{a-def:general-filtration}]{AB2-a}}]
\label{def:general-filtration}
Let $E$ be a linearly ordered finite set.
Let $M$ be a matroid on $E$.
  We call \emph{filtration of $M$} (or $E$) 
  a 
sequence $(F'_\ep, \ldots, F'_0, F_c , F_0, \ldots, F_\io)$ of subsets of $E$
 such~that:%
\begin{itemize}
\itemsep=0mm
\partopsep=0mm 
\topsep=0mm 
\parsep=0mm
\item $\emptyset= F'_\ep\subset...\subset F'_0=F_c=F_0\subset...\subset F_\io= E$;
\item the sequence $\min(F_k\setminus F_{k-1})$, $1\leq k\leq\io$  is increasing with $k$;
\item the sequence  $\min(F'_{k-1}\setminus F'_k)$, $1\leq k\leq\ep$, is increasing with $k$.
\end{itemize}
The sequence is a \emph{connected filtration of $M$} if, in addition:
\emevder{OU : A filtration of $M$ is called \emph{connected (w.r.t. $M$)} if, in addition:}%
\begin{itemize}
\vspace{-1mm}
\itemsep=0mm
\partopsep=0mm 
\topsep=0mm 
\parsep=0mm
\item for $1\leq k\leq\io$, the minor $M(F_{k})/F_{k-1}$ is connected and is not a loop;
\item for $1\leq k\leq\ep$, the minor $M(F'_{k-1})/F'_k$ is connected and is not an isthmus.
\end{itemize}
\end{definition}

Equivalently, 
a filtration  of $M$ is connected if and only if
$$\Bigl(\prod_{1\leq k\leq \io}
\beta \bigl( M(F_k)/F_{k-1}\bigr)\Bigr)
 \ \Bigl(\prod_{1\leq k\leq \ep}\beta^* \bigl( M(F'_{k-1})/F'_{k}\bigr)\Bigr)\ \not=\ 0.$$%

In what follows, we can equally use the notations $(F'_\ep, \ldots, F'_0, F_c , F_0, \ldots, F_\io)$
or $\emptyset= F'_\ep\subset...\subset F'_0=F_c=F_0\subset...\subset F_\io= E$ to denote a filtration of $M$.
The $\io+\ep$ minors involved in Definition \ref{def:general-filtration} are said to be \emph{associated with} or \emph{induced by} the filtration.
%

Observe that filtrations of $M$ are equivalent to pairs of partitions of $M$ formed by a bipartition obtained from the subset $F_c$ (with possibly one empty part, which is a slight language abuse) and a refinement of this bipartition:
$$E=F_c\uplus E\s F_c,$$
$$E= (F'_{\ep-1}\s F'_\ep)\ \uplus\ \dots \ \uplus\ (F'_{0}\s F'_1)\ \uplus\ (F_1\s F_0)\ \uplus\ \dots \ \uplus\ (F_{\io}\s F_{\io-1}).$$
Indeed, one can retrieve the sequence of nested subsets from the pair of partitions since  the subsets in the sequence are unions of parts given by the ordering of the smallest elements of the parts.

Let us note that, by \cite[Lemma \ref{a-lem:des-seq-flats}]{AB2-a}, for a connected filtration of $M$: for $0\leq k\leq\io$, the subset $F_k$ is a flat of $M$; for $0\leq k\leq\ep$, the subset $F'_k$ is a dual-flat of $M$; and the subset $F_c$ is a cyclic-flat of $M$, which we call \emph{cyclic flat of the (connected) filtration}.
\EMEvder{ceci est-il utile?}%
%
%

\begin{observation}[{\cite[Observation \ref{a-lem:dec-seq-observation}]{AB2-a}}]
\label{obs:dec-seq-duality}
Let $\emptyset= F'_\ep\subset...\subset F'_0=F_c=F_0\subset...\subset F_\io= E$ be a connected filtration of an ordered matroid $M$. We have:
\begin{itemize}
\itemsep=0mm
\item $\emptyset= E\s F_\io\subset...\subset E\s F_0=E\s F_c=E\s F'_0\subset...\subset E\s F'_\ep= E$ is a connected filtration of $M^*$, for the cyclic-flat $E\s F_c$ of $M^*$;
\item the minors associated with the above filtration of $M^*$ are the duals of the minors associated with the above filtration of $M$, indeed:
for every $1\leq k\leq\io$, 

\centerline{$\bigl( M(F_{k})/F_{k-1}\bigr)^*=M^*(E\s F_{k-1})/(E\s F_k),$}
and for every $1\leq k\leq\ep$, 

\centerline{$\bigl( M(F'_{k-1})/F'_k\bigr)^*=M^*(E\s F'_{k})/(E\s F'_{k-1}).$}
\end{itemize}
\end{observation}
\ss


\begin{thm}[{\cite[Theorem \ref{a-th:dec_base}]{AB2-a}}] 
\label{th:dec_base}
Let $M$ be a matroid on a linearly ordered set $E$.
%
$$
\Bigl\{\ \text{bases of }M\ \Bigr\}\ 
=\biguplus_
{\substack{
\emptyset=F'_\ep\subset...\subset F'_0=F_c\\
F_c=F_0\subset...\subset F_\iota=E\\
\hbox{connected filtration of $M$}
}}
\Bigl\{B'_1\plus...\plus B'_\ep\plus B_1\plus...\plus B_\iota
\mid$$
$$\text{for all }1\leq k\leq \ep, \
B'_k \hbox{ base of }M(F'_{k-1})/F'_{k}\text{ with $\io(B'_k)=0$ and $\ep(B'_k)=1$,}$$
$$\text{for all }1\leq k\leq \io, \
B_k \hbox{ base of }M(F_k)/F_{k-1}\text{ with $\io(B_k)=1$ and $\ep(B_k)=0$}\Bigr\}$$
With the above notations and $B=B'_1\plus...\plus B'_\ep\plus B_1\plus...\plus B_\iota$,
we then have:
$$\Int(B)=\cup_{1\leq k\leq \iota} \min(F_k\s F_{k-1})=\cup_{1\leq k\leq \iota} \Int(B_k),$$
$$\Ext(B)=\cup_{1\leq k\leq \ep} \min(F'_{k-1}\s F'_{k})=\cup_{1\leq k\leq \ep} \Ext(B'_k).$$
%
\end{thm}
\ss

\begin{definition}
\label{def:act-seq-dec-base}
Let $M$ be a matroid on a linearly ordered set $E$.
Let $B$ be a basis of $M$. 
The \emph{active filtration of $B$} is the unique (connected) filtration of $M$ associated to $B$ by Theorem \ref{th:dec_base}.
\ss

For what follows in the paper, we only need the above definition. 
%
Two other equivalent definitions are available, they are constructive and based on the fundamental graph of the basis only. The active filtration of $B$ can be defined:
\vspace{-2mm}
\begin{itemize}
\partopsep=0mm \topsep=0mm \parsep=0mm \itemsep=0mm
\item by applying successively the active closure to the internal/external active elements, see \cite[Definition \ref{a-def:act-seq-dec-fund-graph}]{AB2-a};
\item by a single pass algorithm over $E$, see \cite[Proposition \ref{a-prop:basori-partact}]{AB2-a}, which is used (and formally contained) in next Theorem \ref{th:basori}.\end{itemize}

The \emph{active partition of $B$} is the partition of $E$ induced by successive differences of subsets in the active filtration of $B$, given with the cyclic flat $F_c$ of the filtration
(so that the active filtration of $B$ is determined, as observed above).
The \emph{active minors of $B$} are the minors induced by the active filtration of $B$.
\end{definition}

Let us eventually recall the following observation, which is a direct consequence of  Theorem \ref{th:dec_base}, and which will yield later a remarkable constructive property.

\begin{observation}[{\cite[Observation \ref{a-lem:dec-seq-bas-observation-suite}]{AB2-a}}]
\label{obs:induced-dec-seq-bas}
Let $\emptyset= F'_\ep\subset...\subset F'_0=F_c=F_0\subset...\subset F_\io= E$ be the active filtration of a basis $B$ of $M$.
Let $F$ and $G$ be two subsets in this sequence such that $F\subseteq G$. Then, 
$B\cap G\s F$ is a basis  of $M(G)/F$, and its active filtration is obtained from the subsequence with extremities $F$ and $G$ (i.e. $F\subset \dots \subset G$) of the active filtration of $B$ by subtracting  $F$ from each subset of the subsequence (with $F_c\s F$ as cyclic flat).
In particular, the subsequence ending with $F$ is the active filtration of $B\cap F$ in $M(F)$, and the subsequence beginning with $F$ yields the active filtration of $B\s F$ in $M/F$ by subtracting  $F$ from each subset.
\end{observation}

\section{The active filtration/partition of an ordered oriented matroid, decomposition into bounded primal/dual regions of minors, and activity classes of reorientations}

\label{sec:dec-mo}

\EMEvder{attention deans preuve des lemmes devenus props qu'ojn emploie pa le terme lemme}

\EMEvder{sous-sections comme dans ABG2??? active partition - activity class - decomp,}%

\EMEvder{preuves plus smples avec compositions et trucs purement m.o. ?}%


%
The decomposition addressed in this section 
refines the classical decomposition of an oriented matroid $\M$ into an acyclic oriented matroid $\M/F$ and a totally cyclic oriented matroid $M(F)$, where $F$ is the union of positive circuits of $M$ and $E\s F$ is the union of positive cocircuits of $M$.
We use the same filtrations as in Section \ref{sec:dec-seq} but in an oriented matroid setting. 
\EMEvder{donner references ?}%
This decomposition technique, whose idea dates back in a certain extent to \cite{LV84b}, was developed in \cite{Gi02} along with the related results addressed here, and it was shortly presented in \cite{GiLV04, GiLV05, GiLV06, GiLV07}.

We show how to canonically decompose an oriented matroid on a linearly ordered set into a sequence of minors, that are  induced by a connected filtration called \emph{the active filtration} (Theorem \ref{thm:unique-dec-seq}). These minors are either bounded or dual-bounded with respect to their smallest element (see Section \ref{subsec:prelim-bounded}).
The partition of the ground set induced by these minors is called \emph{the active partition}.
%
These minors have either $(1,0)$ or $(0,1)$ orientation activities.

This decomposition yields a decomposition theorem for the set of all reorientations of an oriented matroid (Theorem \ref{th:dec-ori}).
%
%
Numerically, this decomposition theorem can be seen as the same  expression of the Tutte polynomial in terms of products of beta invariants as the one from \cite[Theorem \ref{a-th:tutte}]{AB2-a} alluded to in Section \ref{sec:dec-seq}, but this time restricted to orientable matroids.

Moreover,
we define a partition of the set of reorientations into activity classes,
obtained by reorienting independently all the parts of the active filtration/partition, that is, all the ground sets of the above minors (which preserves the active filtration/partition). 
%
Geometrically, active partitions and activity classes of reorientations describe the positions of regions of a peudosphere/hyperplane arrangement with respect to the minimal flag (see Figure \ref{fig:rank4-act-part}).
We show in Section \ref{sec:partitions} how this partition 
yields a simple expression of the Tutte polynomial using four reorientation activity parameters (Theorem~\ref{th:expansion-reorientations}).

A simple example is provided by Figure \ref{fig:K4-dec}, continuing the running example of $K_4$ (note also that all active partitions of orientations on $K_4$ are shown in  Section \ref{subsec:ex-K4}).
A more involved example is provided by Figure \ref{fig:gros-ex-decomp}, showing on a graph the parts associated with positive circuits and the parts associated with positive cocircuits.
Another more involved example is provided by Figure \ref{fig:rank4-act-part},
showing the geometrical interpretation of active partitions on regions of a rank-4 hyperplane arrangement.

\ms

\EME{ajouter desin de fibre supersolv + ptet dessin de K4 petit, d'apres celui de prelim ?}


\begin{definition}
\label{EG:def:ori-act-part}
Let $M$ be an ordered oriented matroid on $E$
with $\io$ dual-orientation-active elements $a_1<\ldots<a_\io$ and $\ep$ orientation-active elements $a'_1<\ldots<a'_\ep$.
The \emph{active filtration of $M$}
\index{active filtration/partition of an oriented matroid (or digraph)|textbf}%
 is the sequence of subsets 
\ss

\centerline{$\displaystyle \emptyset= F'_\ep\ \subset\cdots\subset\ F'_0=F_c=F_0\ \subset\cdots\subset\ F_\io= E$}
\ss

\noindent
defined as follows. 
%
First, 
\ss

\centerline{$\displaystyle F_c\ =\ \bigcup_{C \hbox{\small\ positive circuit}}C
\ =\ E\ \setminus\ \bigcup_{C \hbox{\small\ positive cocircuit}}C.$}
\ss
%
%

\noindent For every $0\leq k\leq\ep-1$,%
\ss

\centerline{$\displaystyle F'_{k}=\bigcup_{\substack{ C \hbox{\small\ positive circuit}\\{{{\small  \Min}}(C)\; \geq\; a'_{k+1}}}}C.$}
\ss

\noindent 
Moreover, $F'_\ep=\emptyset$, $F_\io=E$, and, dually,  for every $0\leq k\leq\io-1$, 
\ss

\centerline{$\displaystyle F_k=E\ \setminus\ \bigcup_{\substack{ C \hbox{\small\ positive cocircuit}\\{{{\small  \Min}}(C)\; \geq\; a_{k+1}}}}C.$}
\index{positive circuit/cocircuit}%
\ss

%
\noindent
The \emph{active partition of $M$} is the partition 
induced by successive differences:
\vspace{-1mm}
$$E= (F'_{\ep-1}\s F'_\ep)\ \uplus\ \cdots \ \uplus\ (F'_{0}\s F'_1)\ \uplus\ (F_1\s F_0)\ \uplus\ \cdots \ \uplus\ (F_{\io}\s F_{\io-1}),$$
%
with $\min(F'_{k-1}\s F'_k)=a'_k$ for $1\leq k\leq \ep,$
and 
$\min(F_{k}\s F_{k-1})=a_k$ for $1\leq k\leq \io$.

\vspace{1mm}
%
%

\noindent
The \emph{active minors of $M$} 
are the $\io+\ep$ minors
\index{active minors|textbf}%
\vspace{-1.5mm}
$$M'_k=(M|F'_{k-1})/F'_{k} \hbox{ for }  1\leq k\leq \ep, \text{ and } M_k=(M|F_k)/F_{k-1} \hbox{ for } 1\leq k\leq\io.$$
\end{definition}

We assume that the active partition is always given with the set $F_c$ (i.e. it can be thought of as a pair of partitions, one for $F_c$, the other for $E\s F_c$).
%
%
By this way, knowing the  active partition, allows us to retrieve the active filtration of $\M$. Indeed, the sequence $\min(F_k\setminus F_{k-1})$, $1\leq k\leq\io$,  is increasing with $k$, and the sequence  $\min(F'_{k-1}\setminus F'_k)$, $1\leq k\leq\ep$, is increasing with $k$, so the position of each part of the active partition with respect to the active filtration is identified.
%

Moreover, we have, for $1\leq k\leq\io$, 
$$F_k\setminus F_{k-1}\ =\ 
\bigcup_{\substack{ D \hbox{\small\ positive cocircuit}\\ {{\small  \Min}}(D)\; =\; a_k}}D
\ \ \setminus\ 
\bigcup_{\substack{ D \hbox{\small\ positive cocircuit}\\ {{\small \Min}}(D)\; >\; a_k}}D,$$
and, for $1\leq k\leq \ep$,  
$$F'_{k-1}\setminus F'_{k}\ =\ 
\bigcup_{\substack{ D \hbox{\small\ positive circuit}\\ {{\small  \Min}}(D)\; =\; a'_k}}D 
\ \ \setminus\ 
\bigcup_{\substack{ D \hbox{\small\ positive circuit}\\ {{\small \Min}}(D)\; >\; a'_k}}D
.$$

\EMEvder{figure modifiee dans chapitre avec plus gros caracteres}%

\begin{figure}[]
\centering
{\scalebox{1.2}{\includegraphics[width=12cm]{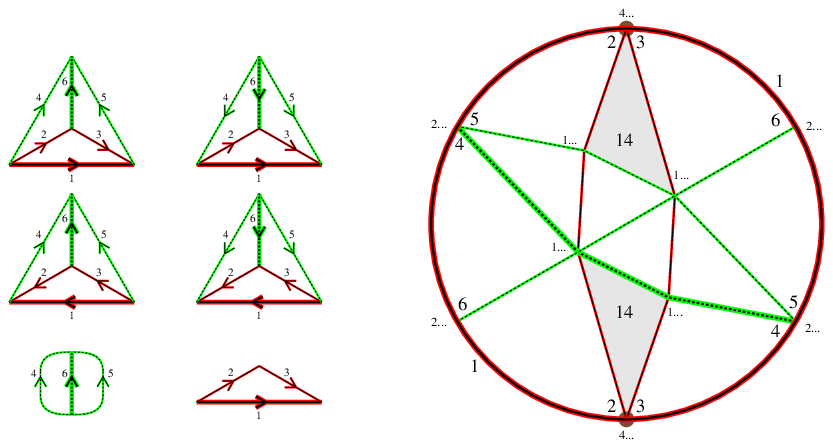}}}
\vspace{5mm}

\begin{tabular}{lccccc}
Dual-active elements: && \red{1} &and&\darkgreen{4}&\\
Active partition:& &\red{$123$}& {$\uplus$}& \darkgreen{$456$}&\\
Active filtration:& $\ \emptyset$\ &\red{$\subset$}\ &$123$\ & \darkgreen{$\subset$}\ & $123456$\\
Active  
minors  :&& \red{$M(123)$}  &and &\darkgreen{$M/123$}  &\\
Activity class:& \multicolumn{5}{c}{ $\{\  $ $M$,  $\ -_{\red{123}}M$,  $\ -_{\darkgreen{456}}M$, $\ -_{\red{123}\darkgreen{456}}M$ $\ \}$}
\end{tabular}
\caption[]{
Active decomposition of an acyclic orientation of $K_4$. 
Consider $M$ as any of the four depicted orientations of the graph $K_4$, or any of the two grey regions of the arrangement $K_4$.
The activities of $M$ are $O(M)=\emptyset$ and $O^*(M)=\{1,4\}$.
By Definition \ref{EG:def:ori-act-part}, the active filtration is $\emptyset=F_c\subset 123\subset 123456$, and the active partition is $123+456$, with cyclic flat $F_c=\emptyset$. The active  minors  are $\M(123456)/123$, which is bounded w.r.t. $4$, and $\M(123)$, which is bounded w.r.t.~$1$ (Theorem \ref{thm:unique-dec-seq}). Those minors are depicted as bipolar digraphs. The activity class is formed by the four depicted graphs, or by the two grey regions and their opposite (Definition \ref{def:act-class}).
\EME{ajouter version geometrique}%
\EMEvder{enlever face ? attention dans chaptre caracteres augmentes, mais pseudroites arrondies}%
}
\label{fig:K4-dec}
\end{figure}

%

\begin{figure}[]
\centering
\includegraphics[scale=1.2]{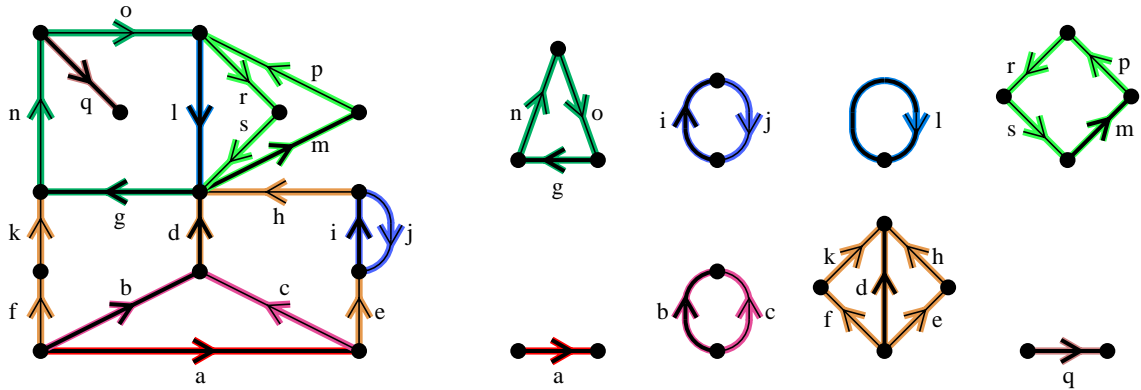}
\caption[]{Active decomposition of an ordered digraph $\G$.  The digraph $\G$ is shown on the left. 
The ordering of the edge set $E$ is given by: $a<b<c<\dots<q<r<s$. The active edges are $O(\G)=\{g,i,l,m\}$ and the dual-active edges are $O^*(\G)=\{a,b,d,q\}$ (bold edges). The active partition is given by: $F_c=gno+ij+l+mprs$ and $E\s F_c=a+bc+defhk+q$ (Definition \ref{EG:def:ori-act-part}).  
The corresponding active minors, which are bounded (or bipolar), resp. dual-bounded (or cyclic-bipolar),  w.r.t. their smallest edges (Theorem \ref{thm:unique-dec-seq}), and whose edge sets are given by the active partition, are shown in the bottom right line, resp. the upper right line. 
}
\label{fig:gros-ex-decomp}
\end{figure}

\null

\begin{figure}[h]
\centering
{\scalebox{0.95}
{\includegraphics[width=12cm]{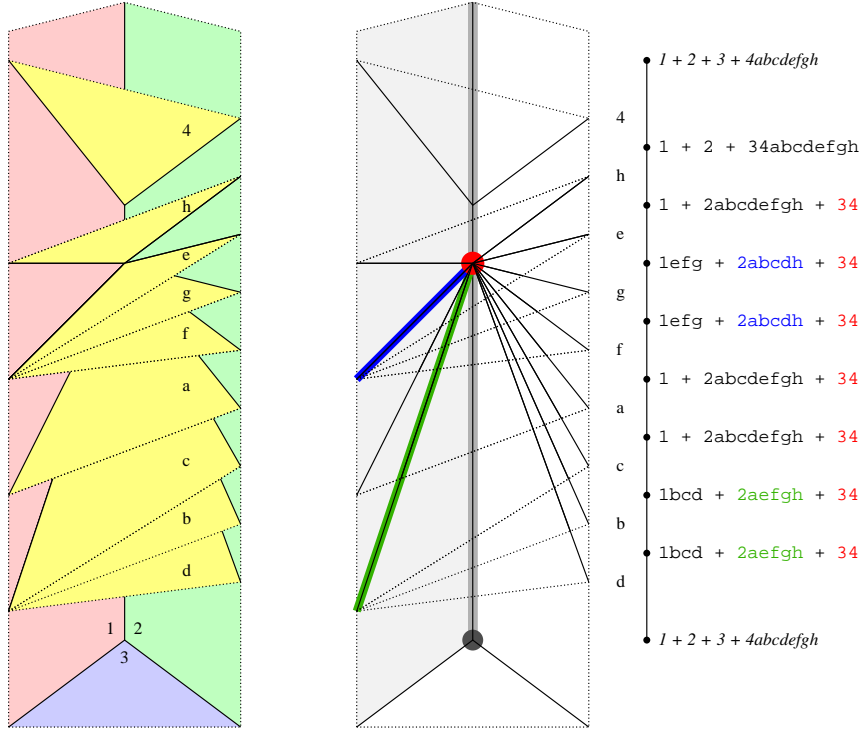}}}
\caption{Active partitions for some rank 4 regions (this picture is part of a more involved picture used in \cite{GiLV06}).
The ordering is given by $1<2<3<4<a<\dots<h$.
On the left and the middle: two views of the region of the space delimited by the hyperplanes $1$, $2$, and $3$. On the right: the active partitions of the regions of the arrangement contained (and forming a path) in this region (Definition \ref{EG:def:ori-act-part}).
The minimal basis is $1234$, it induces the flag of faces $1\cap 2\cap 3\subset 1\cap 2\subset 1$ (in shades of grey in the picture). Geometrically, dual-activities and active partitions situate regions w.r.t. the  flag of faces induced by the minimal basis. Precisely: intersections of regions with this flag of faces correspond to the covectors induced by the subsets of the active filtrations (see the covectors whose supports are $34$, $2abcdh$, and $2aefgh$, drawn in bold). 
\EMEvder{(but one may observe that only the four first elements of the ordering are relevant for the construction in this case). XXX VRAI OU PAS ???? XXX RK FIGURE RENTRE AVEC LEA PRECENDETEN SI DEUX LIGNES EN MOINS}
}
\label{fig:rank4-act-part}
\end{figure}

Remark that each $F_k$, $0\leq k\leq\io-1$, is the complement of the support of a positive covector of $M$, hence it is a flat of $M$, and that each $F'_k$, $0\leq k\leq\ep-1$, is the  support of a positive vector of $M$, hence it is a dual-flat of $M$.  In particular, $F_c$ is a cyclic-flat of $M$.
For convenience, we can refer to $F_c$, or to the parts forming $F_c$, as the \emph{cyclic part} of $M$, and to $E\s F_c$, or to the parts forming $E\s F_c$, as the \emph{acyclic part} of $M$.

Finally, let us point out that, in the definitions that precede and the results that follow, the particular case of acyclic oriented matroids is addressed as the case where $F_c=\emptyset$, and the totally cyclic case is addressed as the case where $F_c=E$.
Those cases are dual to each other. Let us deepen this with the next observation (which can be compared with 
Observation \ref{obs:dec-seq-duality}
for filtrations in general).

\begin{observation}
\label{lem:dec-seq-mo-observation}
Let $\emptyset= F'_\ep\subset...\subset F'_0=F_c=F_0\subset...\subset F_\io= E$ be the active filtration of an ordered oriented matroid $M$. We have:
\begin{enumerate}
\item $\emptyset= E\s F_\io\subset...\subset E\s F_0=E\s F_c=E\s F'_0\subset...\subset E\s F'_\ep= E$ is the active filtration of $M^*$, for the cyclic-flat $E\s F_c$ of $M^*$;
\label{item:sec-dec-act-dual}
\item $\emptyset= F'_\ep\subset...\subset F'_0=F_c=F_c$ is the active filtration of the totally cyclic oriented matroid $M(F_c)$, for the cyclic-flat $F_c$ of $M(F_c)$;
\item $\emptyset=\emptyset= F_0\s F_c\subset...\subset F_\io\s F_c=E\s F_c$ is the active filtration of the acyclic oriented matroid $M/F_c$, for the cyclic-flat $\emptyset$ of $M/F_c$;
\end{enumerate}
\end{observation}

\begin{proof}
The first observation is obvious by the definition and by properties of oriented matroid duality. The second one is direct since positive circuits of $M(F_c)$ are exactly positive circuits of $M$. The third one is dual to the second one.
\end{proof}


\EME{SEE ABG2 for a more involved graph example (ou bien le remettre ???? non a priori pas trop de repettition...)}





In the following proofs, we can arbitrarily focus either on circuits or on cocircuits. We usually focus on cocircuits, because of their natural geometrical interpretation, and we deduce the same results for circuits by duality. Some proofs are written in terms of circuits, when they imply shorter notations, and can be deduced for cocircuits by duality. 

\EMEvquatorze{NB: on doit avoir aussi que $F$ is the unqiue subset such that... avec activites... comme base mais avec ordre en plus... a mettre ? a voir sans ordre en plus ?}

\EME{The next lemma notably means that one can compute the active filtration of $M$ in successive minors as well. BOF}

\begin{lemma}
\label{lem:induction-dec-seq}
Let $\M$ be oriented matroid on a linearly ordered set $E$ with $\io\geq 0$ dual-active elements $a_1<...<a_\io$, with $\ep\geq 0$ active elements $a'_1<...<a'_\ep$, and with active filtration \break $\emptyset=F'_\ep\subset F'_{\ep-1}\subset\ldots\subset F'_0= F_c = F_0\subset \ldots\subset F_{\io-1}\subset F_\io=E$. 

\noindent If $\io>0$ then, denoting $F=F_{\io-1}$, we have:
\begin{itemize}
\item $\M /F$ is 
bounded with respect to $a_\io$,
\item the active filtration of $\M(F)$ is 
 $(F'_\ep, \ldots, F'_0, F_c , F_0, \ldots, F_{\io-1})$.
\end{itemize}
 If $\ep>0$ then, denoting $F'=F'_{\ep-1}$, we have:
 \begin{itemize}
\item $\M(F')$ is 
dual-bounded with respect to $a_\ep$,
\item the active filtration of $\M / F'$ is 
 $(F'_{\ep-1}\setminus F', \ldots, F'_0\setminus F', F_c\setminus F' , F_0\setminus F', \ldots, F_{\io}\setminus F')$.
\end{itemize}
\end{lemma}

\eme{Enlever notation $F$ et $F'$ de enonce ? garder juste pour preuve ?}%



\begin{proof}
Let us assume first that $\io>0$, denoting $F=F_{\io-1}$.
The cocircuits of $\M/F$ are the cocircuits of $\M$ contained in $E\s F$, 
 where $E\s F$ is the union of all 
 positive cocircuits $D$ of $\M$ with smallest element $a_\io$.
 Hence every element of $\M/F$ belongs to a positive cocircuit, hence $\M/F$ is acyclic.
 And $a_\io$ belongs to a positive cocircuit of $\M/F$, hence $a_\io$ is dual-active in $\M/F$.
If another element was dual-active in $\M/F$, then it would also be the smallest element of a positive cocircuit in $\M$ and dual-active in $\M$, a contradiction with $a_\io$ being the greatest dual-active element of $\M$.
So we have $O^*(\M/F)=\{a_\io\}$ and $O(\M/F)=\emptyset$, that is $\M/F$ is bounded with respect to $a_\io$. 

As $E\s F$ is a union of positive cocircuits of $\M$, the positive circuits of $\M(F)$ are the positive circuits of $\M$. Hence, $\M$ and $\M(F)$ have the same active elements, and the cyclic part 
$(F'_\ep, \ldots, F'_0, F_c)$ of their active filtration is the same. 

The cocircuits of $\M(F)$ 
are exactly the non-empty inclusion-minimal intersections of 
 intersections of $F$ and cocircuits of $\M$.
 More precisely, the signed subsets of the form $C\cap F$, where $C$ is a cocircuit of $\M$, are unions of cocircuits of $\M(F)$.
Since every element of $E\s F$ is greatest than $a_\io$ by definition of $a_\io$, we have that $a_k\in F$ for every $1\leq k<\io$. A positive cocircuit $D$ of $\M$ with smallest element $a_k$, for $1\leq k<\io$, induces a positive cocircuit contained in $D\cap F$ of $M(F)$ with smallest element $a_k$, hence $a_1,\dots, a_{\io-1}$ are dual-active in $\M(F)$. 
Let $H_k=F\s \cup\{D\mid D\hbox{ positive cocircuit of }\M(F), \ \min(D)>a_k\}$.
Independently, by definition of $F_k$, we have $F_k=F\cap F_k=F\s \cup\{F\cap D\mid D\hbox{ positive cocircuit of }\M, \ \min(D)>a_k\}$.
For every positive cocircuit $D$ of $\M$, $D\cap F$ is a union of positive cocircuits of $M(F)$, 
so we  have  $F\s F_k\subseteq F\s H_k$, that is $H_k\subseteq F_k$.

Now, conversely,  let $e$ be an element of $F\s H_k$, for some $1\leq k<\io$. It belongs to be a positive cocircuit  $D$ of $\M(F)$ with smallest element $a>a_k$.
We want to prove that $e$ belongs to $F\cap D'$ for some positive cocircuit $D'$ of $\M$ contained in $D\cup (E\s F)$.
%
The cocircuit $D$ is contained in a cocircuit $D_M$ of $\M$ with $D_M\cap F=D$. 
Let $D'_M$ be the composition of all positive cocircuits of $\M$ with smallest element $a_\io$, whose support is $E\s F$ and whose signs are all positive (since given by positive cocircuits).
Then $D'_M\circ D_M$ is positive, since it is positive on $E\s F$ as $D'_M$, and positive on $D_M\cap F=D$ as $D$.
And $D'_M\circ D_M$ has smallest element $a$, since $a<a_\io$. 
By the conformal composition property of covectors in oriented matroid theory, there exists a positive cocircuit $D'$ of $\M$ containing $e$ and contained in $D_M\cup (E\s F)$.
Since every element of $E\s F$ is greater than $a_\io$ and $a_\io\geq a >a_k$, the smallest element of $D'$ is greater than $a$, hence strictly greater than $a_k$.
Since $e$ belongs to $F\cap D'$, we get that $e\in F\s F_k$.
We have proved  $F\s H_k\subseteq F\s F_k$, that is finally $F_k=H_k$, which provides the active filtration of $\M(F)$.

The second property involving $F'$ is dual to the first one involving $F$, hence its dual proof is direct, by Observation \ref{lem:dec-seq-mo-observation} Item \ref{item:sec-dec-act-dual}.
\end{proof}

%


\begin{thm}
\label{thm:unique-dec-seq}
Let $\M$ be oriented matroid on a linearly ordered set $E$.
The active filtration of $\M$ is the unique (connected) filtration $(F'_\ep, \ldots, F'_0, F_c , F_0, \ldots, F_\io)$ of $M$ (or $E$) such that
the $\io$ minors $$\M(F_k)/F_{k-1},\hbox{ for } 1\leq k\leq\io,$$ 
are 
bounded with respect to 
$a_k=\min(F_k\setminus F_{k-1})$, and the $\ep$ minors $$\M(F'_{k-1})/F'_{k}, \hbox{ for }1\leq k\leq \ep,$$ are 
dual-bounded with respect to 
$a'_k=\min(F'_{k-1}\setminus F'_{k})$.
\end{thm}

\eme{$a_k=\min(F_k)$ ? ou seulement $a_k=\min(F_k\setminus F_{k-1}$ ?}


\begin{proof}
\EMEvquatorze{NB. passages de cette preuve pourraient certainement etre simplifie en utilisant composition, ptet essayer, voir these....}
\EMEvquatorze{ATTENTION il y a aussi Observation \ref{obs:dec-seq-mo-observation} !!! ????BOF???}
Observe that the statement of the result is ``self-dual''. Precisely, by Observation \ref{obs:dec-seq-duality}:
\vspace{-1mm}
\begin{itemize}[-] 
\itemsep=0mm
\parsep=0mm
\partopsep=0mm 
\topsep=0mm
\item a (connected) filtration $(F'_\ep, \ldots, F'_0, F_c , F_0, \ldots, F_\io)$ of $M$ corresponds to a (connected) filtration $(E\s F_\io, \ldots, E\s F_0, E\s F_c , E\s F'_0, \ldots, E\s F'_\ep)$ of $M^*$;
\item the minors $\M(F_k)/F_{k-1}$, $1\leq k\leq\io$, of $M$, are bounded w.r.t. $a_k=\min(F_k\s F_{k-1})$, if and only if the corresponding minors  $\M^*(E\s F_{k-1})/(E\s F_{k})$ of $M^*$ are dual-bounded w.r.t. $a_k=\min((E\s F_{k-1})\s (E\s F_{k}))$;
\item the minors $\M(F'_{k-1})/F'_{k}$, $1\leq k\leq\ep$, of $M$, are dual-bounded w.r.t. $a'_k=\min(F'_{k-1}\s F'_{k})$, if and only if the corresponding minors  $\M^*(E\s F'_{k})/(E\s F'_{k-1})$ of $M^*$ are bounded w.r.t. $a'_k=\min((E\s F'_{k})\s (E\s F'_{k-1}))$.
\end{itemize}
\vspace{-1mm}
Hence, in the following proof, we will be allowed to deduce various results by duality, applying the same reasonings to $M^*$.

\EME{OK FAIT - attetnion : il faut avoir verifie avant que suite emboitee, et d'ailleurs c'est utilise dans observation precedente, et c'est necessaire pour definir partition active}
First, we check that the active filtration is a filtration. 
By construction, we have 
$\emptyset= F'_\ep\subset...\subset F'_0=F_c=F_0\subset...\subset F_\io= E$.
Assume $\M$ has  $\io$ dual-active elements $a_1<...<a_\io$, and $\ep$ active elements $a'_1<...<a'_\ep$.
By definition of $a_k$, for $1\leq k\leq \io$, there exists a positive cocircuit of $\M$ whose smallest element is $a_k$, hence $a_k\in F_k\setminus F_{k-1}$ according to the definition of $F_k\setminus F_{k-1}$ given above. So we have $a_k=\min(F_k\setminus F_{k-1})$,  $1\leq k\leq\io$, which is increasing with $k$ by definition of $a_k$.
Similarly, for $1\leq k\leq \ep$, there exists a positive circuit of $\M$ whose smallest element is $a'_k$,
so we get $a'_k=\min(F'_{k-1}\setminus F'_k)$, 
which is increasing with $k$.
Hence the result.



Second, we apply recursively Lemma \ref{lem:induction-dec-seq}.
We directly get that 
the $\io$ minors $\M_k=\M(F_k)/F_{k-1}$, $1\leq k\leq\io$
are 
bounded with respect to $a_k$;
and that the $\ep$ minors $\M'_k=\M(F'_{k-1})/F'_{k}$, $1\leq k\leq \ep$, are 
dual-bounded with respect to $a'_k$. This proves the property stated in the result. This also proves that those minors are connected as soon as they have more than one element, which achieves the proof that the active filtration of $\M$ is a connected filtration of $M$.
%


Now, it remains to prove the uniqueness property.
Assume $(F'_\ep, \ldots, F'_0, F_c , F_0, \ldots, F_\io)$ is a filtration of $E$ satisfying the properties given in the result.
Then it is obviously  a connected filtration of $M$, by the definitions, since being bounded, resp. dual-bounded, implies being either connected  or reduced to an isthmus, resp. a loop.

First, we prove that $F_c$ is the union of all positive circuits of $\M$.
Assume $C$ is a positive circuit of $\M$, not contained in $F_c$.
Let $k$ be the smallest integer such that $C\subseteq F_k$, $1\leq k\leq \io$.
Then $C\s F_{k-1}\not=\emptyset$ (otherwise $k$ would not be minimal), so $C\s F_{k-1}$ contains a positive circuit of $\M/F_{k-1}$.
Moreover $C\s F_{k-1}\subseteq F_k\s F_{k-1}$ by definition of $k$, so  $C\s F_{k-1}$ contains a positive circuit of $\M_k=\M(F_k)/F_{k-1}$, a contradiction with $\M_k$ being acyclic.
Hence the union of positive circuits of $\M$ is contained in $F_c$.
With exactly the same reasoning from the dual viewpoint,
we get that the union of positive cocircuits of $\M$ is contained in $E\s F_c$.
\emenew{DESSOUS EN COMMENTAIRE TERMES DE COCIRCUITS}%
Finally, $F_c$ contains the union of positive circuits of $\M$ and has an empty intersection with the union of all positive cocircuits of $\M$, so $F_c$ is exactly the union of all positive circuits of $\M$.

\emenew{DESSOUS EN COMMENTAIRE PREUVE POUR CIRCUITS}%


Second, we prove the following claim: for every positive cocircuit $D$ of $\M$, the smallest element of $D$ equals 
$a_{k+1}$, where $k$ is the greatest possible such that $D\subseteq E\s F_k$, $0\leq k\leq \io-1$.
%
%
Indeed, for such $D$ and $k$, we have
$D\cap F_{k+1}\not=\emptyset$ (otherwise $k$ would not be maximal), so $D\cap F_{k+1}$ is a union of positive cocircuits of $\M(F_{k+1})$.
Moreover, $D\cap F_{k+1}\subseteq F_{k+1}\s F_{k}$ by definition of $k$, so $D\cap F_{k+1}$ is a union of positive cocircuits of $\M_{k+1}=\M(F_{k+1})/F_{k}$. 
By assumption that $\M_{k+1}$ is bounded with respect to $a_{k+1}$, we have that $a_{k+1}$ belongs to every positive cocircuit of  $\M_{k+1}$,
so $a_{k+1}$ 
is the smallest element of $D\cap F_{k+1}$.
By definition of a filtration, $a_{k+1}$ is the smallest element in $E\s F_k$ (it is the smallest in $F_k\s F_{k-1}$ and the sequence $\min(F_i\s F_{i-1})$ is increasing with $i$), hence we have $\min(C)=a_{k+1}$. 
In particular, we have proved that the dual-active elements of $\M$ are of type $a_k$, $1\leq k\leq \io$.

Dually, 
we get the following claim:
for every positive circuit $C$ of $\M$, the smallest element of $C$ equals 
$a'_{k+1}$, where $k$ is the greatest possible such that $C\subseteq F'_k$, $0\leq k\leq \ep-1$.
And in particular, we get that the active elements of $\M$ are of type $a'_k$, $1\leq k\leq \ep$.

Third, we prove that the parts of the considered filtration are indeed the parts of the active filtration.
\emenew{DESOUS EN COMMENTAIRE PREUVE POUR CIRCUITS}%
%
Let us denote $F=F_{\io-1}$ and so $a_\io=\min (E\s F)$.
We want to prove that $F=E\s \cup\{D\mid D\hbox{ positive cocircuit of }\M, \ \min(D)=a_\io\}$.
By assumption, $M_\io=\M/F$ is bounded.
So, every element of $\M/F$ belongs to a positive cocircuit of $\M/F$ with smallest element $a_\io$. 
The cocircuits of $\M/F$ are the cocircuits of $\M$ contained in $E\s F$. Hence, every element of $\M$ belonging to $E\s F$ belongs to a positive cocircuit of $\M$ with smallest element $a_\io$,
which proves that
$E\s F\subseteq \cup\{D\mid D\hbox{ positive cocircuit of }\M, \ \min(D)=a_\io\}$.
Conversely, let $D$ be a positive cocircuit of $\M$ with smallest element $a_\io$.
By the above claim, we have that $\io-1$ is the greatest possible such that $D\subseteq E\s F_{\io-1}$, that is $D\subseteq E\s F$, hence the result.

Dually, let us denote $F=F_{\ep-1}$ and so $a'_\ep=\min (F)$. By the same reasoning applied to $M^*$, we get that 
$F=\cup\{C\mid C\hbox{ positive circuit of }\M, \ \min(C)=a'_\ep\}$.

Now, we can conclude by induction, assuming the result is true for minors of $\M$.
Assume $\io>0$ and denote again $F=F_{\io-1}$, we have proved above that $F$ is indeed the largest part different from $E$ in the active filtration of $\M$.
It is direct to check that the sequence of subsets
$(F'_\ep, \ldots, F'_0, F_c , F_0, \ldots, F_\io-1)$  is a connected filtration of $M(F)$.
Moreover this filtration obviously satisfies the properties of the result for the oriented matroid $\M(F)$, as the involved minors are unchanged.
Hence, this filtration is the active filtration of $\M(F)$, by induction assumption.
Hence, by Lemma \ref{lem:induction-dec-seq}, we have that the subsets $F'_\ep, \ldots, F'_0, F_c , F_0, \ldots, F_\io-1$ are indeed the same subsets as in the active filtration of $\M$.
\emenew{DESSOUS EN COMMENTAIRE PREUVE POUR CIRCUITS}%
%
%
Dually, assume that $\ep>0$ and denote again $F'=F_{\ep-1}$, we get that the subsets $F'_{\ep-1}, \ldots, F'_0, F_c , F_0, \ldots, F_\io$ are indeed the same subsets as in the active filtration of $\M$.
So, finally, we have proved that the result is true for $M$.
\end{proof}


The next observation 
is the counterpart of Observation \ref{obs:induced-dec-seq-bas} for oriented matroid activities. 


\begin{observation}
\label{obs:induced-dec-seq-ori}
Let us continue and refine Observation  \ref{lem:dec-seq-mo-observation}.
Let $\emptyset= F'_\ep\subset...\subset F'_0=F_c=F_0\subset...\subset F_\io= E$ be the active filtration of the ordered oriented matroid $M$.
Let $F$ and $G$ be two subsets in this sequence such that $F\subseteq G$.
\EMEvder{faire enonce precis genre corollaire ???}%
Then, by Theorem \ref{thm:unique-dec-seq},
the active filtration  of $M(G)/F$ is obtained from the subsequence with extremities $F$ and $G$ (i.e. $F\subset \dots \subset G$) of the active filtration of $M$ by   subtracting $F$ from each subset of the subsequence (with $F_c\s F$ as cyclic flat).
In particular, the subsequence ending with $F$ (i.e. $\emptyset\subset \dots \subset F$) is the active filtration of $M(F)$, and the subsequence beginning with $F$  (i.e. $F\subset \dots \subset E$) yields the active filtration of  $M/F$ by subtracting  $F$ from each subset.
\end{observation}
\ss

\begin{thm}
\label{th:dec-ori}
Let $M$ be an oriented matroid on a linearly ordered set $E$. We have

\vspace{2.5mm}
\centerline{$\Bigl\{\hbox{reorientations $-_AM$ of $M$ for $A\subseteq E$}\Bigr\}\ $}
\vspace{-5.5mm}
\begin{align*}
 = \ \biguplus\ & \Biggl\{\ \ -_AM \ \ \mid\ \    
-_A\M(F_k)/F_{k-1},\ \ 1\leq k\leq\io, \hbox{ bounded with respect to } \min (F_k\setminus F_{k-1}),\\
&
\hbox{ and }
-_AM(F'_{k-1})/F'_{k}, \ \ 1\leq k\leq \ep, \hbox{ dual-bounded with respect to } \min (F'_{k-1}\setminus F'_{k})\ \ \Biggr\}
\end{align*}
\noindent where the disjoint union is over all connected filtrations $(F'_\ep, \ldots, F'_0, F_c , F_0, \ldots, F_\io)$ of $M$. The connected filtration of $M$ associated  to a reorientation $-_AM$ in the right-hand side of the equality is the active filtration of $-_A\M$.
\EMEvder{over all filtrations ? (when the filtration is not connected then it yields an empty subset) a prouver}%
\end{thm}

\begin{proof}
This result consists in a bijection between all reorientations $-_AM$ of $M$ and sequences of reorientations of the minors involved in decomposition sequences of $M$. It is given directly by  Theorem \ref{thm:unique-dec-seq}.
From the first set to the second set, 
the active filtration of $\M$ provides the required decomposition.
Conversely, from the second set to the first set,
first choose a connected filtration of $M$.
Then, for each minor of $M$ defined by this sequence, choose a bounded/dual-bounded reorientation for this minor as written in the second set statement.
This defines a reorientation $-_AM$ of $M$ (since every element of $M$ appears in one and only one of these minors).
Now, for this reorientation  $-_AM$, the chosen filtration satisfies the property of Theorem \ref{thm:unique-dec-seq},
hence this filtration is the active filtration of the reorientation $-_A\M$ of $M$. 
%
Finally, the uniqueness in Theorem \ref{thm:unique-dec-seq} ensures that the union in the second set is disjoint.
%
\end{proof}

As mentioned above, Theorem \ref{th:dec-ori} applies in particular to a decomposition of the set of acyclic, resp. totally cyclic, reorientations of an oriented matroid $M$, involving only bounded, resp. dual-bounded, minors (by restriction to connected filtrations with cyclic flat $F_c=\emptyset$, resp. $F_c=E$).


\begin{remark}
\label{rk:tutte}
\rm
Theorem \ref{th:dec-ori}, along with the \ref{eq:reorientation-activities}, directly yields a proof, in orientable matroids, of the Tutte polynomial expression in terms of beta invariants of minors stated in the companion paper No. 2.a \cite[Theorem \ref{a-th:tutte}]{AB2-a}. The proof given in \cite{AB2-a} is available for general matroids (it uses Theorem \ref{th:dec_base} and the \ref{eq:basis-activities} in a similar way). Notice that, in graphs, both proofs can be used, as all graphs yield orientable matroids (the proof of this result given in \cite{ABG2} for graphs uses the orientation decomposition).
\end{remark}

\begin{prop}
\label{prop:act-classes}
Let $\M$ be an oriented matroid on a linearly ordered set $E$, 
with $\io$ dual-active elements and $\ep$ active elements.
The $2^{\io+\ep}$ reorientations of $M$ obtained by reorienting any union of parts of the active partition (or filtration) of $\M$ have the same active partition (or filtration) as $\M$, and hence the same active and dual-active elements.
\end{prop}





\begin{proof}

%
\EMEvder{preuve condensee, mais a l'air ok pour cet article}%
The result can be proved directly from Definition \ref{EG:def:ori-act-part}. Let us give such a proof in a condensed way (see \cite{Gi18} for a more detailed and more general proof). Consider  the union $A$ of all positive circuits  of $\M$ whose smallest element is greater than a given element $a$ (the same reasoning holds for cocircuits by duality). We prove the following claim: any union of all positive circuits  whose smallest element is greater than a given element $e$ is the same in $\M$ and $-_A\M$. 
Let $C$ be a  positive circuit of $\M$, with $\min(C)\geq e$. If $e\geq a$, then $C\subseteq A$, then $C$ is a positive circuit of $-_A\M$.
Assume now $e<a$. The set $A$ can be considered as a positive vector of both $M$ and $-_AM$ (it is a conformal composition of positive circuits). Since $C$ is positive on $C\s A$ in $-_AM$, then $A\circ C$ is also a positive vector of $-_AM$. By the generation of vectors by conformal composition, for every element $f$ of $C$, there exists a positive circuit $D$ of $-_AM$ containing $f$ and contained in $C\cup A$. Also, we have $\min(D)\geq e$. So $C$ is contained in the union of positive circuits of $-_AM$ with smallest element greater than $e$. So the claim is proved. The rest of the proof is straightforward: apply the claim to the subsets of the cyclic part of active filtration, and apply the same claim dually and independently to subsets built from
the acyclic part of active partition.

Alternatively, the result can also be seen as a direct corollary of 
Theorem \ref{thm:unique-dec-seq}. Indeed, reorienting a union of parts of the active partition of $\M$ implies reorienting completely some of the active minors of $\M$.
Then, by Theorem \ref{thm:unique-dec-seq}, the resulting reorientation is obtained from the same minors, that is, from the same filtration of $M$, which is the same active filtration as that of $\M$.
\end{proof}


\begin{definition}
\label{def:act-class}
Let $\M$ be an oriented matroid on a linearly ordered set, 
with $\io$ dual-active elements and $\ep$ active elements.
We call \emph{activity class of }$\M$ the set of $2^{\io+\ep}$ reorientations of $M$ obtained by reorienting any union of parts of the active partition of $\M$. 

By Proposition \ref{prop:act-classes}, activity classes of reorientations of $M$ can be seen as equivalence classes, which partition the set of reorientations of $M$.
\end{definition}


\EMEvder{topology and geometrical intepretation of activity classes (balls...)}%

We will continue to study activity classes in Section \ref{sec:partitions}: from their boolean lattice structure, one can derive four refined activity parameters and a Tutte polynomial expansion in these terms.
The notions of active filtration/partition and of activity classes are generalized to oriented matroid perspectives in \cite{Gi18}. 
See Figure \ref{fig:K4-dec} for an example of activity class (see also  Figure \ref{fig:K4-iso}).


\EMEvquatorze{attention def reorietnation activity classES of M pour M seul ou pour totues reorietnations ???? EDIT : pas vraiment important de le preciser rigouresuement....}

For an ordered oriented matroid $M$,
a reorientation $-_AM$ is said to be \emph{active-fixed}, resp. \emph{dual-active fixed}, with respect to $M$ if no active, resp. no dual-active, element has been reoriented with respect to $M$, that is, if $O(-_AM)\cap A=\emptyset$, resp. $O^*(-_AM)\cap A=\emptyset.$


\begin{cor}
\label{cor:enum-classes}
Let $\M$ be an ordered oriented matroid.
The number of activity classes of reorientations of $M$ with activity $i$ and dual activity $j$ equals $t_{i,j}$.

Each activity class of reorientations of $M$ contains exactly one reorientation which is
active-fixed and dual-active-fixed w.r.t. $M$.
The number of such reorientations of $M$ with activity $i$ and dual activity $j$ thus equals $t_{i,j}$.

Furthermore, each activity class of reorientations of $M$ with activity $i$ and dual activity $j$
contains $2^j$ active-fixed reorientations and $2^i$ dual-active-fixed reorientations.
Finally, we have the enumerations and representatives of activity classes given by Table~\ref{table:enum-classes}.
\end{cor}

\begin{table}[H]
\def\interligne{&\\[-11pt]}%
\vspace{-5mm}
\begin{center}
\begin{tabular}{|l|c|}
\hline
\interligne
{\bf reorientations of $M$} /
{\bf activity classes of reorientations of $M$}   
& \multicolumn{1}{|c|}{\bf number}\\
\hline
\interligne
 active-fixed and   dual-active-fixed / all
&  $t(M;1,1)$\\
 \interligne
  acyclic and dual-active-fixed / acyclic
& $t(M;1,0)$\\
 \interligne
 active-fixed and totally cyclic / totally cyclic
& $t(M;0,1)$\\
\interligne
 active-fixed (/ non-applicable) & $t(M;2,1)$\\
\interligne
  dual-active-fixed (/ non-applicable) & $t(M;1,2)$\\
\hline
\end{tabular}  
\end{center}
\vspace{-5mm}
\caption{Enumeration of certain reorientations based on  representatives of activity classes (Corollary  \ref{cor:enum-classes}).
 }
 \label{table:enum-classes}%
\end{table}

\begin{proof}
The first claims are obvious by the above construction and Proposition \ref{prop:act-classes}.
Then, the enumerations of the table are obvious by the first claims and by the Tutte polynomial expression \ref{eq:reorientation-activities}.
\EMEvder{mettre claimsur nb d'ative fixed $= 2^j$ dans enonce du coroallire?}%
\EMEvder{n'a t'on pas besoin de mineurs ? de formule de convolution? 
These two enumerations can also be easily derived by addressing separately the cyclic and acyclic parts of reorientations and using the Tutte polynomial convolution formula --- donner ref?}%
These enumerations are also implied by the forthcoming Tutte polynomial expression in terms of four refined orientation activities (Theorem \ref{th:expansion-reorientations}).
\end{proof}


%
%
%
%





%
%

\EMEvder{simplifier le paragraphe remarque ci-dessous ?}
\begin{remark}
\rm
Let us mention properties which are specific to graphs, studied in \cite{GiLV05}, and how they generalize. 
As observed above, activity classes can be represented by reorientations that are active fixed and dual-active fixed (see also Section \ref{sec:refined}). An application is that, for suitable orderings of the edge set of a graph (roughly when all branches of the smallest spanning tree are increasing), there is one and only one acyclic reorientation with a unique sink in each activity class of acyclic reorientations, see \cite[Section 6]{GiLV05}. 
Moreover, as shown in \cite[Section 7]{GiLV05}, the notion of active partition for a directed graph on a linearly ordered set of edges generalizes the notion of components of acyclic reorientations with a unique sink.
This last notion relies on certain linear orderings of the vertex set. It was studied  by Lass in \cite{La01} in relation with the chromatic polynomial, by Viennot in \cite{Vi86} in terms of heaps of pieces, and by Cartier and Foata in \cite{CaFo69} in terms of non-commutative monoids (see also Gessel \cite{Ge01}).  
For every such vertex ordering, there exists a consistent edge ordering such that active partitions exactly match these acyclic orientation components. 
With respect to this construction, our generalization by means of active partitions allows us to consider any orientation, any ordering of the edge set, 
along with a generalization to any oriented matroid.
\end{remark}

\section{The uniactive bijection between bounded/dual-bounded reorientations and their fully optimal uniactive internal/external bases}
\label{sec:bij-10}

\EMEvder{regrouper bounded et dual boudned ? permet de factoriser theoreme et proposition, et plus joli... je le laisse pour l'instant pas le temps !}%






This section mainly recalls (and also reformulates, reorganizes, or completes) definitions and important results from \cite{AB1} (No. 1 of the same series of papers).
Moreover, we first informally explain, as an 
overview, 
how 
these
can be considered under different perspectives and related to other papers of the series.%
\ss

We consider an ordered oriented matroid $M$ on $E$ which is bounded with respect to $p=\min(E)$ (or, in geometrical terms, 
a bounded region w.r.t. the element $p$ considered as a hyperplane at infinity, or, in terms of graphs, an acyclic bipolar directed graph whose unique source and unique sink are the extremities of $p$).
The main result of \cite{AB1} is that such an oriented matroid has a unique fully optimal basis, satisfying a simple combinatorial criterion. This directly yields a bijection between (pairs of opposite) bounded reorientations of $M$ w.r.t. $p$ and bases of $M$ with internal/external activity equal to $1/0$.

The existence and uniqueness of the fully optimal basis of a bounded region is a  difficult fundamental result, that can be seen by different manners, and has important connections with duality. 
First, its proof essentially relies upon a topological obstruction in oriented matroids, and a tricky use of this obstruction, see \cite[The Crescent Lemma 4.4 and Proposition 4.3]{AB1}.

Second, it witnesses a curious duality geometric property.
The fully optimal basis of a bounded region induces a flag of faces adjacent to this region (``Dual-adjacency'': the successive compositions of fundamental cocircuits of the basis yield positive covectors). Dually, the complementary basis of the dual induces a flag of faces adjacent to the corresponding bounded region of the dual obtained by reorienting $p$ (``Adjacency'': the successive compositions of fundamental circuits of the inital basis yield positive vectors, up to the sign of $p$). The fully optimal basis then appears naturally, as it is the unique basis satisfying at the same time these two properties ``Adjacency'' and ``Dual-adjacency''.
See Definition \ref{def:acyc-alpha2} below, and 
see the geometrical interpretation given after \cite[Proposition 3.3]{AB1} for more details.
See also the flag representations in Figures \ref{fig:ex-arrgt} and \ref{fig:D13refined} in Section \ref{sec:example}.
\EMEvder{*** ajouter ceci ? it is observed in GiLV04, AB3, AB4 unicity property + bijection  avec adjacence when real but not when non ecludiean}

Third, it can be seen as a refinement of (pseudo-)linear programming. 
See \cite[Chapter 10]{OM99} for information on pseudo-linear programming in oriented matroids.
In real hyperplane arrangements in general position or in uniform oriented matroids, building the optimal cocircuit of a real/pseudo linear program is equivalent to build the fully optimal basis \cite{GiLV04}. 
%
In general real hyperplane arrangements or oriented matroids,
we optimize a sequence of nested faces (the successive covectors obtained by composition of the fundamental cocircuits of the basis), each with respect to a sequence of objective functions (provided by the linearly ordered minimal basis of the matroid), yielding finally a unique fully optimal basis. 
This refines standard linear programming where just one vertex is optimized with respect to just one objective function,
but this can be computed inductively using standard linear programming.
Linear programming duality is then strengthened by the \ref{eq:act_mapping_strong_duality}  property recalled in 
Theorem \ref{th:active-duality} below,
and proved in \cite[Section 5]{AB1}%
\footnote{See footnote \ref{footnote:typo} in Section \ref{sec:prelim} for a correction in the statements of \cite[Proposition 5.1 and Theorem 5.3]{AB1}.
}%
.
This relation with linear programming is addressed in \cite{AB1}, and the construction of the fully optimal basis by this manner is detailed in \cite{AB3} (see also \cite{GiLV09} for a description in terms of real hyperplane arrangements, see also \cite{ABG2LP} for a reformulation  and a simplification  in the graph case). 

Fourth, it is noticeable that in one direction, from bases to reorientations, the bijection is given by a simple single pass algorithm  (Proposition \ref{prop:alpha-10-inverse}) whereas in the other direction, from reorientations to bases, it is more difficult than real/pseudo linear programming. Hence, this bijection can be thought of as a sort of ``one-way function''.

Fifth, a construction of this bijection can be made by deletion/contraction of the greatest element. It is detailed in \cite{AB4} (see also \cite[Section 6.1]{ABG2} or \cite{ABG2LP}\EMEvder{verifier ref} in the graph case).
This construction can be seen as an elaborated translation of the usual linear programming solving by variable/constraint deletion.
It is more direct and more simple from the formal structural viewpoint than the full optimality algorithm alluded to above, but not from the computational complexity viewpoint, as it involves an exponential number of minors,
whereas the previous construction involves a linear number of minors.
This deletion/contraction construction can be  equally made from the primal or the dual viewpoint, and this fact is non-trivial, as it is equivalent to the existence and uniqueness property of the fully optimal basis (see \cite{AB4} for more details).

Finally, let us mention that the two constructions mentioned above (full optimality algorithm, primal/dual deletion/contraction construction) do not give a proof of this existence and uniqueness result. On the contrary, one uses this fundamental combinatorial result to prove that the algorithms are well defined and yield the required result.

\label{subsec:alpha-bounded}

%



\eme{PEUT ETRE AJOUTER DEF 6 DANS GRAPHE AVEC CETTE RELAXATION?}%


\begin{definition}[{\cite[Definition 3.1]{AB1}}]
\label{def:acyc-alpha}
Let $M$ be an oriented matroid on a linearly ordered set $E$, which is bounded 
with respect to the minimal element $p$ of $E$. The \emph{fully optimal basis} $\alpha(\M)$ of $\M$ is the unique basis $B$ of $M$ such~that:
\smallskip

\bul for all $b\in B\setminus p$, the signs of $b$ and $\min(C^*(B;b))$ are opposite in $C^*(B;b)$;

\bul for all $e\in E\s B$, the signs of $e$ and $\min(C(B;e))$ are opposite in $C(B;e)$.

\hfill \emph{(Full Optimality Criterion)}
\end{definition}

\begin{definition}[equivalent to Definition \ref{def:acyc-alpha} by {\cite[Proposition 3.3]{AB1}}]
\label{def:acyc-alpha2}
Let $M$ be an oriented matroid on a linearly ordered set $E$, which is bounded 
with respect to the minimal element $p$ of $E$. The \emph{fully optimal basis} $\alpha(\M)$ of $\M$ is the unique basis $B$ of $M$ such~that, denoting $B=b_1<\dots<b_r$ and $E\s B=c_1<\dots<c_{n-r}$:
\smallskip

\bul the maximal covector $C^*(B;b_1)\circ\dots\circ C^*(B;b_r)$ is positive;
\hfill \emph{(Adjacency)}

\bul the maximal vector $C(B;c_1)\circ\dots\circ C(B;c_{n-r})$ is positive on $E\s \{p\}$ and negative on $p$.

\hfill \emph{(Dual-Adjacency)}
\end{definition}

The existence and uniqueness of  
a basis 
satisfying the 
criteria of
Definitions \ref{def:acyc-alpha} or \ref{def:acyc-alpha2}
is the main result of \cite{AB1}, namely \cite[Theorem 4.5]{AB1}. 
It yields the next theorem.
%
Notice that a bounded oriented matroid and its opposite have the same fully optimal basis.
An example of  fully optimal basis of a bounded region is given in Figure \ref{fig:ex-fob}.

\begin{thm}[Key theorem {\cite[Theorem 4.5]{AB1}}]
\label{thm:bij-10}
Let $M$ be a matroid on a linearly ordered set $E$ with $\min(E)=p$.
The mapping $M\mapsto \alpha(M)$ yields a bijection between all bounded reorientations of $M$ w.r.t. $p$, with fixed orientation for $p$, and all uniactive internal bases of $M$.
Equally, it yields a bijection between all pairs of opposite bounded reorientations of $M$ w.r.t. $p$, and all uniactive internal bases of $M$.
\end{thm}


The above mapping $M\mapsto \alpha(M)$  is called \emph{the uniactive bijection of $M$} (bounded case).
%
A direct computation of $\alpha(M)$ for bounded oriented matroids is given in \cite{AB3} by means of elaborations on linear programming (see also \cite{GiLV09} in real hyperplane arrangements, and  \cite{ABG2LP} in graphs). 
Moreover,  $\alpha(M)$ can be built by deletion/contraction, as shown in \cite{AB3} (see also \cite{ABG2} in graphs).
See more details in the introduction of the section.

The mapping $M\mapsto \alpha(M)$ was built in \cite{AB1} by its inverse, from uniactive internal bases to bounded reorientations, provided by a single pass algorithm over the base (see \cite[Figure 5]{AB1} for an example), or equally (dually) over its complement, or equally over the ground set,
so that the criterion for element signs from Definition \ref{def:acyc-alpha} is satisfied . We recall one of these algorithms below in Proposition \ref{prop:alpha-10-inverse}, that we will use later in Theorem \ref{th:basori}. 
%
%
Let us mention that internal uniactive bases can be characterized by several ways, see \cite{GiLV05, AB1, AB2-a}.
\EMEvder{redonner caracteristions de internal uniactive (ne pas priviliegeier que les roeitnations) donner en ref dans AB2a ?}%
\EMEvder{autre caracterisation de internal uniactive: GiVL05 Prop 2, bien a remettre quelque aprt}%

%


\begin{prop}[{\cite[Proposition 4.2, Algorithm 3]{AB1}}]
\label{prop:alpha-10-inverse}
Let $M$ be an oriented matroid on  a linearly ordered set of elements $E=\{e_1,\dots,e_n\}_<$. 
For a basis $B$ with internal activity $1$ and external activity $0$,
the two opposite reorientations of $M$ in $\alpha^{-1}(B)$ are computed by the following algorithm.


\begin{algorithme}
Reorient $e_1$ or not, arbitrarily.\par
For $k$ from $2$ to $n$ do\par
\hskip 10 mm if $e_k\in B$ then \par
\hskip 20 mm let $a=\min (C^*(B;e_k))$\par
\hskip 20 mm \vbox{\hsize=14cm reorient $e_k$ if necessary in order to have $a$ and $e_k$ with opposite signs in $C^*(B;e_k)$}\par
\hskip 10 mm if $e_k\not\in B$ then \par
\hskip 20 mm let $a=\min (C(B;e_k))$\par
\hskip 20 mm \vbox{\hsize=14cm reorient $e_k$ if necessary in order to have $a$ and $e_k$ with opposite signs in $C(B;e_k)$}\par
\end{algorithme}

\end{prop}

\begin{figure}[h]

\def\boxedplus{\fbox{+}}
\def\boxedminus{\fbox{--}}

 \begin{minipage}[c]{.2\linewidth}
   \centering
   	{\scalebox{0.2}{\includegraphics[width=10cm]{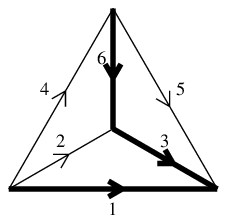}}}
   \end{minipage}
    \begin{minipage}[c]{.3\linewidth}
   \centering
{\scalebox{0.55}{\includegraphics[width=10cm]{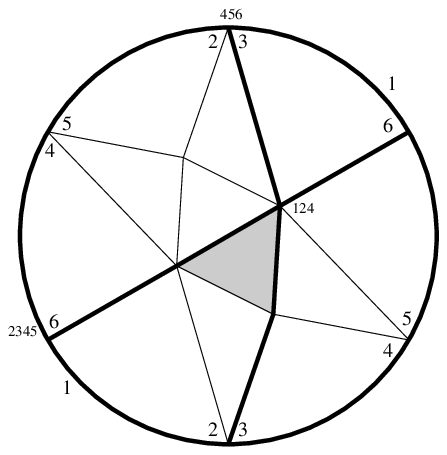}}}
   \end{minipage}
   %
%
     \begin{minipage}[c]{.5\linewidth}
   \centering
     	\scalebox{1}{      \includegraphics{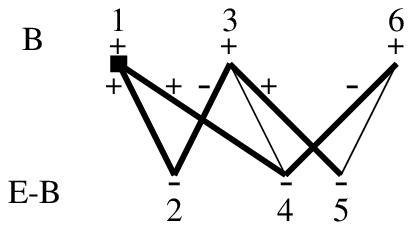}      }     	
\vspace{3mm}
     	{
\renewcommand{\arraystretch}{1.2}
\small
		\begin{tabular}{|c|c|c|c|c|c|c|}
		\hline
		 & \coltab{$C^*_1$} & \ptt{2} & \coltab{$C^*_3$}&\ptt{4}&\ptt{5}&\coltab{$C^*_6$} \\
		\hline
		\ptt{1}& $\pbvbulletn$ && &&&\\
		\coltab{$C_2$}& $\pbbulletn$ &$\mbvbulletn$ & $\mbbulletn$ &&&\\
		\ptt{3}& && $\pbvbulletn$ &&&\\
		\coltab{$C_4$}& $\pbbulletn$ &&$-$ &$\mbvbulletn$ &&$\mbbulletn$ \\
		\coltab{$C_5$}& &&$\pbbulletn$&&$\mbvbulletn$& $+$ \\
		\ptt{6}&&& &&&$\pbvbulletn$  \\
		\hline
		\end{tabular}				
}
   \end{minipage}

\caption{%
The basis $136$ is the fully optimal basis of the grey region (the arrangement is the same as in Figures \ref{fig:K4exbase256} and \ref{fig:K4act}).\break
We add signs to the fundamental graph and fundamental tableau of the basis in order to illustrate the full optimality criterion in a practical and visual way. This will be used again in Section \ref{sec:can-act-bij} to illustrate decompositions involving fully optimal bases of minors and to build on examples given in \cite{AB2-a}.
Let us now detail each part of the figure.
In the middle: the acyclic reorientation is depicted as the grey region, the basis is depicted with bold lines, its fundamental cocircuits are written at the vertices defined by the basis.
On the left: the corresponding orientation of the graph $K_4$, with its fully optimal spanning tree $136$ in bold.
On the upper right: the fundamental graph of the basis $136$, where signs are added accordingly with signs of the fundamental circuits and cocircuits of the basis w.r.t. this reorientation. 
Precisely: by convention, elements of $B$, resp. $E\s B$, are provided with a $+$, resp. $-$, sign; and then edges are provided with a $+$ or $-$ sign depending on the reorientation, so that one can read signs of elements in $C^*(B;b)$ for $b\in B$ and in $-C(B;e)$ for $e\in E\s B$ (as well, by orthogonality).
\EMEvder{***PTET A FAIRE : reprendre mieux phrase precedente***}%
Light edges are not given signs, they are not useful for the criterion that the basis is fully optimal.
On the bottom right: the fundamental tableau of the basis $136$, where signs are added accordingly with signs of the fundamental circuits and cocircuits of the basis w.r.t. this reorientation.
Signs of elements in $C^*(B;b)$ for $b\in B$, resp. in $-C(B;e)$ for $e\in E\s B$, appear in columns, resp. rows, of the tableau (by orthogonality). 
Precisely: by convention, diagonal elements corresponding to elements of $B$, resp. $E\s B$, are provided with a $\protect\pbvbulletn$ , resp. $\protect\mbvbulletn$, sign;  other non-zero elements of the tableau are given signs depending on the reorientation. They are either provided with a $\protect\pbbulletn$ or $\protect\mbbulletn$ sign, when they correspond to a minimal element of a fundamental circuit or cocircuit, or they are provided with a $+$ or $-$ sign, when they are not used in the full optimality criterion.
Finally, the reader is invited to check on the fundamental graph or tableau that the signs for the basis $136$ w.r.t. the given reorientation satisfy the full optimality criteria given by Definitions \ref{def:acyc-alpha} and \ref{def:acyc-alpha2}. On the tableau: smallest non-zero (non-diagonal) entries of rows, resp. columns. are all $\protect\pbbulletn$, resp.  $\protect\mbbulletn$. %
\EMEvder{ceci mis au nivau des autrs figures, mais pourrait etre ici... Conversely, given a basis, building a reorientation for which this basis is fully optimal consists in reorieting elements so that the signed fundamental graph or tableau has this shape, see Proposition \ref{prop:alpha-01-inverse}.}
\EMEvder{ajotuer : sign pattern provided by the boxed signs ?--- eviter repetitions avec figures suivatnes ?}
}
\label{fig:ex-fob}
\end{figure}


Now, let us extend by duality the above definitions and result to dual-bounded reorientations (this was not made explicitly in \cite{AB1}).
One can observe that Definition \ref{def:cyc-alpha0} below is contained in the following general Definition \ref{def:om-alpha}.

\begin{definition}
\label{def:cyc-alpha0}
Let $\M$ be an oriented matroid on a linearly ordered set $E$,
dual-bounded with respect to the minimal element $p$ of $E$. 
Then $M^*$ is bounded w.r.t. $p$ and we define 
\begin{equation}
\tag{Duality}
\label{eq:act_mapping_duality}
\alpha(M)=E\s \alpha(M^*).
\end{equation}
\end{definition}

\begin{definition}[equivalent to Definition \ref{def:cyc-alpha0} by Definitions \ref{def:acyc-alpha} and \ref{def:acyc-alpha2}]
\label{def:cyc-alpha1}
Let $\M$ be an oriented matroid on a linearly ordered set of elements,
dual-bounded with respect to the minimal element $p$ of $E$. 
%
Then,  $\alpha(\M)$ is the unique basis $B$ of $M$ such~that:
\smallskip

\bul for all $b\in B$,  the signs  of $b$ and $\min(C^*(B;b))$ are opposite in $C^*(B;b)$;
\smallskip

\bul for all $e\in (E\s B)\s\{p\}$,
the signs of $e$ and $\min(C(B;e))$ are opposite in $C(B;e)$.
\ss

\noindent Equivalently, $\alpha(\M)$  is the unique basis $B$ of $M$ such~that, denoting $B=b_1<\dots<b_r$ and $E\s B=c_1<\dots<c_{n-r}$:
\smallskip

\bul the maximal covector $C^*(B;b_1)\circ\dots\circ C^*(B;b_r)$ is positive on $E\s \{p\}$ and negative on $p$;

\bul the maximal vector $C(B;c_1)\circ\dots\circ C(B;c_{n-r})$ is positive.
\end{definition}


\begin{thm}[dual of Theorem \ref{thm:bij-10}]
\label{thm:bij-01}
Let $M$ be a matroid on a linearly ordered set $E$ with $\min(E)=p$.
The mapping $M\mapsto \alpha(M)$ yields a bijection between all dual-bounded reorientations of $M$ w.r.t. $p$, with fixed orientation for $p$, and all uniactive external bases of $M$.
Equally, it yields a bijection between all pairs of opposite dual-bounded reorientations of $M$ w.r.t. $p$, and all uniactive external bases of $M$.
\end{thm}

The above mapping $M\mapsto \alpha(M)$ is the dual-bounded case of the \emph{the uniactive bijection of $M$}.

\begin{prop}
\label{prop:alpha-01-inverse}
Let $M$ be an oriented matroid on  a linearly ordered set. 
For a basis $B$ with internal activity $0$ and external activity $1$,
the two opposite reorientations of $M$ in $\alpha^{-1}(B)$ are computed by exactly the same algorithm as in Proposition \ref{prop:alpha-10-inverse}.
\hfill\square
\end{prop}

Finally, let us recall a different and more involved duality   property of the active bijection, called \emph{active duality}, that can be seen as a strengthening of linear programming duality
(see \cite[Section 5]{AB1}%
\footnote{\label{footonote:typobis}See footnote \ref{footnote:typo}  in Section \ref{sec:prelim}  for a correction in the statements of \cite[Proposition 5.1 and Theorem~5.3]{AB1}.
}%
).
%
%
This important  property shows that the active bijection is compatible with the two canonical bijections provided by the following properties  (\cite[Propositions 5.1 and~5.2]{AB1}%
\textsuperscript{\ref{footonote:typobis}}%
).
For a linearly ordered set $E$ 
with at least two elements:
\begin{itemize}
\item an oriented matroid $M$ on  $E$ is bounded w.r.t. $p=\min(E)$ if and only if $-_pM$ is dual-bounded w.r.t. $p$ (if and only if $-_pM^*$ is bounded w.r.t. $p$);
\item a basis $B$ of a matroid $M$ on  $E$, with $p=\min(E)$ and $p'=\min(E\s \{p\})$, is uniactive internal if and only if $B\s \{p\}\cup \{p'\}$ is a uniactive external basis.
\end{itemize}


Combining these properties with usual duality yields the commutative diagram of Figure \ref{fig:bounded-duality-diagram}.

\begin{thm}%
[{\cite[Theorem 5.3]{AB1}}%
\textsuperscript{\ref{footonote:typobis}}%
]%
\label{th:active-duality}
Let $E$ be a linearly ordered set with $|E|>1$.
Let $\M$ be an oriented matroid on  $E$,
bounded with respect to $p=min(E)$.
 Let $p'=\min(E\s\{p\})$. We have:
\begin{equation}
\tag{Active Duality}
\label{eq:act_mapping_strong_duality}
\alpha(\M)=\Bigl(E\s \alpha(-_p\M^*)\Bigr)\setminus \{p'\}\cup \{p\}.
\end{equation}
\end{thm}



\begin{definition}%
[equivalent to Definition \ref{def:cyc-alpha0} by Theorem \ref{th:active-duality}]
\label{def:cyc-alpha2}
Let $E$ be a linearly ordered set with $|E|>1$.
Let $p=min(E)$ and $p'=\min(E\s\{p\})$.
Let $\M$ be an oriented matroid on  $E$,
dual-bounded with respect to $p$.
Then $-_p\M$ is bounded w.r.t. $p$, and 
we can define $\alpha(\M)$ by:
\begin{equation*}
\alpha(\M)=\alpha(-_p\M)\setminus \{p\}\cup \{p'\}.
\end{equation*}
\end{definition}
\smallskip

\def\hadistance{3cm}
\def\vadistance{0cm}
\def\hbdistance{6cm}
\def\vbdistance{0cm}
\def\vdistance{4cm}

\vspace{-3mm}

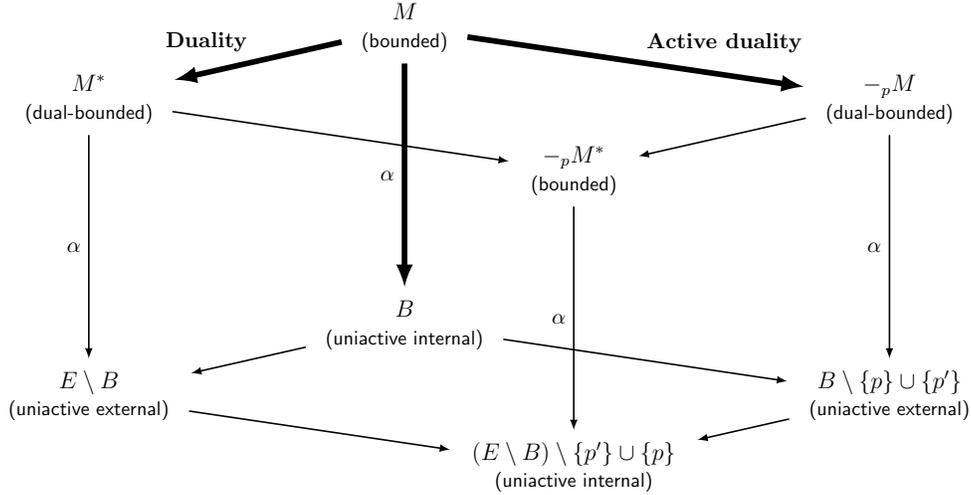
\begin{figure}[h]
\centering

\scalebox{0.75}{
\begin{tikzpicture}[->,>=triangle 90, thick,shorten >=1pt,auto, node distance=\hdistance,  thick,  
main node/.style={rectangle,font=\sffamily\large}
]

  \node[main node] (1) {\begin{tabular}{c}$M$\\ \small (bounded)\end{tabular}};
  \node[main node] (1a)[below left= \vadistance and \hadistance of 1] {\begin{tabular}{c}$M^*$\\ \small (dual-bounded)\end{tabular}};
  \node[main node] (1b)[below right= \vbdistance and \hbdistance of 1] {\begin{tabular}{c}$-_pM$\\ \small (dual-bounded)\end{tabular}};
   \node[main node] (1ab) [below left= \vadistance and \hadistance of 1b] {\begin{tabular}{c}$-_pM^*$\\ \small (bounded)\end{tabular}};

  \node[main node] (2) [below =\vdistance of 1] {\begin{tabular}{c}$B$\\ \small (uniactive internal)\end{tabular}};
  \node[main node] (2a) [below =\vdistance of 1a] {\begin{tabular}{c}$E\s B$\\  \small (uniactive external)\end{tabular}};
  \node[main node] (2b) [below =\vdistance of 1b] {\begin{tabular}{c} $B\s\{p\}\cup\{p'\}$\\ \small (uniactive external)\end{tabular}};
   \node[main node] (2ab)  [below =\vdistance of 1ab] {\begin{tabular}{c}  $(E\s B)\s\{p'\}\cup\{p\}$\\ \small (uniactive internal)\end{tabular}};

\path[every node/.style={font=\sffamily}]
	
	(1) edge [->, >=latex, line width=1mm] node [above left] 
	{\bf Duality}
	(1a)
	
	(1) edge [->, >=latex, line width=1mm] node [above right] 
	{\bf Active duality}
	(1b)
		
	(1a) edge [->, >=latex]
	(1ab)

	(1b) edge [->, >=latex] 
	(1ab)
		
	(2) edge [->, >=latex] 
	(2a)
	
	(2) edge [->, >=latex] 
	(2b)
		
	(2a) edge [->, >=latex]
	(2ab)

	(2b) edge [->, >=latex] 
	(2ab)
				
	(1) edge [->, >=latex, line width=1mm] node [left] 
	{$\alpha$}
	(2)
	
	(1a) edge [->, >=latex] node [left] 
	{$\alpha$}
	(2a)
	
	(1b) edge [->, >=latex] node [left] 
	{$\alpha$}
	(2b)
	
	(1ab) edge [->, >=latex] node [left] 
	{$\alpha$}
	(2ab)
;
\end{tikzpicture}
}
\label{fig:bounded-duality-diagram}
\caption{Commutative diagram of duality properties of the uniactive bijection. It involves the usual oriented matroid duality and the active duality (Theorem \ref{th:active-duality}).}
\end{figure}


\vspace{-5mm}
\section{The active basis of an ordered oriented matroid, and the canonical active bijection between reorientation activity classes and matroid bases}
\label{sec:can-act-bij}

\EMEvder{pour refined on peut ajouter aussi def par union over active fitlration....}




In this section,
we define the canonical active bijection
by means of the two decompositions into the case of $(1,0)$ or $(0,1)$ activities from the previous Sections \ref{sec:dec-seq} (for bases) and  \ref{sec:dec-mo} (for reorientations), along with the bijection for $(1,0)$ activities from Section \ref{sec:bij-10}.
One can compute this decomposition and glue together the fully optimal bases associated with the active minors in order to get the base associated with the initial oriented matroid. 
This canonical bijection between activity classes of reorientations and bases not only preserves activities and active elements, but also active partitions (Theorem \ref{th:alpha}).
Conversely, a single pass algorithm can be used to compute the inverse bijection, from bases to reorientations of a given oriented matroid.
It uses only fundamental circuits and cocircuits, 
by some propagation from the smallest to the greatest element (Theorem \ref{th:basori}).
The bijection and its inverse can be also built by deletion/contraction of the greatest element \cite{AB4}.

We call it the \emph{canonical} active bijection
since it is based on three canonical constructions, so that it finally depends only on the reorientation class of the oriented matroid, that is on the non-signed pseudosphere/hyperplane arrangement in terms of a topological representation (or the underlying undirected graph in the graph case).
See Observation \ref{obs:canonical}.
Let us also point out that, at every step of the construction, important duality properties are satisfied.

Formally, for an ordered oriented matroid $M$, we will get   \emph{the active basis of $M$}, which is denoted $\alpha(M)$, and \emph{the canonical active bijection of $M$}, which is the bijection between preimages (activity classes of reorientations of $M$) and images (bases of $M$) of the mapping $M\mapsto \alpha(M)$ applied to all reorientations of $M$.
%
%
First, we give several definitions for the 
active basis.
Then,  in Theorem \ref{th:alpha} below, we state that these definitions are well-defined and equivalent.

\begin{definition}
\label{def:om-alpha}
For an
oriented matroid $M$ on a linearly ordered set $E$, the \emph{active basis} $\alpha(M)$ of $M$ 
 satisfies the three following properties:

\begin{itemize}
\itemsep=2mm

\itemlabel{it1-def-alpha}{(F.o.b.)}
\item[\rm \itemref{it1-def-alpha}]
If $M$ is bounded with respect to $p=\min(E)$ then $\alpha(M)$ is the fully optimal basis of $M$.

\itemlabel{it2-def-alpha}{(Duality)}
\item[\rm \itemref{it2-def-alpha}]

$\alpha(M^*)=E\s \alpha(M).$

\itemlabel{it3-def-alpha}{(Induction)}
\item[\rm \itemref{it3-def-alpha}]
$\alpha(M)=\alpha(M/ F)\ \uplus\ \alpha(M(F))$
where $F$ is the complement of the union of all positive cocircuits of $M$ whose smallest element $a$ is the greatest possible smallest element of a positive cocircuit of~$M$ (i.e. $a$ is the greatest dual-active element of $M$).

\end{itemize}
\end{definition}

\begin{definition}[equivalent variants of Definition \ref{def:om-alpha}]
\label{def:om-alpha-var}
In Definition \ref{def:om-alpha}, the first property \itemref{it1-def-alpha} can be replaced with Definition \ref{def:cyc-alpha1}, assuming $M$ is dual-bounded w.r.t. $p=\min(E)$.
%
%
In Definition \ref{def:om-alpha}, the third property \itemref{it3-def-alpha} can be replaced with any of the following ones.

%
%

\begin{itemize} 
\itemlabel{it1-alpha-var}{(Ind.$^*$)}
\item[\rm \itemref{it1-alpha-var}]
 $\alpha(M)=\alpha(M/ F)\ \uplus\ \alpha(M(F))$
where $F$ is the union of all positive circuits of $M$ whose smallest element is the greatest possible smallest element of a positive circuit of $M$.

\itemlabel{it2-alpha-var}{(Ind.$_+$)}
\item[\rm \itemref{it2-alpha-var}]
$\alpha(M)=\alpha(M/ F))\ \uplus\ \alpha(M(F))$
where $F$ is  the complement of the union of all positive cocircuits of $M$ whose smallest element is greater than any given element of $E$.

\itemlabel{it3-alpha-var}{(Ind.$_+^*$)}
\item[\rm \itemref{it3-alpha-var}]
$\alpha(M)=\alpha(M/ F))\ \uplus\ \alpha(M(F))$
where $F$ is the union of all positive circuits of $M$ whose smallest element is greater than any given element of $E$.
\end{itemize}

\end{definition}

\begin{definition}[equivalent to Definition \ref{def:om-alpha}]
\label{def:alpha-seq-decomp}
Let $\M$ be an oriented matroid on a linearly ordered set of elements, with active filtration
$(F'_\ep, \ldots, F'_0, F_c , F_0, \ldots, F_\io)$.
If $\io+\ep=1$ then $\alpha(M)$ is defined by any of the equivalent definitions given in Section \ref{sec:bij-10} (bounded case if $\io=1$, dual-bounded case if $\ep=1$).
Otherwise,  $\alpha(M)$ is defined by
$$\alpha(\M)=\ \biguplus_{1\leq k\leq \io} \alpha\bigl( \M(F_k)/F_{k-1}\bigr)\ \uplus\ \biguplus_{1\leq k\leq \ep}
\alpha\bigl( \M(F'_{k-1})/F'_{k}\bigr).$$
\end{definition}

\begin{thm} 
\label{th:alpha}
Let $M$ be an oriented matroid on a linearly ordered set $E$. 
\begin{enumerate}
\item The image $\alpha(M)$ is well defined. Definitions \ref{def:om-alpha}, \ref{def:om-alpha-var}  and \ref{def:alpha-seq-decomp} are equivalent.
\item The image $\alpha(M)$ is a basis of $M$, and this basis has the same active filtration/partition as $M$, 
which implies in particular 
\vspace{-3mm}
\begin{eqnarray*}
\Int\bigl(\alpha(M)\bigr)&=&O^*(M),\\
\Ext\bigl(\alpha(M)\bigr)&=&O(M).
\end{eqnarray*}
\item The $2^{\io+\ep}$ reorientations of $M$ in the activity class of $M$, which has dual-activity $\io$ and activity $\ep$, and  are mapped onto the same  basis $\alpha(M)$.
\item The mapping $M\mapsto \alpha(M)$, applied to the set of reorientations of $M$, provides a surjection onto the set of bases, and a bijection between all activity classes of reorientations of $M$ and all  bases of $M$. 
\item In particular, we obtain the bijections listed in Table \ref{table:activity-classes}.
\end{enumerate}
\end{thm}

\vspace{-4mm}

\begin{table}[H]
\begin{center}
\def\interligne{&&\\[-11pt]}
\parindent=-1.5cm
\begin{tabular}{|l|l|c|}
\hline
\interligne
activitiy classes of reorientations & bases & $t(M;1,1)$\\
\interligne
act. classes of  acyclic reorientations & internal bases & $t(M;1,0)$\\
\interligne
act. classes of  totally cyclic reorientations & external bases & $t(M;0,1)$\\
 \interligne
bounded reorientations w.r.t. $\min(E)$& uniactive internal bases & $b_{1,0}=\beta(M)$\\
 \interligne
dual-bounded reorientations w.r.t. $\min(E)$& uniactive external bases & $b_{0,1}=\beta^*(M)$\\
\hline
\end{tabular}
\caption{Canonical active bijection enumeration (the third column indicates the Tutte polynomial evaluation or coefficient that counts the involved objects).}
\label{table:activity-classes}
\end{center}
\end{table}

\vspace{-8mm}
\begin{definition}
Let $M$ be an ordered oriented matroid on $E$.
The bijection between activity classes of reorientations of $M$ and  bases of $M$ provided by  the mapping $M\mapsto \alpha(M)$ applied to reorientations of $M$, is called the \emph{canonical active bijection} of $M$.
\end{definition}


\begin{observation}
\label{obs:canonical}
It is very important to observe that the canonical active bijection of $M$ depends only on the reorientation class of $M$ (in the sense that the canonical active bijection of $-_XM$ for $X\subseteq E$ is isomorphic to that of $M$ up to symmetric difference with $X$).
In other words: $\alpha(-_AM)$ depends only on the resulting oriented matroid $-_AM$, not on $M$ and $A$.
Equivalently, in terms of a pseudosphere arrangement representation of $M$,  the canonical active bijection depends only on the non-signed arrangement. 
More precisely: it is a bijection between signatures and bases of the non-signed arrangement, not depending on the choice of a reference signature. 
In other words: it depends only on the topology of the arrangement, not on an initial signature.
\EMEvder{alleger ce qui precede}%
%
In particular, activity classes of  regions of the non-signed arrangement are in bijection with internal bases, 
and  bounded regions w.r.t. $\min(E)$ of the non-signed arrangement,  on one side of $\min(E)$,  are in bijection with uniactive internal bases,
and these bijections are independent of any signature of the arrangement.
%
\end{observation}

\begin{observation}
\label{obs:induced-dec-seq-act-bij}
As a direct consequence of Definition \ref{def:alpha-seq-decomp}, we get the following result, which puts together (in terms of the active mapping) Observation \ref{obs:induced-dec-seq-ori} (in terms of active filtrations of oriented matroids) and Observation \ref{obs:induced-dec-seq-bas} (in terms of active filtrations of bases, coming from
{\cite[Observation \ref{a-lem:dec-seq-bas-observation-suite}]{AB2-a}}).

Let $\emptyset= F'_\ep\subset...\subset F'_0=F_c=F_0\subset...\subset F_\io= E$ be the active filtration of the ordered oriented matroid $M$
 (or equivalently of the basis $\alpha(M)$ by Theorem \ref{th:alpha}).
Let $F$ and $G$ be two subsets in this sequence such that $F\subseteq G$. 
We have:
$$\alpha(M)=\alpha\bigl(M(F)\bigr)\uplus \alpha\bigl(M(G)/F\bigr)\uplus \alpha\bigl(M/G\bigr).$$
\EMEvder{convolution formula?}
\end{observation}

\begin{proof}[Proof of Theorem \ref{th:alpha}]
Throughout the proof, we may call $\alpha$  the mapping $M\mapsto \alpha(M)$ applied to reorientations of $M$.
Definition \ref{def:alpha-seq-decomp} is properly defined. Let us use notations from this definition.
By Theorem \ref{th:dec-ori}, we have that
for all $1\leq k\leq \io$, $M(F_k)/F_{k-1}$ is bounded, and 
for all $1\leq k\leq \ep$, $M(F'_{k-1})/F_{k}$ is dual-bounded.
Then, by definitions in Section \ref{sec:bij-10} in the case where $\io+\ep=1$, we have that
for all $1\leq k\leq \io$, $\alpha(M(F_k)/F_{k-1})$ is an uniactive internal basis of $M(F_k)/F_{k-1}$, and 
for all $1\leq k\leq \ep$, $\alpha(M(F'_{k-1})/F_{k})$ is an uniactive external basis of $M(F'_{k-1})/F_{k}$.
Then, by Theorem \ref{th:dec_base}, we have that $\alpha(M)$ is a basis of $M$ with the same active filtration as $M$.

Since two opposite bounded, resp. dual-bounded, reorientations are mapped onto the same spanning tree by $\alpha$ (obvious by the definitions), we directly have by Definition \ref{def:act-class} that 
the $2^{\io+\ep}$ reorientations of $M$ in the activity class of $M$ are mapped onto the same  basis $\alpha(M)$.

Since $\alpha$ provides a bijection between all pairs of opposite bounded, resp. dual-bounded, reorientations of $M$ and uniactive internal, resp. external, bases of $M$ by Theorem \ref{thm:bij-10}, then we directly have by Theorem \ref{th:dec_base} and Theorem \ref{th:dec-ori} that $\alpha$ provides a bijection between  all activity classes of reorientations of $M$ and all  bases of $M$.

Now let us prove that Definitions \ref{def:om-alpha} and \ref{def:om-alpha-var} are well-defined and equivalent to 
Definition \ref{def:alpha-seq-decomp}.

First, observe that Properties \itemref{it1-def-alpha} and \itemref{it2-def-alpha} in 
Definition \ref{def:om-alpha} are consistent with Definition \ref{def:cyc-alpha0} of $\alpha(M)$ when $M$ is dual-bounded w.r.t. $p$.
Moreover, the variant of Property \itemref{it1-def-alpha} is consistent, since Properties \itemref{it1-def-alpha} and \itemref{it2-def-alpha} put together define $\alpha$ in the bounded and the dual-bounded case, as well as Definition \ref{def:cyc-alpha0} and Property \itemref{it2-def-alpha} put together.
So both Definitions \ref{def:om-alpha} and \ref{def:om-alpha-var} are well-defined and consistent in the bounded and the dual-bounded cases.

Second, observe that the variant \itemref{it1-alpha-var} of Property \itemref{it3-def-alpha} 
put together with Property \itemref{it2-def-alpha}, is consistent with
Property \itemref{it3-def-alpha} 
put together with Property \itemref{it2-def-alpha}. Indeed, they define the same properties for $M$ and $M^*$. Precisely: if $F=F_{\io-1}$ is the complementary set of the union of all positive cocircuits of $M$ with smallest element $a$, then $E\s F$ is the union of all positive circuits of $M^*$ with smallest element $a$, and we have $M^*/(E\s F)=\bigl(M(F)\bigr)^*$ and $M^*(E\s F)=\bigl(M/F\bigr)^*$. Then, definitions are consistent as we have simultaneously:
\vspace{-2mm}
\begin{eqnarray*}
 \alpha(M)&=&\alpha(M(F))\uplus \alpha(M/F), \\
 E\s \alpha(M)&=&\bigl(F\s \alpha(M(F))\bigr)\uplus \bigl((E\s F)\s\alpha(M/F)\bigr),\\
\alpha(M^*)&=&\alpha\bigl(\bigl(M(F)\bigr)^*\bigr)\uplus\alpha\bigl(\bigl(M/F\bigr)^* \bigr),\\ 
\alpha(M^*)&=&\alpha\bigl(M^*/(E\s F)\bigr)\uplus \alpha\bigl(M^*(E\s F)\bigr).
 \end{eqnarray*}
 
\vspace{-1mm}
Now, let us show briefly why Definition \ref{def:om-alpha}  is well-defined, assuming $M$ is not bounded  w.r.t. $p$ or not dual-bounded w.r.t. $p$.
If  $F\not=\emptyset$, defined in Property \itemref{it3-def-alpha}, then $M/F$ is bounded w.r.t. its smallest element, and $M(F)$ has one dual-active element less than $M$.
If $F=\emptyset$ then we apply Property \itemref{it2-def-alpha} and consider the dual $M^*$ (as above), then Property \itemref{it3-def-alpha} yields
a set $F^*\not=\emptyset$ such that $M(F^*)$ is dual-bounded w.r.t. its smallest element, and $M/F^*$ has one active element less than $M$. Also these two constructions can be used alternatively  in any order.

More precisely, first, let us assume that $\io>0$. 
Observe that the set $F$ considered in Property \itemref{it3-def-alpha}
is the set $F_{\io-1}$ of the active filtration of $M$, by Definition \ref{EG:def:ori-act-part}.
Observe that the active filtration of $M(F_{\io-1})$ is 
$(F'_\ep, \ldots, F'_0, F_c , F_0, \ldots, F_{\io-1})$
(this has been observed in Section \ref{sec:dec-mo}, and it is direct from Theorem \ref{th:dec-ori}). Then, applying Definition \ref{def:alpha-seq-decomp} to $M(F_{\io-1})$, we directly have that $\alpha(M)=\alpha(M(F_{\io-1})\uplus \alpha(M/F_{\io-1})$, which is exactly Property \itemref{it3-def-alpha} in Definition \ref{def:om-alpha}.
Now, assume that $\ep>0$, then we use Property \itemref{it2-def-alpha} in order to apply the same reasoning in the dual $M^*$, as explained above with Property \itemref{it1-alpha-var},
and we get similarly that Definition \ref{def:om-alpha}  is  equivalent to Definition \ref{def:alpha-seq-decomp}. At the same time, we have proved that the definition provided by the variant \itemref{it1-alpha-var} is also well-defined and equivalent to Definition \ref{def:alpha-seq-decomp}.

Finally, let us observe that the variant \itemref{it2-alpha-var} given in Definition \ref{def:om-alpha-var}, which allows to consider various possible sets $F$ instead of one set $F$ in Property \itemref{it3-def-alpha}, is consistent.
Indeed, by definition of such a set $F$, there exists $k$, $0\leq k\leq \io-1$ such that $F=F_k$. Hence, the active filtration of $M(F)$ is
$(F'_\ep, \ldots, F'_0, F_c , F_0, \ldots, F_{k})$ and the active filtration of $M/F$ is $(F_k\s F_k,F_k\s F_k,\ldots,E\s F_k)$
(as above, this has been observed in Section \ref{sec:dec-mo}, and it is direct from Theorem \ref{th:dec-ori}). So we have $\alpha(M)=\alpha(M(F))\uplus \alpha(M/F)$ by Definition \ref{def:alpha-seq-decomp} applied to $M(F)$ and $M/F$.
So, the variant \itemref{it2-alpha-var} yields an equivalent definition of $\alpha$.
Dually, we obtain that the variant \itemref{it3-alpha-var} also yields an equivalent definition of $\alpha$.
\end{proof}

\emenew{dessous ancienne discussion poru jsutifier theorem, gardee dans ABG2}%
\vspace{-1mm}
Let us illustrate the construction by continuing the running example of $K_4$.
Figures \ref{fig:exbasedecomp-136}, \ref{fig:exbasedecomp-126}, \ref{fig:exbasedecomp-146}, and \ref{fig:exbasedecomp-256} give the details of the sign pattern that characterize the active basis, on four different representative situations.
Observe that this sign pattern depends only on the fundamental graph/tableau of the basis.
Figure \ref{fig:exK4-complete-primal-dual} gives the complete canonical active bijection of $K_4$.

\vspace{-1mm}
\begin{figure}[H] 

   \begin{minipage}[c]{.46\linewidth}
   \centering
   	\scalebox{1.3}{
      \includegraphics{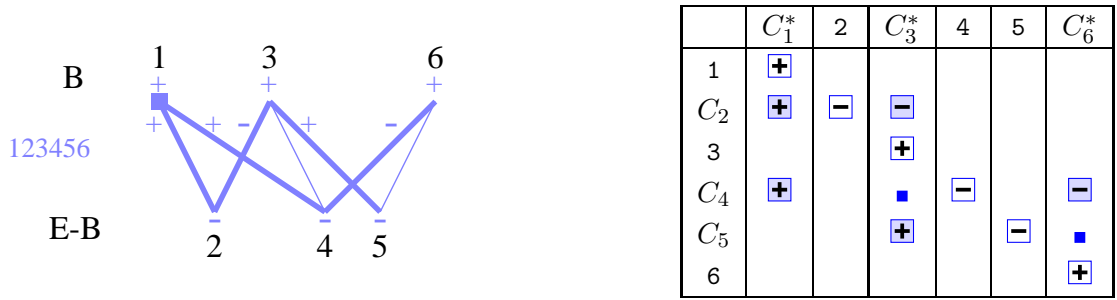}
      }
   \end{minipage}
   \hfill
      \begin{minipage}[c]{.46\linewidth}
      \centering
      \renewcommand{\arraystretch}{\taille}
		\begin{tabular}{|c|c|c|c|c|c|c|}
		\hline
		 & \coltab{$C^*_1$} & \ptt{2} & \coltab{$C^*_3$}&\ptt{4}&\ptt{5}&\coltab{$C^*_6$} \\
		\hline
		\ptt{1}& $\pbvbullet$ && &&&\\
		\coltab{$C_2$}& $\pbbullet$ &$\mbvbullet$ & $\mbbullet$ &&&\\
		\ptt{3}& && $\pbvbullet$ &&&\\
		\coltab{$C_4$}& $\pbbullet$ &&$\bgbullet$ &$\mbvbullet$ &&$\mbbullet$ \\
		\coltab{$C_5$}& &&$\pbbullet$&&$\mbvbullet$& $\bgbullet$ \\
		\ptt{6}&&& &&&$\pbvbullet$  \\
		\hline
		\end{tabular}												
   \end{minipage}
\caption{Signed fundamental graph/tableau of basis $136$ with activities $(1,0)$ and active partition $E=123456$.\break
This figures continues \cite[Figure \ref{a-fig:exbasedecomp-136}]{AB2-a}, and
 repeats the right part of Figure \ref{fig:ex-fob}, using the same sign representation except that useless signs are replaced with the $\protect\bgbullet $  symbol.
A uniactive basis of $M$ 
equals $\alpha(M)$
if and only if its signed fundamental graph/tableau satisfies the full optimality criterion,    that is satisfies the sign pattern given in the figure (Section~\ref{sec:bij-10}).
Let us recall that the cases of $(1,0)$-active and $(0,1)$-active bases can be handled the same way, up to exchanging the role of $B$ and $E\s B$ in the fundamental graph, or up to transposing the fundamental tableau.
Also, let us recall that, 
given a uniactive basis, and given its (non-signed) fundamental graph/tableau, one can build  a reorientation of $M$ for which this basis is fully optimal by reorienting elements one by one so that the signed fundamental graph/tableau w.r.t. the reorientation satisfies the full optimality criterion (Propositions \ref{prop:alpha-10-inverse} and~\ref{prop:alpha-01-inverse}).%
\EMEvder{raccourcir ce qui pr?c?de ?}
\EMEvder{*** layout au lieu de pattern ? bof non je ne crois pas}
}
\label{fig:exbasedecomp-136}
\end{figure}

\vspace{-1mm}

\begin{figure}[H] 
   \begin{minipage}[c]{.46\linewidth}
   \centering
   	\scalebox{1.3}{
      \includegraphics{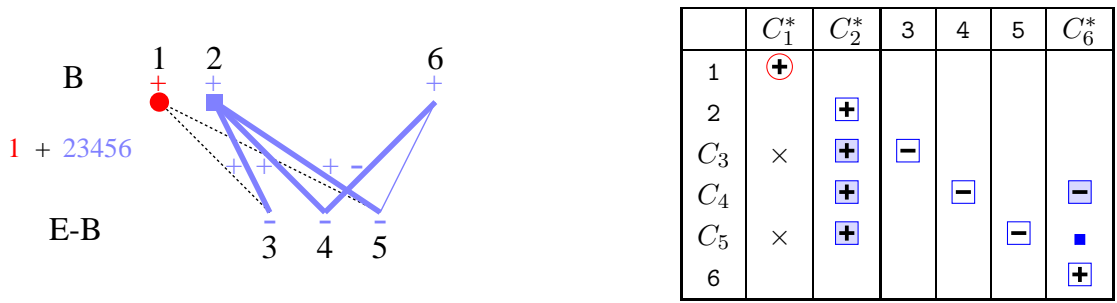}
      }
   \end{minipage}
   \hfill
      \begin{minipage}[c]{.46\linewidth}
      \centering
      \renewcommand{\arraystretch}{\taille}
		\begin{tabular}{|c|c|c|c|c|c|c|}
		\hline
		 & \coltab{$C^*_1$} & \coltab{$C^*_2$} & \ptt{3}&\ptt{4}&\ptt{5}&\coltab{$C^*_6$} \\
		\hline
		\ptt{1}& $\prvbullet$ && &&&\\
		\ptt{2}& & $\pbvbullet$& &&&\\
		\coltab{$C_3$}& $\gbullet$ &$\pbbullet$ & $\mbvbullet$ &&&\\
		\coltab{$C_4$}&  &$\pbbullet$&&$\mbvbullet$&&$\mbbullet$ \\
		\coltab{$C_5$}&$\gbullet$&$\pbbullet$& &&$\mbvbullet$ &$\bgbullet$ \\
		\ptt{6}& &&&&& $\pbvbullet$ \\
		\hline
		\end{tabular}												
   \end{minipage}
\caption{Signed fundamental graph/tableau of basis $126$ with activities $(2,0)$ and active partition $E=1+23456$.\break
This figures continues \cite[Figure \ref{a-fig:exbasedecomp-126}]{AB2-a} by adding the sign pattern that the basis must satisfy to be equal to $\alpha(M)$. Observe that this sign pattern depends only on the fundamental graph/tableau, not on the whole structure. The decomposition of the (fundamental graph/tableau of) the basis from \cite{AB2-a} yields the active partition $E=1+23456$ and the active minors $M(1)$ and $M/1$. Here the first minor consists of a single isthmus.
Accordingly with the construction of this section, the basis 
equals $\alpha(M)$
 when the two bases induced in the two minors $M(1)$ and $M/1$ satisfy the full optimality criterion, meaning that the subgraphs/subtableaux induced on $1$ and on $23456$ satisfy the same sign pattern as illustrated in Figure \ref{fig:exbasedecomp-136}.
\EMEvder{definir subgraphj, subtableau ?}%
In this figure and the next ones, for one part of the tableau, we use circled symbols such as $\protect\prvbullet$, and for the other part we use boxed symbols  such as $\protect\pbvbullet$, with the same meanings as in Figure \ref{fig:ex-fob}.
The elements of fundamental circuits or cocircuits which disappear when restricting to a part (i.e. elements that do not belong to fundamental circuits or cocircuits induced in the two minors), are depicted by dashed edges in the fundamental graph, and by  the symbol $\gbullet$ in the fundamental tableau.
Given a basis, and given its (non-signed) fundamental graph or tableau, one can build  a reorientation of $M$ in the preimage of this basis by $\alpha$ by reorienting elements one by one so that the signed fundamental graph/tableau w.r.t. the reorientation corresponds to the pattern given in the figure (see Theorem \ref{th:basori-refined} for a full statement).
}
\label{fig:exbasedecomp-126}
\end{figure}

\begin{figure}[H] 
   \begin{minipage}[c]{.46\linewidth}
   \centering
   	\scalebox{1.3}{
      \includegraphics{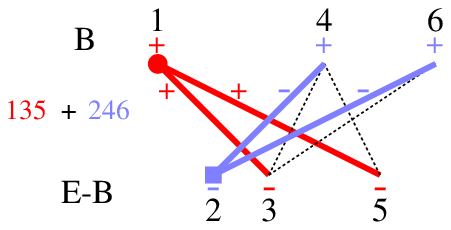}
      }
   \end{minipage}
   \hfill
      \begin{minipage}[c]{.46\linewidth}
      \centering
      \renewcommand{\arraystretch}{\taille}
		\begin{tabular}{|c|c|c|c|c|c|c|}
		\hline
		 & \coltab{$C^*_1$} & \ptt{2} &\ptt{3} &\coltab{$C^*_4$}&\ptt{5}&\coltab{$C^*_6$} \\
		\hline
		\ptt{1}& $\prvbullet$ && &&&\\
		\coltab{$C_2$}& &$\mbvbullet$ & &$\mbbullet$&&$\mbbullet$\\
		\coltab{$C_3$}&$\prbullet$ && $\mrvbullet$  &$\gbullet$ &&$\gbullet$ \\
		\ptt{4}&  && &$\pbvbullet$ && \\
		\coltab{$C_5$}&$\prbullet$ &&  &$\gbullet$&$\mrvbullet$ & \\
		\ptt{6}& &&&&& $\pbvbullet$ \\
		\hline
		\end{tabular}											
   \end{minipage}
\caption{Signed fundamental graph/tableau of basis $146$ with activities $(1,1)$ and active partition $E=135+246$.\break
This figures continues \cite[Figure \ref{a-fig:exbasedecomp-146}]{AB2-a} 
by adding the sign pattern that the basis must satisfy to be equal to $\alpha(M)$. Comments on 
Figure \ref{fig:exbasedecomp-126} (and Figure \ref{fig:exbasedecomp-136})
concern this figure as well, considering that we deal here with a $(1,0)$-active restriction of the fundamental graph/tableau to $135$, yielding $1=\alpha(M/246)$, and a $(0,1)$-active restriction of the fundamental graph/tableau to $246$, yielding $46=\alpha(M(246))$.
}
\label{fig:exbasedecomp-146}
\end{figure}

\begin{figure}[H] 
   \begin{minipage}[c]{.46\linewidth}
   \centering
   	\scalebox{1.3}{
      \includegraphics{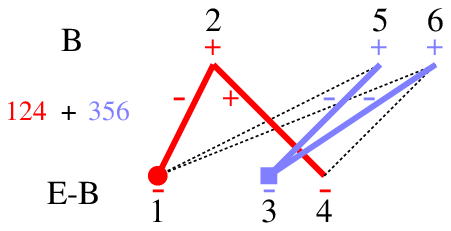}
      }
   \end{minipage}
   \hfill
      \begin{minipage}[c]{.46\linewidth}
      \centering
      \renewcommand{\arraystretch}{\taille}
		\begin{tabular}{|c|c|c|c|c|c|c|}
		\hline
		 & \ptt{1} & \coltab{$C^*_2$} &\ptt{3} &\ptt{4}&\coltab{$C^*_5$}&\coltab{$C^*_6$} \\
		\hline
		\coltab{$C_1$}& $\mrvbullet$ &$\mrbullet$& &&$\gbullet$&$\gbullet$\\
		\ptt{2}& &$\prvbullet$ & &&&\\
		\coltab{$C_3$}& && $\mbvbullet$  &&$\mbbullet$ &$\mbbullet$ \\
		\coltab{$C_4$}&  &$\prbullet$ & &$\mrvbullet$ &&$\gbullet$  \\
		\ptt{5}& &&  &&$\pbvbullet$ & \\
		\ptt{6}& &&&&& $\pbvbullet$ \\
		\hline
		\end{tabular}									
   \end{minipage}
\caption{Signed fundamental graph/tableau of basis $256$ with activities $(0,2)$ and active partition $E=124+356$.\break
This figures continues \cite[Figure \ref{a-fig:exbasedecomp-256}]{AB2-a} 
by adding the sign pattern that the basis must satisfy to be equal to $\alpha(M)$. Comments on 
Figure \ref{fig:exbasedecomp-126}  (and Figure \ref{fig:exbasedecomp-136})
concern this figure as well, considering that we deal here with a $(0,1)$-active restriction of the fundamental graph/tableau to $124$, yielding $2=\alpha(M/356)$, and a $(0,1)$-active restriction of the fundamental graph/tableau to $356$, yielding $56=\alpha(M(356))$.}
\label{fig:exbasedecomp-256}
\end{figure}


\EMEvder{attention figures modifiees and AB2a avec bleu plus clair... a verifier !}




\EMEvder{je n'ai pas mis l'exemple donne pour decompsiton 123+456 ! dommage ?}
\begin{figure} 
   \centering
   	\scalebox{1.3}{
      \includegraphics{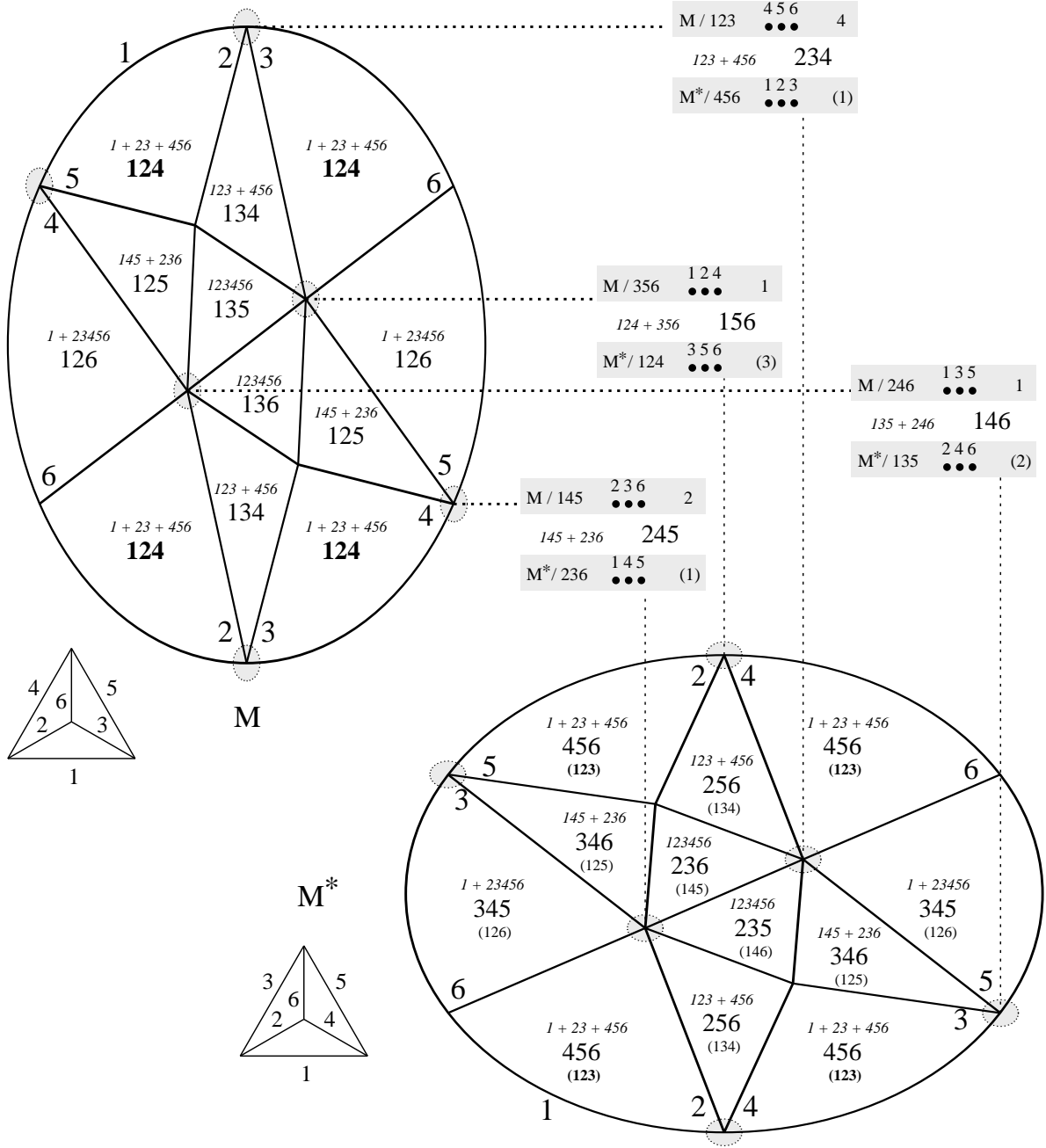}
      }
\caption{Geometrical representation of the entire canonical active bijection for $M=K_4$, 
continuing and completing the running example begun with Figures \ref{fig:K4exbase256} and \ref{fig:K4act}. 
Since the canonical active bijection depends only on the reorientation class, that is on the non-signed arrangement, we do not specify the signature of the arrangement.
Internal, resp. external, bases, associated with acyclic, resp. totally cyclic, reorientations of $M$ are written in regions of the primal, resp. dual, arrangement. Active partitions are written in italic.
The minimal basis $124$ is written in bold.
In the dual arrangement,   cobases are written in brackets (corresponding to acyclic reorientations of $M^*$).
Other bases are associated with combinations of acyclic orientations of $M/F$ and acyclic orientations of $M^*/(E\s F)$,  for cyclic flats $\emptyset \subset F\subset E$, where $F=123$, $145$, $246$ or $356$.
We represent these minors in grey boxes with their associated bases, and we indicate with dashed lines how these cyclic flats situtate geometrically in the primal and dual arrangements.
Details of the construction for bases 136, 126, 146, and 256 are given by Figures \ref{fig:exbasedecomp-136}, \ref{fig:exbasedecomp-126}, \ref{fig:exbasedecomp-146}, and \ref{fig:exbasedecomp-256}, respectively.
This figure will be exhaustively completed in Section \ref{sec:example},  by Figures  \ref{fig:fig:K4bij-primal},  \ref{fig:fig:K4bij-dual}, \ref{fig:fig:K4bij-cyclic-flats}.
\EMEvder{The Tutte polynomial equals...? dit dans section example}%
}
\label{fig:exK4-complete-primal-dual}
\end{figure}



Now, let us give two  constructions of 
the inverse of the canonical active bijection, from bases to reorientation activity classes.
The first  simply consists in rephrasing Definition \ref{def:alpha-seq-decomp} in the inverse way. It can be combined with Proposition \ref{prop:alpha-10-inverse} to compute the preimages of involved uniactive bases.




\EMEvder{attetnion ordre $B,E\s B$ a ete change dans section bases... PAS COMPRIS MA NOTE.. a verifier}

\EMEvder{attention definir ce qu'on apele reorienter ?}


\begin{prop}[inverse of Definition \ref{def:alpha-seq-decomp}] 
\label{prop:preimage-basis}
Let $M$ be an ordered oriented matroid on $E$. Let $B$ be a basis of $M$, with active filtration  $\emptyset= F'_\ep\subset...\subset F'_0=F_c=F_0\subset...\subset F_\io= E$.
Let us denote $\alpha_M^{-1}(B)$ the set of reorientations of $M$ whose active basis is $B$. Then,
$$\alpha_M^{-1}(B)\ =\ 
\bigtimes_{1\leq k\leq\io}
\alpha_{M(F_k)/F_{k-1}}^{-1}(B\cap (F_{k}\s F_{k-1}))\
\times \ 
\bigtimes_{1\leq k\leq\ep}
\alpha_{M(F'_{k-1})/F'_{k}}^{-1}(B\cap (F'_{k-1}\s F_{k}))\
   $$
where $\times$ means that the $2^{\io+\ep}$ resulting reorientations of $M$ are inherited from the reorientations of the involved minors the natural way (each induced basis in the involved active minors of $B$ is uniactive and has 2 preimages).%
\EMEvder{mal dit ? e eplus egalite a verifier, recuperee de ABG2 derniere version}%
\qed
\end{prop}

The second construction of the inverse consists in a simple algorithm, by a single pass over $E$ following the ordering, signing the arrangement element after element the suitable way to fit in the active basis criterion.
Furthemore, it uses only the fundamental circuits and cocircuits of the basis, not the whole oriented matroid structure. 
It combines the computation of the active filtration (in fact the active partition, see Definition \ref{def:act-seq-dec-base}) of the basis from \cite[Proposition \ref{a-prop:basori-partact}]{AB2-a}, along with the computation, in each induced minor, of the bounded (or dual-bounded) reorientation associated with a uniactive internal (or external) basis from Proposition \ref{prop:alpha-10-inverse} (or Proposition  \ref{prop:alpha-01-inverse}), consistently with Definition \ref{def:alpha-seq-decomp}.

\begin{thm}[{combination of Propositions \ref{prop:alpha-10-inverse} and \ref{prop:alpha-01-inverse} with \cite[Proposition \ref{a-prop:basori-partact}]{AB2-a}}]
\label{th:basori}
Let $M$ be a matroid on a linearly ordered set of elements $E=e_1<\ldots<e_n$.
Let $B$ be a basis of $M$. 
%
In the algorithm below, the active partition of $B$ is computed as a mapping, denoted $\ass$, from $E$ to $\Int(B)\cup \Ext(B)$, that maps an element onto the smallest element of its part in the active partition of $B$. An element is called internal, resp. external, if its image is in $\Int(B)$, resp. $\Ext(B)$.
The set of $2^{|\Int(t)|+|\Ext(B)|}$ reorientations formed by the preimages of $B$ under $\alpha$, denoted here $\alpha^{-1}(B)$,  is computed by doing all possible arbitrary choices to reorient or not an element
during the algorithm.
%
%
Equivalently, those preimages under $\alpha$ can also be retrieved from one another since we have
$$\alpha^{-1}(B)=\{\ A\ \triangle \ \ass^{-1}(P\cup Q)\ \mid\ 
P\subseteq \Int(B),\ Q\subseteq \Ext(B),\  A\in \alpha^{-1}(B) \ \}.$$
The algorithm consists in a single pass over $E$ and only relies upon the fundamental graph/tableau of $B$.
 Note that the rules when $e_k\in B$ are dual to the rules when $e_k\not\in B$.
\EMEvder{dans AB2a ceci, mais ici bof pour internal/external: Note that the rules when $e_k\in B$ are dual to the rules when $e_k\not\in B$, and that the rules when $e_k$ is internal are dual to the rules when $e_k$ is external.}

%

%
%

\begin{algorithme} \par

\underbar{Input}: a basis $B$ of $M$\par
\underbar{Output}: a reorientation of $M$ in $\alpha^{-1}(B)$ (along with the active partition of $B$)\par

For $k$ from $1$ to $n$ do\par

\hskip 5mm if $e_k\in B$ then \par
\hskip 5mm\hskip 5mm if $e_k$ is internally active  w.r.t. $B$ then\par
\hskip 5mm\hskip 10mm $e_k$ is internal \par
\hskip 5mm\hskip 10mm let $\ass(e_k):=e_k$\par
\hskip 5mm\hskip 10mm reorient $e_k$ or not, arbitrarily \par
\hskip 5mm\hskip 5mm otherwise\par

\hskip 5mm\hskip 10mm {\it\small (active partition computation)}

\hskip 5mm\hskip 10mm it there exists $c<e_k$ external in $C^*(B;e_k)$ then\par

\hskip 5mm\hskip 15mm $e_k$ is external\par
\hskip 5mm\hskip 15mm let $c$ $\in$ $C^*(B;e_k)$ with $c<e_k$, $c$ external and $\ass(c)$ the greatest possible \par
\hskip 5mm\hskip 15mm let $\ass(e_k):=\ass(c)$ \par

\hskip 5mm\hskip 10mm otherwise\par
\hskip 5mm\hskip 15mm $e_k$ is internal\par
\hskip 5mm\hskip 15mm let $c$ $\in$ $C^*(B;e_k)$ with $c<e_k$ and $\ass(c)$ the smallest possible\par
\hskip 5mm\hskip 15mm let $\ass(e_k):=\ass(c)$ \finsi\par

\hskip 5mm\hskip 10mm {\it\small (reorientation computation)}

\hskip 5mm\hskip 10mm let $a$ be the smallest possible in $C^*(B;e_k)$ with $\ass(a)=\ass(e_k)$\par
\hskip 5mm\hskip 10mm reorient $e_k$ if necessary so that $e_k$ and $a$ have opposite signs in 
$C^*(B;e_k)$

\hskip 5mm if $e_k\not\in B$ then \EMEvder{{\emph{(note: the below rules are dual to the above ones)\EMEvder{deja mis dans enonce}}}}\par
\hskip 5mm\hskip 5mm if $e_k$ is externally active w.r.t. $B$ then \par
\hskip 5mm\hskip 10mm $e_k$ is external\par
\hskip 5mm\hskip 10mm  $\ass(e_k):=e_k$\par
\hskip 5mm\hskip 10mm reorient $e_k$ or not, arbitrarily \par
\hskip 5mm\hskip 10mm \emph{or} with the same sign as in $\M$ if and only if $e_k\not\in Q$ \emph{(to compute $\alpha_\M^{-1}(X)$)}\par
\hskip 5mm\hskip 5mm otherwise\par

\hskip 5mm\hskip 10mm {\it\small (active partition computation)}

\hskip 5mm\hskip 10mm if there exists $c<e_k$ internal in $C(B;e_k)$ then\par

\hskip 5mm\hskip 15mm $e_k$ is internal\par
\hskip 5mm\hskip 15mm let $c$ $\in$ $C(B;e_k)$ with $c<e_k$, $c$ internal and $\ass(c)$ the greatest possible\par
\hskip 5mm\hskip 15mm let $\ass(e_k):=\ass(c)$\par
\hskip 5mm\hskip 10mm otherwise\par

\hskip 5mm\hskip 15mm $e_k$ is external\par
\hskip 5mm\hskip 15mm let $c$ $\in$ $C(B;e_k)$ with $c<e_k$ and $\ass(c)$ the smallest possible\par
\hskip 5mm\hskip 15mm let $\ass(e_k):=\ass(c)$ \finsi\par

\hskip 5mm\hskip 10mm {\it\small (reorientation computation)}

\hskip 5mm\hskip 10mm let $a$ be the smallest possible in $C(B;e_k)$ with $\ass(a)=\ass(e_k)$\par

\hskip 5mm\hskip 10mm reorient $e_k$ if necessary  so that $e_k$ and $a$ have opposite signs in 
$C(B;e_k)$\par

\end{algorithme}%
\end{thm}
\emenew{QUESTION A VOIR AVEC AB2:
dans algo basori, ne peut-on prendre le c (qui sert a definir partie) directement egal au e (ou a) qui sert a comparer 
orietnations dans circuit/cocyle fondamental ?}%
%

\begin{proof}
The computation of the active partition of $B$, that is of the mapping $\ass$, is exactly the algorithm given in
 \cite[Proposition \ref{a-prop:basori-partact}]{AB2-a}.
 By Definition \ref{def:alpha-seq-decomp}, and since $\alpha$ preserves active filtrations by Theorem \ref{th:alpha}, the reorientations associated with $B$ are obtained from the reorientations of the minors induced by the active filtration of $B$.
 
If $e_k$ is internally or externally active, then it can be obviously reoriented or not, arbitrarily, as the reorientation of the associated minor in the active filtration of $B$ will follow this initial reorientation. Each of these minors will get two possible opposite reorientations, yielding  the reorientation activity class of $M$ associated with $B$.

Now, assume that $e_k\in B$ is not internally active, and let $a$ be the smallest possible in $C^*(B;e_k)$ with $\ass(a)=\ass(e_k)$. Let $M'$ be the minor of $M$ associated with the active filtration of $B$ whose ground set is $\ass^{-1}(\ass(a))$. By Definition \ref{def:act-seq-dec-base} of the active filtration of $B$, we have $M'=M(G)/F$ for some $F\subseteq G\subseteq E$,
and  $B'=B\cap (G\s F)$ is a uniactive internal or external basis of $M'$.
Since $B'$ is a basis of $M'$, we have $C^*_{M'}(B';e_k)=C^*_M(B;e_k)\cap (G\s F)$
 (this is an easy matroid result, details are given in \cite[Property \ref{a-pty:fund_graph_flat}]{AB2-a}).
So, by definition of $a$,  we have $a=\min\bigl(C^*_{M'}(B';e_k)\bigr)$.
This implies $a\not=e_k$, otherwise $e_k$ would be internally active in the basis $B'$ of $M'$, and hence internally active in the basis $B$ of $M$ (by properties of the active filtration of $B$), which has been forbidden by assumption.
  By Proposition \ref{prop:alpha-10-inverse} or Proposition \ref{prop:alpha-01-inverse}, the bounded or dual-bounded reorientation of $M'$ associated with $B'$ by $\alpha$ is obtained by reorienting if necessary $e_k$ so that  $e_k$ and $a$ have opposite signs in 
$C^*_{M'}(B';e_k)$, that is so that  $e_k$ and $a$ have opposite signs in 
$C^*_{M}(B;e_k)$, which is exactly the algorithm statement.

The case where $e_k\not\in B$ is dual to the above case. Hence the algorithm is correct.
\end{proof}

Before ending this section, let us mention that there exist a construction of the canonical active bijection by deletion/contraction, that builds the whole bijection in a global way from the bijections in $M\backslash \omega$ and $M/\omega$, where $\omega$ is the greatest element. 
%
Actually, one can build by deletion/contraction various classes of activity preserving correspondences between reorientations and bases, among which the canonical active bijection is uniquely and canonically determined.
These features are developed in \cite{AB4} (see also \cite[Section 6]{ABG2}).
\EMEvder{verifier ref}%
%
The next final remark shows that the decomposition technique used in this section yields another whole class of activity preserving (actually active partition preserving) correspondences between reorientations and bases, among which the canonical active bijection is also uniquely and canonically determined.

\EMEvder{dessous en commentaire remarque sur 'depending only on the roeientation class' enlevee car elle etait sous forme de mappin $\psi:2^E\rightarrow 2^E$...}%

\begin{remark}[General decomposition framework for classes of activity preserving bijections]
\label{rk:preserv-act-bij-class}
%
%
%
\rm
Let us observe that the  construction by decomposition of this section is rather independent of the construction of Section \ref{sec:bij-10}.
In the construction of this section, and in particular in Definition \ref{def:alpha-seq-decomp}, one can replace everywhere the notation $\alpha$ with any mapping $\psi$ from reorientations to bases, assuming it is preliminary defined for ordered bounded/dual-bounded oriented matroids and provides a bijection with uniactive internal/external bases. Then Theorem \ref{th:alpha} can be directly extended, relying upon the two decompositions of Theorems \ref{th:dec_base} and \ref{th:dec-ori}, and this directly yields a whole class of activity preserving (and active partition preserving) bijections between reorientation activity classes and bases, without any other change needed
(and from this one can also derive a whole class of refined
bijections between reorientations and subsets as in the next section, see also Remark~\ref{rk:refined-bij}).
We call this class of mappings $\psi$ the  \emph{active partition preserving mapping class%
} of $M$.
\EMEvder{nom maladroit, bof?}%
See more details in \cite[Section 4.4]{ABG2} or \cite{AB4}.\EMEvder{verifier cette reference- bien parler de ca and AB4}
\end{remark}



\section{Partitions of $2^E$ into basis intervals and reorientation activity classes, and Tutte polynomial in terms of four activity parameters for subsets and for reorientations}
\label{sec:partitions}


%



\EMEvder{faire sous-sections?}%

In this section, we address results involving partitions of the power set of the ground set into boolean lattices, first from basis intervals of an ordered matroid, second from reorientation activity classes of an ordered oriented matroid. We also address related Tutte polynomial expressions in terms of refined activity parameters,  first in terms of basis/subset activities, and second in terms of reorientation  activities.
Those partitions and expressions for bases/subsets and for reorientations are presented in this section independently of each other. 
They are presented in a suitable similar way so that they will be easily related to each other by the refined active bijection in the next Section \ref{sec:refined}:
the canonical active bijection of the previous section can be seen as a bijection between boolean lattices of the two types, and, then, the elements of these isomorphic boolean lattices are in bijection by the refined active bijection, preserving those refined activity parameters, and thus transforming the formulas of Theorems \ref{th:Tutte-4-variables} and \ref{th:expansion-reorientations} below into each other.%
\bigskip



\EMEvder{terminologie, basis intervals ? basis activity intervals ??? ici et dans intro et diagram}%

First, let $M$ be a matroid on a linearly ordered set $E$.
\ss

We 
recall the classical partition of the power set of the ground set into intervals (for inclusion), each associated with one basis with respect to basis activities, as discovered by Crapo in \cite{Cr67}
(let us mention distinct further generalizations of this result beyond matroids: by  Dawson \cite{Da81} to set families, by Gordon and McMahon \cite{GoMM97} to greedoids, by Las Vergnas \cite{LV13} to matroid perspectives).
\EMEvder{phrase faite detaillee sur ces refs, juste petite parenthese dans source avant}%
Then, basis activities can be extended to subsets in such a way that four subset activities indicate the position of a subset inside its interval with respect to the associated basis (among several possible ways of understanding these subset activities, this is the viewpoint we introduce for the sake of further constructions). This directly yields a Tutte polynomial formula in terms of these four refined subset activities, 
 as expressed by Las Vergnas \cite{LV13}, which is essentially a specification of a more general formula in terms of generalized activities originally  discovered by Gordon and Traldi \cite{GoTr90}.
Numerous Tutte polynomial formulas can be directly derived from this general four parameter formula by specifying variables, see \cite{GoTr90, LV13}. 
This formula is generalized to matroid perspectives in \cite{LV13} (completed in \cite{Gi18}). 
A summary about these notions can also be found in \cite{GiChapterPerspectives}. 
Here, we follow the notations used in~\cite{LV13}.
 We also give a very short proof of this formula, which highlights how it is the enumerative counterpart of the partition into intervals (as done in~\cite{Gi18}).

Let $B$ be a basis of $M$. The set of subsets of $E$ containing $B\setminus \Int(B)$ and contained in $B\cup \Ext(B)$ will be called the \emph{interval of} $B$, denoted $[B\setminus \Int(B), B\cup \Ext(B)]$.
%
These sets considered for all bases form a partition of $2^E$:
$$2^E=\biguplus_{B\hbox{ basis of }M}\Bigl[\ B\setminus \Int(B),\ B\cup \Ext(B)\ \Bigr].$$
Observe that the interval of $B$ has a boolean lattice structure and can also be denoted:
$$\bigl[\ B\setminus \Int(B),\ B\cup \Ext(B)\ \bigr]\ =\ \bigl\{\ B\triangle\bigl(P\cup Q\bigr)\ \mid \ P\subseteq \Int_M(B),\ Q\subseteq \Ext_M(B)\ \bigr\}.$$

\begin{definition}[\cite{LV13}]
\label{def:gene-act-base}
Let $M$ be a matroid on a linearly ordered set $E$. Let $B$ be a base of $M$. Let $A$ be in the boolean interval $[B\s \Int_{M}(B), B\cup \Ext_M(B)]$. We denote:
\begin{eqnarray*}
\Ext_M(A) & =& \Ext_M(B)\s A; \\
Q_M(A) & =& \Ext_M(B)\cap A; \\
\Int_{M}(A)& =& \Int_{M}(B)\cap A; \\
P_{M}(A) & =& \Int_{M}(B)\s A.
\end{eqnarray*}
\end{definition}


Let us mention that
these four parameters can be defined directly from $A$ without using $B$.
In particular, $Q_M(A)$, resp. $P_M(A)$, counts smallest elements of circuits, resp. cocircuits, contained in $A$, resp. $E\s A$.
In particular, $Q_M(A)$, resp. $P_M(A)$, counts smallest elements of circuits, resp. cocircuits, contained in $A$, resp. $E\s A$.
This yields
$| P_M(A)|=r(M)-r_M(A)$
and
$| Q_M(A)|=|A|-r_M(A)$
(which do not depend on the associated base).
%

\EMEvder{dessous e commntaire autres formulations de ces parametres sans utiliser $B$}%

\begin{thm}[{\cite{GoTr90, LV13}}]
\label{th:Tutte-4-variables}
Let $M$ be a matroid on a linearly ordered set $E$. We have
\begin{Large}
$$t(M;x+u,y+v)=\sum_{A\subseteq E}\ x^{\mid \Int_M(A)\mid}\ u^{\mid P_M(A)\mid}\ y^{\mid \Ext_M(A)\mid}\ v^{\mid Q_M(A)\mid}$$
\end{Large}
\end{thm}

\begin{proof}
By the expression \ref{eq:basis-activities}, we have:

$\displaystyle t(M;x+u,y+v)\ =\ \sum_{B \text{ basis}}\ (x+u)^{\mid \Int_M(B)\mid}\ (y+v)^{\mid \Ext_M(B)\mid}.$

\noindent By the binomial formula, this expression equals:

$\displaystyle  \sum_{B \text{ basis}}\ 
\Bigl(\ \sum_{A'\subseteq \Int_M(B)} x^{\mid A'\mid}u^{\mid \Int_M(B)\s A'\mid}\ \Bigr)\ 
\Bigl(\ \sum_{A''\subseteq \Ext_M(B)} y^{\mid \Ext_M(B)\s A''\mid}v^{\mid A''\mid}\ \Bigr).$

\noindent Since $\Int_M(B)\cap \Ext_M(B)=\emptyset$, one has a bijection between couples $(A',A'')$ involved in this expression and subsets $A=A'\uplus A''$ of $\Int_M(B)\uplus \Ext_M(B)$, hence this expression equals: 

$\displaystyle  \sum_{B \text{ basis}}\ 
\Bigl(\ \sum_{A \subseteq \Int_M(B)\uplus \Ext_M(B)} x^{\mid \Int_M(B)\cap A\mid}u^{\mid \Int_M(B)\s A\mid}\  y^{\mid \Ext_M(B)\s A\mid}v^{\mid \Ext_M(B)\cap A\mid}\ \Bigr).$


\noindent Since 
$\bigl(\ B\setminus \Int_{M}(B)\ \bigr)\cap\bigl(\ \Int_M(B)\uplus \Ext_M(B)\ \bigr)=\emptyset$, 
the mapping $A\mapsto A\cup\bigl(\ B\setminus \Int_{M}(B)\ \bigr)$ yields an isomorphism between 
the two boolean intervals $[\emptyset,\ \Int_M(B)\uplus \Ext_M(B)]$ and
$[B\s \Int_{M}(B), B\cup \Ext_M(B)]$, 
which does not change the sets $\Int_M(B)\cap A$, $\Int_M(B)\s A$, $\Ext_M(B)\s A$, and $\Ext_M(B)\cap A$. 
So the above expression can be equivalently written:



$\displaystyle  \sum_{B \text{ basis}}\ 
\Bigl(\ \sum_{A \in [B\s \Int_{M}(B), B\cup \Ext_M(B)]} x^{\mid \Int_M(B)\cap A\mid}u^{\mid \Int_M(B)\s A\mid}\  y^{\mid \Ext_M(B)\s A\mid}v^{\mid \Ext_M(B)\cap A\mid}\ \Bigr).$

\noindent Since $2^E=\cup_{B \text{ basis}} [B\s \Int_{M}(B), B\cup \Ext_M(B)]$, this expression equals:

$\displaystyle  \sum_{A\subseteq E}\ 
x^{\mid \Int_M(B)\cap A\mid}u^{\mid \Int_M(B)\s A\mid}\  y^{\mid \Ext_M(B)\s A\mid}v^{\mid \Ext_M(B)\cap A\mid}.$

\noindent Finally, by Definition \ref{def:gene-act-base}, this expression equals the required one.
%
\end{proof}

\eme{dessous en commenatire formualtion avec cardianux}%




%



\bigskip

Second, let $M$ be an oriented matroid on a linearly ordered set $E$.
\ss

We build on the partition of the power set of the ground set into activity classes of reorientations, introduced in Definition \ref{def:act-class}, and on their boolean lattice structure. 
We can naturally define four reorientation activity parameters that indicate the position of a reorientation inside its activity class. 
We obtain a short  proof of a simple expression of the Tutte polynomial using these four reorientation activity parameters (Theorem \ref{th:expansion-reorientations}).
Let us mention that this result generalizes to oriented matroid perspectives: it was proposed with a rather technical proof  in \cite{LV12}%
\EMEvder{footnote enlevee, dessous en commentaire, la meme que la suviante mais plus detaillee comme dans [Gi18]}%
, and it is shortly proved in terms of activity classes of oriented matroid perspectives in \cite{Gi18} in a similar way as in the present~paper.

Let us fix a reorientation $-_AM$ of $M$. 
The active partition of $-_A\M$ (Definition \ref{EG:def:ori-act-part}) 
can be denoted as:
$$E\ =\ \biguplus_{a\; \in\; O(-_A\M)\; \cup\; O^*(-_AM)}A_a$$
where $a=\min(A_a)$ for all $a\in O(-_A\M)\cup O^*(-_AM)$.
%
Then the  activity class $cl(-_AM)$ of $-_A\M$ (Definition \ref{def:act-class}) can be denoted the following way, highlighting its boolean lattice structure: 
$$cl(-_AM)\ =\ \Biggl\{ \ -_{A'}M\ \ \ \mid \ \ A'=A\ \triangle\ \Bigl(\ \biguplus_{a\in P\cup Q}A_a\ \Bigr)\ \hbox{ for } \ P\subseteq O^*(-_AM),\ \ Q\subseteq O(-_AM)\ \Biggr\}.$$
%
%
%
As addressed in Section \ref{sec:dec-seq},
activity classes of reorientations of $M$ form a partition of the set of reorientations of $M$:
$$2^E\ \sim\ \biguplus_{\text{\scriptsize one $-_AM$ chosen in each activity class}}
cl(-_AM).$$

\eme{dessous en commentaire DEF avec cardinaux en plus}

\begin{definition} 
\label{def:gene-act-ori}
Let $\M$ be an ordered oriented matroid. We define:
%
\begin{eqnarray*}
   \Theta_\M(A)&=&O(-_A\M)\s A,  \\
 \bar\Theta_\M(A)&=&O(-_A\M)\cap A, \\ 
   \Theta^*_\M(A)&=&O^*(-_A\M)\s A, \\
\bar\Theta^*_\M(A)&=&O^*(-_A\M)\cap A. 
\end{eqnarray*}
\eme{dessous avec complementaires en plus}
%
%
%
Hence we have $O(-_A\M)=\Theta_\M(A)\uplus\bar\Theta_\M(A)$ 
and (dually) $O^*(-_A\M)=\Theta^*_\M(A)\uplus\bar\Theta^*_\M(A)$.
\end{definition}


%

In contrast with the definition of the activity and dual activity of a reorientation of $M$, that depends only on the resulting oriented matroid $-_AM$, the definition above depends on $A$ and $M$.
By this way, it refines  reorientation activities into four parameters that 
substantially apply to reorientations of a given \emph{reference oriented matroid $\M$}. 
These parameters actually situate any reorientation in its activity class  (which is independent of the reference oriented matroid, see Observation \ref{obs:canonical}). 

Precisely,
consider an activity class of reorientations of $M$, and the representative $-_AM$ of this class  which is active-fixed and dual-active fixed with respect to the reference oriented matroid $M$ (Corollary \ref{cor:enum-classes}).
By definition,  it satisfies: 
 \begin{center}
   $
   \begin{array}{lclcl}
   \bar\Theta_\M(A)&=&O(-_A\M)\cap A&=&\emptyset,\\[1mm]
   \bar\Theta^*_\M(A)&=&O^*(-_A\M)\cap A&=&\emptyset.
   \end{array}
   $
\end{center}
 Furthermore, the other reorientations $-_{A'}M$ in the same activity class correspond to other possible values of $\bar\Theta_M(A')\subseteq O(-_AM)$ and $\bar\Theta^*_M(A')\subseteq O^*(-_AM)$.
 \EMEvder{enlever $\subseteq$ dans phrase precedente ? (redit dessous)}%
Using the above notation for the activity class $cl(-_AM)$ as a boolean lattice, we have  
\EMEvder{mettre ces P= et Q= en tableau comme ci-dessus pour ler etpreentative ?}%
 \begin{center}
   $
   \begin{array}{lclcl}
   Q&=&\bar\Theta_M(A')&\subseteq& O(-_AM),\\[1mm]
   P&=&\bar\Theta^*_M(A')&\subseteq& O^*(-_AM).
   \end{array}
   $
\end{center}
A way of understanding the role of the reference oriented matroid $M$ is that it breaks the symmetry in each activity class, so that its boolean lattice structure can be expressed relatively to the aforemetioned representative.
This representative is noticeably used in Section \ref{sec:refined}.
Other choices are possible for a representative.

%

%


Finally, we derive the following Tutte polynomial expansion formula in terms of these four parameters.
A (technical) proof for Theorem \ref{th:expansion-reorientations} below is proposed in 
the preprint \cite{LV12}
by deletion/contraction in the more general setting of oriented matroid perspectives.
%
This theorem can also be directly proved by means of the above construction on activity classes (as announced
 in \cite{LV12}\footnote{See footnote \ref{footnote:refined} in Section \ref{sec:refined} for a correction on this announce as written in the preprint \cite{LV12}. 
},  this theorem can also be seen as a direct corollary of the similar formula for subset activities from Theorem  \ref{th:Tutte-4-variables} and the refined active bijection from Theorem \ref{th:ext-act-bij}).
We give this short proof below for completeness of the paper, though it is a translation of the proof given in \cite{Gi18} for oriented matroid perspectives.%

\begin{thm} 
\label{th:expansion-reorientations}
Let $M$ be an oriented matroid on a linearly ordered set $E$. We have
%
%
\eme{dessous en commanriter enonce vec cardianux}%
\begin{Large}
$$t(M;x+u,y+v)=\sum_{A\subseteq E} \ x^{|\Theta^*_\M(A)|}\ u^{|\bar\Theta^*_\M(A)|}\ y^{|\Theta_\M(A)|}\ v^{|\bar\Theta_\M(A)|}.$$
\end{Large}
\end{thm}

\begin{proof}
The proof is obtained by a simple combinatorial transformation.
Let us start with  the right-hand side of the equality, where we denote $\theta^*_M(A)$ instead of $|\Theta^*_M(A)|$, etc., by setting:
\vspace{-1mm}
$$\displaystyle [Exp] = \sum_{A\subseteq E} x^{\theta^*_M(A)}u^{\bar\theta^*_M(A)}y^{\theta_M(A)}v^{\bar\theta_M(A)}.$$

\vspace{-2mm}
Since $2^E$ is isomorphic to the set of reorientations, which is partitioned into activity classes of reorientations of $M$ (Definition \ref{def:act-class}), and by choosing a representative for each activity class which is active-fixed an dual-active-fixed (as discussed above), we get:
\vspace{-1mm}
$$[Exp] =  \sum_{\substack{\text{activity classes $cl(-_AM)$ of reorientations of }M\\ \text{with one $-_AM$ chosen in each class}\\ \text{such that $O(-_AM)\cap A=\emptyset$ and $O^*(-_AM)\cap A=\emptyset$}}}\ \ \sum_{-_{A'}M\ \in\  cl(-_AM)} x^{\theta^*_M(A')}u^{\bar\theta^*_M(A')}y^{\theta_M(A')}v^{\bar\theta_M(A')}$$

\vspace{-2mm}
As discussed above, when $-_{A'}M$ ranges the activity class of $-_AM$, 
$\bar\Theta_M(A')$ and $\bar\Theta^*_M(A')$  range subsets of 
$O(-_AM)$ and  $O^*(-_AM)$,  respectively. So, we get the following expression (where ``idem'' refers to the text below the first above sum), which we then transform using the binomial formula:%
\begin{eqnarray*}
[Exp] & = & \sum_{\text{idem}}\ \ \sum_{\substack{P\subseteq O^*(-_AM)\\ Q\subseteq O(-_AM)}} x^{\mid O^*(-_AM)\s P\mid}u^{\mid P\mid}y^{\mid O(-_AM)\s Q\mid}v^{\mid Q\mid}\\
&=&  \sum_{\text{idem}} \ \
\biggl(\,\sum_{P\subseteq O^*(-_AM)} 
x^{\mid O^*(-_AM)\s P\mid}
u^{\mid P\mid}\biggr)
\biggl(\,\sum_{Q\subseteq O(-_AM)} 
y^{\mid O(-_AM)\s Q\mid}
v^{\mid Q\mid}\biggr)
\\
&=&  \sum_{\text{idem}}\ \ (x+u)^{|O^*(-_AM)|} (y+v)^{|O(-_AM)|}
\end{eqnarray*}
%

\vspace{-2mm}
Since the activity class of $-_AM$ has ${2^{|O(-_AM)|+|O^*(-_AM)|}}$ elements with the same orientation activities,
we have (denoting for short $o(A)=|O(-_AM)|$ and $o^*(A)=|O(-_AM)|$):
\begin{eqnarray*}
%
[Exp] & = & 
%
\sum_{\text{idem}}\ \ {1\over{2^{o(A)+o^*(A)}}}\ \sum_{-_{A'}M\ \in\  cl(-_AM)}\ (x+u)^{o^*(A')}(y+v)^{o(A')}\\
&=& \sum_{\text{idem}}\ \ \sum_{-_{A'}M\ \in\  cl(-_AM)}\ {\Bigl({x+u\over 2}\Bigr)}^{o^*(A')}{\Bigl({y+v\over 2}\Bigr)}^{o(A')}\\
&=&  \sum_{A\subseteq E}\ {\Bigl({x+u\over 2}\Bigr)}^{o^*(A)}{\Bigl({y+v\over 2}\Bigr)}^{o(A)}\\
&=&  t(G;x+u,y+v)
\end{eqnarray*}

\vspace{-2mm}

\noindent using at the end the \ref{eq:reorientation-activities} from \cite{LV84a} recalled 
in Section \ref{sec:prelim}.
\end{proof}

\rm
Obsere that Theorem \ref{th:expansion-reorientations} provides a proof of the enumerations of activity classes and their representatives from Corollary \ref{cor:enum-classes} and Table \ref{table:enum-classes}.
Finally, let us mention that numerous Tutte polynomial formulas can be directly obtained from Theorem \ref{th:expansion-reorientations}, for instance by replacing variables ($x$, $u$, $y$, $v$) with $(x/2,x/2,y/2,y/2)$, 
or $(x+1,-1,y+1,-1)$, 
or $(2,0,0,0)$, etc., as well as expressions for derivatives of the Tutte polynomial.
%
%
%
These formulas are given in \cite{LV12, Gi18} (see also \cite{GiChapterOriented}, and see \cite{LV12} for a detailed example).
\EMEvder{mettre 'see detailed example'?}


\EMEvder{ai enleve corollaire hors sujet}%

\section{The refined active bijection between reorientations and subsets}

\label{sec:refined}

\EMEvder{texte dessous avant df /thm repris de ABG2, un peu copier/coller mais bien ecrit... a voir si OK de faire ca}%
The present construction is a 
natural development of the canonical active bijection (sketchily introduced in \cite{Gi02,GiLV06,GiLV07}).
\EMEvder{mettre cette parentehse}%
Let us consider an ordered oriented matroid $M$ and its active basis $B=\alpha(M)$.  On one hand, the activity class of $M$ (Definition \ref{def:act-class} and Section \ref{sec:partitions}) obviously has a boolean lattice structure isomorphic to the power set of $O(M)\cup O^*(M)$.
On the other hand, the interval $[B\setminus \Int(B), B\cup \Ext(B)]$ of $B$ (Section \ref{sec:partitions})
also has a boolean lattice structure isomorphic to the power set of $\Int(B)\cup \Ext(B)$.
Since we have $\Int(B)\cup \Ext(B)=O(M)\cup O^*(M)$ by properties of $\alpha(M)$ (Theorem \ref{th:alpha}), those two boolean lattices are isomorphic.
Furthermore, activity classes of reorientations of $M$ form a partition of the set of reorientations of $M$ (Definition \ref{def:act-class}), intervals of bases form a partition of the power set of $E$ (Section \ref{sec:partitions}), and activity classes of orientations are in bijection with bases 
under $M\mapsto \alpha(M)$
(Theorem \ref{th:alpha}).
Hence, selecting a boolean lattice isomorphism for each couple formed by an activity class and its active basis directly yields a bijection between all reorientations and all subsets of $E$, which refines the canonical active bijection of $M$, and transforms activity classes of reorientations into intervals of bases. 
The most natural way   to select such isomorphisms (see also Remark \ref{rk:refined-bij} for variants) is to use the oriented matroid $M$ as a \emph{reference}, whose role is to ``break the symmetry'' in activity classes, just as in Section \ref{sec:partitions}. 
See Figure \ref{fig:K4-iso} for an illustration.
\emevder{la figue illustre aussi le choix par rapport a orientation de reference, a preciser plus loin? dans la caption?}%
By this way, we shall obtain below \emph{the refined active bijection $\alpha_M$ of $M$}, which relates\emevder{preserves ?} the refined activities 
for reorientations  and for subsets from Definitions \ref{def:gene-act-base} and \ref{def:gene-act-ori}
\emevder{renvoyer a (see Section \ref{subsec:act-map-class-decomp}) pour varaintes ?}%
(as announced in \cite{LV12}%
\footnote{\label{footnote:refined}
Beware that the definition for the refined active bijection proposed at the very end of the unpublished preprint \cite{LV12} in terms of the active bijection is not correct: it is not complete, and given with a wrong parameter correspondence. It is different from the present one, which is consistent with the one given in \cite{Gi02,GiLV06,GiLV07}.
}%
),
giving a bijective transformation between the formulas of Theorems \ref{th:Tutte-4-variables} and \ref{th:expansion-reorientations}:
\begin{eqnarray*}
T(M;x+u,y+v)&
=&\sum_{A\subseteq E}\ x^{\mid \Int_M(A)\mid}\ u^{\mid P_M(A)\mid}\ y^{\mid \Ext_M(A)\mid}\ v^{\mid Q_M(A)\mid}\\
&=&\sum_{A\subseteq E} \ x^{|\Theta^*_M(A)|}\ u^{|\bar\Theta^*_M(A)|}\ y^{|\Theta_M(A)|}\ v^{|\bar\Theta_M(A)|}.
\end{eqnarray*}

Let us insist that, in contrast with the canonical active bijection,  which depends only on the reorientation class since the active basis is intrinsically defined for an ordered oriented matroid (see Observation \ref{obs:canonical}), the refined active bijection  depends on (or is induced by the choice of) a given reference  oriented matroid (that is, a reference signature in terms of a topological representation).
\EMEvder{a dire mieux ?}%


\begin{figure}[h]
\centering
\includegraphics[scale=0.9]{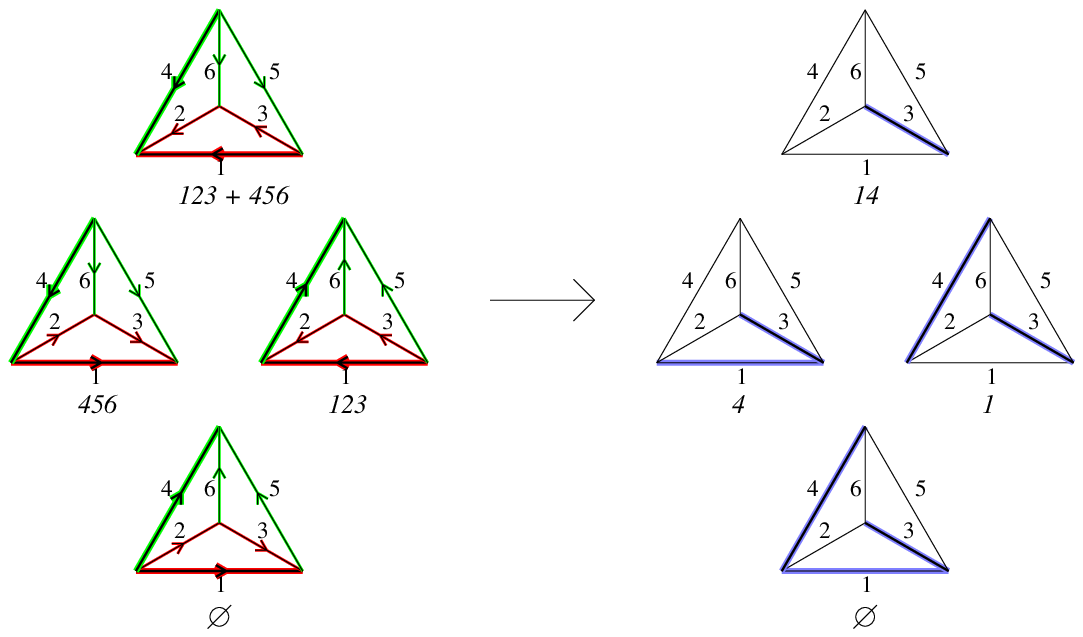}
\hfill 
\includegraphics[scale=1.3]{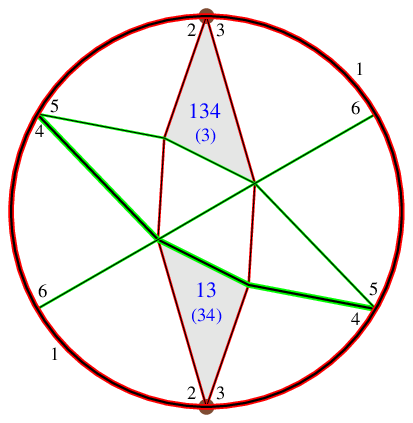}
\caption[]{Boolean lattice isomorphism between an activity class of reorientations and the interval of the corresponding basis, figured for the activity class from Figure \ref{fig:K4-dec} with active partition $123+456$ and active basis $B=134$ 
with $[B\s \Int(B),B]=[3,134]$. The layout reflects the bijection. Edges written below the graphs in the middle are those removed from $B$, they correspond to reoriented parts in the digraphs on the left and on the regions on the right (brackets refer to the opposite regions in the opposite half of the arrangement).
The reference oriented matroid $M$ is given by any signature/orientation such that the reorientation associated with $B$ is active-fixed and dual-active fixed w.r.t. $M$ (e.g., simply the region/digraph associated with $B$).
\EMEvder{derniere phrase bien dite ?}%
}
\label{fig:K4-iso}
\end{figure}

Technically, let us take up the notations and discussion of Section \ref{sec:partitions}.
Let $M$ be an oriented matroid on a linearly ordered set $E$, thought of as  \emph{the reference oriented matroid}.
Let $A\subseteq E$.
%
The active partition of $-_AM$ can be denoted as:
$$E\ =\ \biguplus_{a\ \in\ O(-_AM)\ \cup\ O^*(-_AM)}A_a$$
where the index of each part is the smallest element of the part.
%
The activity class of $-_AM$ is:
$$cl(-_AM)=\Biggl\{\ -_{A'}M \ \mid\ A'= A\ \triangle\ \Bigl(\ \biguplus_{a\in P\cup Q}A_a\Bigr)\ \text{ for } \ P\subseteq O^*(-_AM),\ Q\subseteq O(-_AM)\ \Biggr\}.$$
Let $B=\alpha(-_AM)$ be the active basis of $-_AM$. 
%
The interval 
of $B$ 
can be also denoted:
$$[B\setminus \Int(B), B\cup \Ext(B)]\ =\ \Biggl\{\ B'\subseteq E\ \mid\  B'=B\triangle\bigl(\biguplus_{a\in P\cup Q}\{a\}\bigr) \text{ for  }P\subseteq \Int(B),\ Q\subseteq \Ext(B)\ \Biggr\}.$$
%
The above notations emphasize the two boolean lattice structures. 
Then, we define an isomorphism between the two by choosing that the representative of the activity class which is active-fixed and dual-active fixed w.r.t. $M$
is associated with the basis~$B$.
Assume $-_{A}M$ is the representative of its class with these properties, then we formally have:
$\bar\Theta_M(A)=O(-_AM)\cap A=\emptyset$ and $\bar\Theta^*_M(A)=O^*(-_AM)\cap A=\emptyset,$ which corresponds to $A'=A$, $P=\emptyset$ and $Q=\emptyset$ in the above setting, and which corresponds to the subset $B'=B$ in the interval of the basis $B$, that is to $P_M(B')=\Int(B)\cap B'=\emptyset$ and $Q_M(B')=\Ext(B)\cap B'=\emptyset$ (Definitions \ref
{def:gene-act-ori} and \ref{def:gene-act-base}).
Finally, all reorientations in the same activity class and all subsets in the same associated basis interval correspond to all possible values of $P$ and $Q$ in the above notations, so that:
\begin{center}
   $
   \begin{array}{lclclclcl}
    P&=&\bar\Theta^*_M(A')&=&P_M(B')&\subseteq& O^*(-_AM)&=&\Int(B), \\
    Q&=&\bar\Theta_M(A')&=&Q_M(B')&\subseteq& O(-_AM)&=&\Ext(B).
   \end{array}
   $
\end{center}
%
By this way, we naturally obtain the following definition and theorem.

\begin{definition}
\label{def:act-bij-ext}
Let $\M$ be an oriented matroid on a linearly ordered set $E$.
For  $A\subseteq E$, 
we define
%
 $$\alpha_\M(A)=\alpha(-_A\M)\ \setminus\ \Bigl(A\cap O^*(-_A\M)\Bigr)\ \cup\ \Bigl(A\cap O(-_A\M)\Bigr).$$
That is: $\alpha_\M(A)=B \setminus P \cup Q$ with $B=\alpha(-_A\M),$
$P=A\cap \Int(B)=A\cap O^*(-_A\M),$ and $Q=A\cap \Ext(B)=A\cap O(-_A\M).$

The mapping $A\mapsto \alpha_M(A)$ from $2^E$ to $2^E$ is called \emph{the refined active bijection of $\M$.}
\EMEvder{ou: refined active bijection W.R.T. $M$? a verifier avnt j'avais mis wrt puis chagne en of}%
%
%
\end{definition}

\begin{thm} 
\label{th:ext-act-bij}
Let $M$ be an oriented matroid on a linearly ordered set $E$.
%
We have the following.
\begin{itemize}
\item The mapping $A\mapsto \alpha_M(A)$ from $2^E$ to $2^E$ is a bijection.
It yields a bijection between reorientations $-_AM$ of $\M$ and subsets of $E$, which maps activity classes of reorientations of $\M$ onto intervals of bases of $M$ (and these restrictions are boolean lattice isomorphisms).
%
\item  For all $A\subseteq E$, 
with $B=\alpha(-_A\M)$ and $\alpha_\M(A)=B\setminus P\cup Q$,
we have:
\begin{align*}
\Int_M(\alpha_\M(A))&&=&&\scriptstyle \Int_M(B)\s P&&\scriptstyle =&&\scriptstyle O^*(-_A\M)\s P&&=&&\Theta^*_\M(A),\\
P_M(\alpha_\M(A))&&=&&&&\scriptstyle P&&&&=&&\bar\Theta^*_\M(A),\\
\Ext_M(\alpha_\M(A))&&=&&\scriptstyle \Ext_M(B)\s Q&&\scriptstyle =&&\scriptstyle O(-_A\M)\s Q&&=&&\Theta_\M(A),\\
Q_M(\alpha_\M(A))&&=&&&&\scriptstyle Q&&&&=&&\bar\Theta_\M(A).\\
\end{align*}
\vspace{-1.5cm}
%
%
%

%
%
%
%
%
%
\item  In particular, $\alpha_M(A)$ equals the active basis $\alpha(-_AM)$ if and only if $-_AM$ is active fixed and dual-active fixed w.r.t. $M$. Similarly, restrictions of the mapping $\alpha_M$ yield the  bijections listed in Table \ref{table:thm-refined}.


%
\end{itemize}
\end{thm}

\begin{table}[h]
\begin{center}
\def\interligne{&&\\[-11pt]}
\parindent=-1.5cm
\begin{tabular}{|l|l|c|}
\hline
\interligne
reorientations & subsets & $t(M;2,2)$\\
\interligne
acyclic reorientations & subsets of internal bases & $t(M;2,0)$\\
&  (or no-broken-circuit subsets)& \\
\interligne
 totally cyclic reorientations & supersets of external bases & $t(M;0,2)$\\
 \interligne
dual-active-fixed acyclic reorientations & internal bases & $t(M;1,0)$\\
 \interligne
\interligne
active-fixed totally cyclic reorientations & external bases & $t(M;0,1)$\\
\interligne
\interligne
active-fixed reorientations & subsets of bases & $t(M;2,1)$\\
& (or independents)&\\
\interligne
dual-active-fixed reorientations & supersets of bases  & $t(M;1,2)$\\
& (or spanning subsets) \eme{??????????}&\\
\interligne
active-fixed and dual-active-fixed reorientations & bases & $t(M;1,1)$\\
\hline
\end{tabular}
\end{center}
\vspace{-5mm}
\caption{Remarkable restrictions of the refined active bijection of   $M$, between particular types of reorientations (first column) and  particular types of edge subsets (second column) enumerated by Tutte polynomial evaluations (third column). See Theorem \ref{th:ext-act-bij}.
\EMEvder{je n'ai pas mis les bijectins avec min des itnervalles, fait dans ABG2, pas ici, a ajouter?}%
}
\label{table:thm-refined}
\end{table}

\eme{POUR PLUS TARD : voir si \min subsetes of itnernal sp. tree intervals peuvent etre decrits combinatoriement autrement ! probable vu que P ne depend pas de arbre...}%


\begin{proof}
The first point comes directly from Definition \ref{def:act-bij-ext} and
the above discussion.
The second point also easily comes from this discussion. 
Let us precisely check the equalities of parameters in the second point anyway.
In order to simplify notations, we omit subscripts $M$ of activity parameters.
Let $A_B$ be the reorientation of $\M$ whose image under $\alpha_\M$ is the base $B$. 
Let $E=\uplus_{a\in O(-_{A_B}\M)\cup O^*(-_{A_B}\M)} A_a$, with $a=\min(A_a)$, be the active partition associated with $B$ or $-_{A_B}\M$.
Let $A$ be a subset in the associated activity class, we have
$A=A_B\triangle\bigl(\cup_{a\in P\cup Q}A_a\bigr)$ for some $P\subseteq \Int(B)=O^*(-_{A_B}\M)=O^*(-_A\M)$ and $Q\subseteq \Ext(B)=O(-_{A_B}\M)=O(-_A\M)$ with $P\cap A_B=\emptyset$ and $Q\cap A_B= \emptyset$.
By Definition \ref{def:act-bij-ext}, we have $\alpha_\M(A)=B\setminus P\cup Q$ .
%

By Definition \ref{def:gene-act-base}, we have
$\Int(\alpha_\M(A))=\Int(B)\cap \alpha_\M(A)$.
We have $\Int(B)\cap \alpha_\M(A)=\Int(B)\cap(B\s P\cup Q)=\Int(B)\s P$.
By Theorem \ref{th:alpha}, we have $\Int(B)\s P=O^*(-_A\M)\s P$.
By properties of $P$, we have
$O^*(-_A\M)\s P=O^*(-_A\M)\s \bigl(A_B\triangle (\cup_{a\in P\cup Q}A_a)\bigr)=O^*(-_A\M)\s A$.
By Definition \ref{def:gene-act-ori}, we have
$O^*(-_A\M)\s A=\Theta^*(A)$.
So finally $\Int_M(\alpha_\M(A))=\Theta^*(A)$.

On one hand, by Definition \ref{def:gene-act-base}, we have $\Int(\alpha_\M(A))\cup P(\alpha_\M(A))=\Int(B)$.
On the other hand, by Definition \ref{def:gene-act-ori}, we have $\Theta^*(A)\cup \bar\Theta^*(A)=O^*(-_A\M)$. By Theorem \ref{th:alpha}, we have $\Int(B)=O^*(-_A\M)$, so, by the above result, we get $P(\alpha_\M(A))=\bar\Theta^*(A)$.

Similarly, by Definition \ref{def:gene-act-base}, we have
$\Ext(\alpha_\M(A))=\Ext(B)\s \alpha_\M(A)$.
We have $\Ext(B)\s \alpha_\M(A)=\Ext(B)\s(B\s P\cup Q)=\Ext(B)\s Q$.
By Theorem \ref{th:alpha}, we have $\Ext(B)\s Q=O(-_A\M)\s Q$.
As above, 
by properties of $Q$, we have 
$O(-_A\M)\s Q=O(-_A\M)\s \bigl(A_B\triangle (\cup_{a\in P\cup Q}A_a)\bigr)=O(-_A\M)\s A$.
As above, by Definition \ref{def:gene-act-ori}, we have
$O(-_A\M)\s A=\Theta(A)$.
So finally $\Ext(\alpha_\M(A))=\Theta(A)$.
And, as above, we deduce that $Q(\alpha_\M(A))=\bar\Theta(A)$.


Now, let us consider the list of bijections of the third point.
They are all obtained as restrictions of $\alpha_\M$.
Observe that a reorientation is active-fixed, resp. dual-active-fixed, if it is obtained by $Q=\emptyset$, resp. $P=\emptyset$.
Therefore, all these bijections are obvious by the definitions, except the two ones involving $t(M;1,2)$ and $t(M;2,1)$. 
For the first one, resp. second one, of these two, we can use that subsets, resp. supersets, of bases are exactly the subsets of type $B\s P$, resp. $B\cup Q$, for some base $B$ and $P\subseteq \Int(B)$, resp. $Q\subseteq \Ext(B)$.
This result is stated separately in Lemma \ref{lem:decomp-intervals} below.
%
%
%
%
\end{proof}

\begin{lemma}
\label{lem:decomp-intervals}
Let $M$ be an ordered matroid. 
The set of subsets of bases of $M$ (i.e. independents) is the union of intervals $[B\s \Int_M(B),B]$ over all bases $B$ of $M$.
The set of supersets of bases of $M$ (i.e. spanning subsets) is the union of intervals $[B, B\cup \Ext_M(B)]$ over all bases $B$ of $M$.
\end{lemma}

\EMEvder{preuve ci-dessous en citan AB2-a au lieu de EtLV98 et KRS99, ok ? ou bien citer encore ces papiers ? (mais pas Lass allemand)}

\begin{proof}
It is known that
bases $B$ of $M$ are exactly subsets of the form $B_\io\uplus B_\ep$
where $B_\io$ is an internal base of $M/F$, 
$B_\ep$ is an external base of $M(F)$, 
and $F$ is a cyclic flat of $M$
(see details and references in \cite[Corollary \ref{a-th:EtLV98}]{AB2-a}). 
Moreover $\Int(B)=\Int_{M/F}(B_\io)$ and 
$\Ext(B)=\Ext_{M(F)}(B_\ep)$ (for short, we omit these subscripts below).

We have $[B\s \Int(B),B\cup \Ext(B)]=[(B_\io\uplus B_\ep)\s \Int(B_\io),(B_\io\uplus B_\ep)\cup \Ext(B_\ep)]$.
Using the classical partition of $2^E$ into basis intervals recalled 
Section \ref{sec:partitions}
we have:
$$2^E=\biguplus_{B\text{ base}}[B\s \Int(B),B\cup \Ext(B)]
=\biguplus_{F,\ B_\io,\ B_\ep\text{ as above}}
[B_\io\s \Int(B_\io),B_\io]\times [B_\ep, B_\ep\cup \Ext(B_\ep)]$$
(where $\times$ yields all unions of a subset of the first set and a subset of the second set).
So we have $$\biguplus_{B\text{ base}}[B\s \Int(B),B]
=\biguplus_{F,\ B_\io,\ B_\ep\text{ as above}}
[B_\io\s \Int(B_\io),B_\io]\times [B_\ep]$$
The size of the second set of the equality equals 
$\sum_{F}t(M/F;2,0)t(M(F);0,1)$ by classical evaluations of the Tutte polynomial.
And this number is known to be equal to $t(M;2,1)$ 
(convolution formula for the Tutte polynomial, see details and references in \cite[Corollary \ref{a-cor:convolution}]{AB2-a}), which equals the number of subsets of bases (as well known).
The first set of the equality is included in the set of subsets of bases, and it has the same size,
hence it equals the set of subsets of bases.
Dually, we get the result involving supersets of bases, whose number equals $t(M;2,1)$.
\eme{a verifier ! pas si evident...}%
%
%
%
%
%
%
%
%
\end{proof}
%
%
%
%

%

\eme{rk suivante inutile ? (verifiable dans preuve)
The reader might be surprised by the symmetry of the correspondence between parameters in Theorem \ref{th:ext-act-bij}, compared with the non-symmetry of the definitions:
$\Int_{M}(A) = \Int_{M}(B)\cap A$ and $\Ext_M(A) = \Ext_M(B)\s A$ on one hand, 
and $\Theta^*_\M(A)=O^*(-_A\M)\s A$ and $\Theta_\M(A)=O(-_A\M)\s A$ on the other hand.
The reason is that we have chosen to associate a base with the active-fixed and dual-active-fixed reorientation in its associated activity class, so in this case: an internally active element belongs to the base and is not reoriented on one hand, and an externally active element does not belong to the base and is also not reoriented on the other hand.}%

%
%


\eme{th dessus provue sur feuilles volantes}%

\eme{peut etr alleger notations dans theoreme en definissant $B'$ et $A'$ comme dans intro de section}%

\eme{denote $\alpha_\M$ or denote $\bar\alpha_\M$}%

\eme{ATTENTION bien verifier tout ca, surtout derniere phrase, a ete vite fait !}%


\eme{DESSOUS EN COMMENTAIRE premiere versio plusn litteraire de refined bijection}%

Now, let us give two results for building the inverse of the refined active bijection, from subsets to reorientations. They are directy obtained from the inverse constructions of the canonical active bijection.
The first specifies Proposition \ref{prop:preimage-basis}.
The second is an immediate adaptation of the single pass algorithm of Theorem \ref{th:basori}.

%

%

\EMEvder{prop ci-dessous a ete ajoutee vite fait dans ABG2 avant soumission, et transposee ici vite fait, A VERIFIER !!! piegeuse en notations...}%

\begin{prop}[refined active bijection from subsets]
\label{prop:preimage-subset}
Let $M$ be an ordered oriented matroid on $E$.
Let 
$A$ be a subset 
in the interval of a basis of $M$ with active filtration  $\emptyset= F'_\ep\subset...\subset F'_0=F_c=F_0\subset...\subset F_\io= E$.
Then,
%
$$
\displaystyle\alpha_M^{-1}(A)\ =\ 
\biguplus_{1\leq k\leq\io}
\alpha_{M(F_k)/F_{k-1}}^{-1}(A\cap (F_{k}\s F_{k-1}))\
\uplus \ 
\biguplus_{1\leq k\leq\ep}
\alpha_{M(F'_{k-1})/F'_{k}}^{-1}(A\cap (F'_{k-1}\s F_{k})).
   $$
\end{prop}

\emevder{A BIEN REVERIFIER !!!! (je l'ai ecrit tres vite fait, ainsi que la preuve)}

\emevder{on pourrait aussi donner une formula directe pour refined de ce type du coup je pense !!! en terms d'orietnations et d'ative filtation d'orietnation...}

\begin{proof}
This  is a straightforward reformulation, in terms of Proposition \ref{prop:preimage-basis}, of the construction of the refined active bijection discussed above. Let us give details anyway.
Consider any of the  active minors $H$, and the uniactive basis $B_H$ induced in the minor $H$ by the basis $B$ associated with $A$. 
The inverse image of $B_H$ under $\alpha$ in $H$  consists of two opposite reorientations of $H$.
Now consider the refined active bijection of $H$, and denote $a$ the smallest edge of $H$.
One of the two above reorientations is associated to $B_H$ (the one for which $a$ has not been reoriented, that is, $a$ is active-fixed or dual-active-fixed w.r.t. $M$),  and the other to $B_H\triangle \{a\}$.
Applying this to each minor $H$ and to any subset $A$ in the same interval, we always obtain a reorientation of $M$ whose image under  $\alpha_M$ is $A$.\emevder{reorientation mal dit, c'est un subset}
\end{proof}


\begin{thm}[completing Theorem \ref{th:basori}]
\label{th:basori-refined}
Let $M$ be an oriented matroid on a linearly ordered set of elements $E=e_1<\ldots<e_n$.
Let $X$ be a subset of $E$. 
We denote $Q=Q(X)$ and $P=P(X)$ (Definition \ref{def:gene-act-base}).
We denote $B$ the basis of $M$ defined by
$B=X\setminus Q\cup P$ 
(equivalently: $B$ is the basis such that $X$ belongs to the interval of $B$, that is: $X=B\setminus P\cup Q$ with $B$ a basis, $P\subseteq \Int(B)$ and $Q\subseteq \Ext(B)$).

The preimage of the subset $X$ under $\alpha_M$ is built by 
applying the algorithm of Theorem \ref{th:basori}, as for building the preimage of $B$ under $\alpha$, with the two following changes.



\noindent\begin{tabular}{ll}
In the case where $e_k\in B$ and $e_k\in \Int(B)$, 
& replace
{\algofont
``reorient $e_k$ or not, arbitrarily''
}
\\
&
with:
{\algofont
``reorient $e_k$ if and only if $e_k\in P$.''
}\\
In the case where  $e_k\not\in B$ and $e_k\in \Ext(B)$, 
& replace
{\algofont
``reorient $e_k$ or not, arbitrarily''
}
\\
&
with:
{\algofont
``reorient $e_k$ if and only if $e_k\in Q$.''
}
\\
\end{tabular}
\end{thm}

\begin{proof}
Let us denote $A=\alpha^{-1}_M(X)$.
The computation of the reorientation class of $A$ is given by the algorithm of Theorem \ref{th:basori} applied to the basis $B$.
Furthermore, by Theorem \ref{th:ext-act-bij}, we have $P(X)=\bar\Theta^*_M(A)=O^*(-_AM)\cap A$ and
$Q(X)=\bar\Theta_M(A)=O(-_AM)\cap A$.
So, for $e_k\in \Int(B)=O^*(-_AM)$, we have $e_k\in A$ if and only if $e_k\in P$,
and for $e_k\in \Ext(B)=O(-_AM)$, we have $e_k\in A$ if and only if $e_k\in Q$.
This is exactly the condition stated in the algorithm.
\end{proof}

Before ending this section, let us mention that a deletion/contraction construction exists for $\alpha_M$, that is also derived directly from that of $\alpha$, see \cite{ABG2,AB4}.
And let us end with a general remark on possible variations in the construction of this section.

\begin{remark}[Variations of the refined active between reorientations and subsets]
\label{rk:refined-bij}
\rm

Let us observe that variants of $\alpha_\M$ can easily be defined, again using a boolean lattice isomorphism at each activity class / basis interval. For instance, in Definition \ref{def:act-bij-ext}, replace $A$ with $X_B\triangle A$ for some $X_B\subseteq E$ that can vary with $B$ (i.e. the boolean lattice isomorphism can change at each considered boolean lattice).
This yields other reorientations-subsets bijections refining the canonical active bijection.
 
By this way, for instance, one can define active-fixed and dual-active-fixed reorientations with respect to two different references reorientations respectively. Also, suitable choices of $X_B$ allow us to exchange the correspondences between the four parameter activities for bases and reorientations (i.e. make $\Int$ correspond to $\bar\Theta^*$ instead of $\Theta^*$,
and/or make $\Ext$ correspond to $\bar\Theta$ instead of $\Theta$). 

Moreover, as reorientation activity classes correspond to basis intervals, one can derive various bijections by composition with a further boolean lattice isomorphism.
For instance, one can define a bijection between dual-active-fixed acyclic reorientations and minimal subsets of internal basis intervals, enumerated by $t(M;1,0)$, or, dually, 
a bijection between active-fixed totally cyclic reorientations and maximal subsets of external basis intervals, enumerated by $t(M;1,0)$, et caetera.

At last, let us recall (see Remark \ref{rk:preserv-act-bij-class}) that a general class of active partition preserving bijections can be obtained by replacing $\alpha$ with any mapping $\psi$ defined for bounded/dual-bounded reorientations and yielding a bijection with uniactive internal/external bases. The same construction as above can be applied to such a mapping $\psi$, yielding a whole class of bijections $\psi_M$ between reorientations and subsets, preserving the four parameter activities for reorientations/subsets.
%
\EMEvder{verifier cette reference pas mise: See more details in \cite[Section 4.4]{ABG2} or \cite{AB4}.}%
\end{remark}

\section{Further examples and illustrations}
\label{sec:example}



%


In this section, we complete the paper with a few more illustrative examples.

\subsection{Example of $K_3$}

\EMEvder{footnote en commentaire ci-dessous sur Tutte et $K_3$, mise dans ABG2, pas ici, a mettre ? bof non}%
%

The canonical and refined bijections are shown in Table \ref{table:tabK3ori} and Figure~\ref{fig:K3}.
The Tutte polynomial of $K_3$ is $$t(K_3;x,y)=x^2+x+y.$$

\def\interligne{&&&\\[-11pt]}
\def\fcyc #1{\fbox{\hbox{#1}}}

\begin{table}[H]
\begin{center}
\begin{tabular}{|c|c|c|c|}
\hline
Active filtrations & Active partitions & Reorientation activity classes & {Bases}   \\
\hline
\interligne
$\fcyc{\O}\subset 1\subset E$ & $1+23$ & $123$, $1\ovl{23}$,  $\ovl{1}23$, $\ovl{123}$ & 12 \\
$\fcyc{\O}\subset E$ & $123$ & $12\ovl{3}$, $\ovl{12}3$ & 13 \\
$\emptyset\subset \fcyc{E}$ & $123$ & $1\ovl{2}3$, $\ovl{1}2\ovl{3}$ & 23 \\
\hline
\end{tabular}
\caption{Table of the canonical active bijection of $K_3$, where reorientations are written with a bar over reoriented edges w.r.t. the reference orientation given in the upper left of Figure \ref{fig:K3}. The cyclic flat of each active filtration  is boxed in the first column.}
\label{table:tabK3ori}
\end{center}
\vspace{-0.5cm}
\end{table}

\begin{figure}[h]
\centering
{
\parindent=-1mm
\scalebox{1}[0.85]
{\includegraphics[width=9cm]{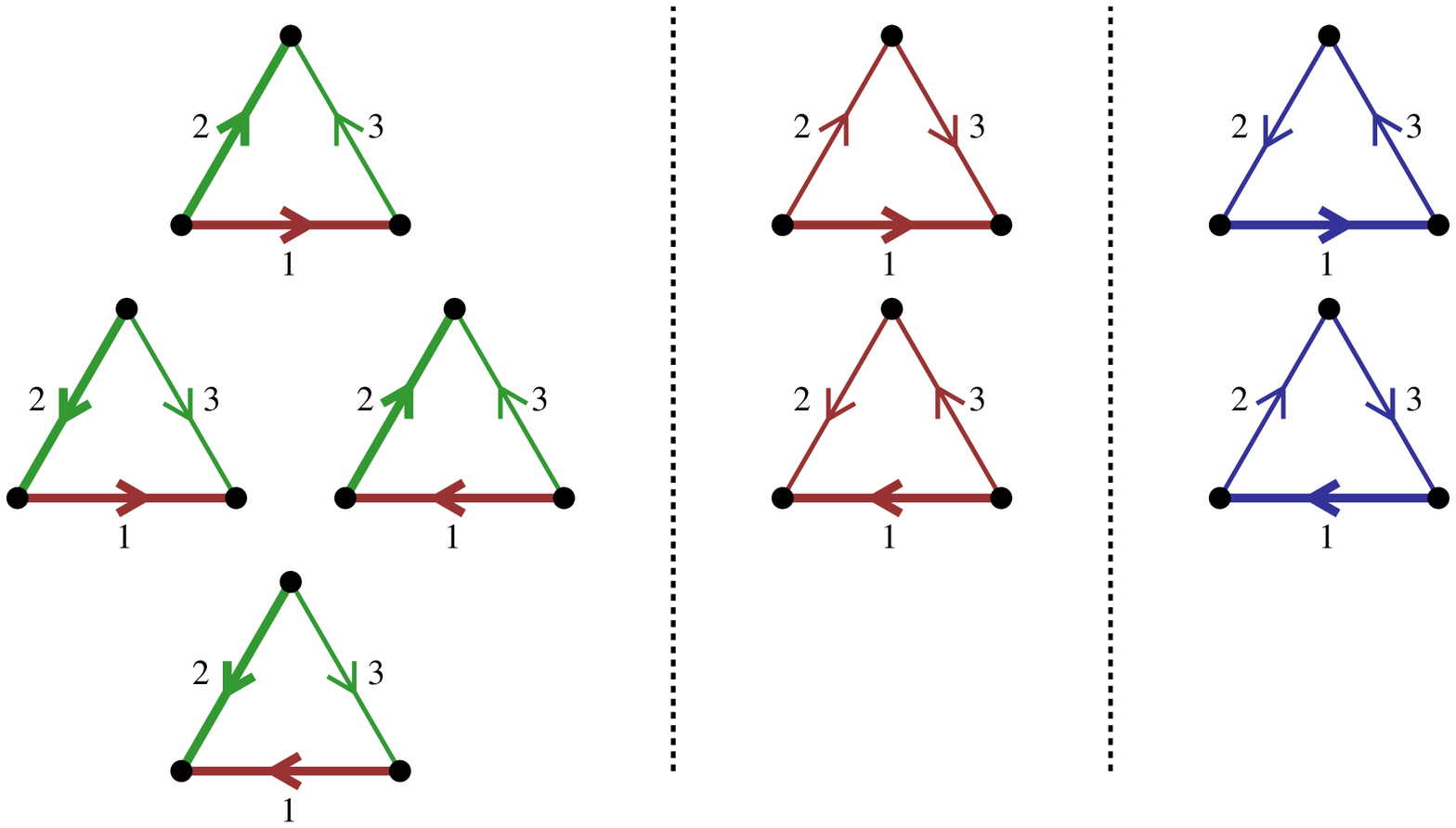}}
}
\hrule
\vspace{-1mm}
\flushleft
\begin{tabular}{l@{\hspace{15mm}}ccc;{1pt/1pt}c;{1pt/1pt}c}
 &  & $+ x^2$ & & $+x$ & $+y$\\
 $T(K_3;x+u,y+v)=$& $+xu$ & & $+ux$ & $ +u$ &$ +v$\\
  & & $+u^2$ & &   & \\
\end{tabular}
\vspace{1mm}
\hrule
\vspace{1mm}
%
\centering
{
\scalebox{1}[0.85]
{\includegraphics[width=9cm]{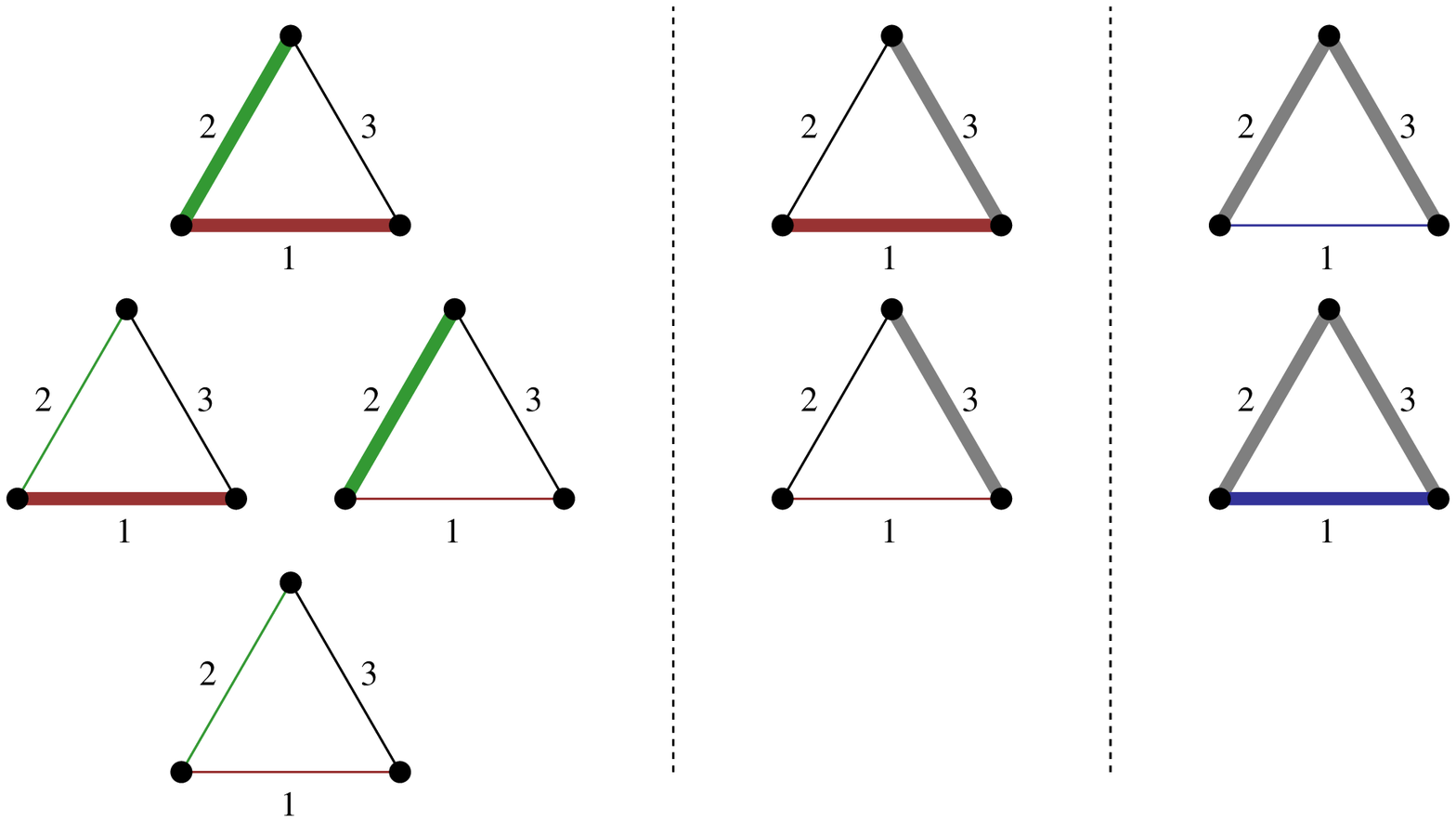}}
}
\caption{The active bijection illustrated on the graph $K_3$. We have $T(K_3;x,y)=x^2+x+y$.
The layout reflects the bijections.
Each monomial corresponds to an activity class of (re)orientations in the top part and to a basis (spanning tree) in the bottom part,
associated by the canonical active bijection.
Each basis yields a boolean lattice of subsets (shown by bold edges). Orientations in the top part and subsets in the bottom part are associated by the refined active bijection (with respect to the 
orientation 
displayed
first
in the top row), consistently with the four variable~formula, in the way shown by the layout.
%
}
\label{fig:K3}
\end{figure}

%
%
%
%


\subsection{Example of $K_4$ exhaustively completed}
\label{subsec:ex-K4}


We consider the graph $K_4$.
Its Tutte polynomial is 
$$t(K_4;x,y)=x^3+3x^2+2x+4xy+2y+3y^2+y^3.$$
We consider it
with the ordering $1<\ldots <6$ as shown in Figure \ref{fig:K4exbase256}, and with the reference orientation shown in Figure \ref{fig:K4act}.
%
Here, we complete this example 
that served as a running example in Figures \ref{fig:K4exbase256},  \ref{fig:K4act}, \ref{fig:K4-dec}, \ref{fig:ex-fob}, 
 \ref{fig:exbasedecomp-136}, \ref{fig:exbasedecomp-126},  \ref{fig:exbasedecomp-146}, \ref{fig:exbasedecomp-256},  \ref{fig:exK4-complete-primal-dual}, \ref{fig:K4-iso},
 and in \cite{AB2-a}.
%
Let us mention that the canonical and refined active bijections on this example are also exhaustively listed in a graph setting in \cite{ABG2}, whereas here we continue to present this list geometrically.

Table \ref{fig:tabK4ori} sums up the canonical active bijection (Theorem \ref{th:alpha}). 
Figures  \ref{fig:fig:K4bij-primal},  \ref{fig:fig:K4bij-dual}, and \ref{fig:fig:K4bij-cyclic-flats} provide details for Figure \ref{fig:exK4-complete-primal-dual}, indicating the graph orientations, and adding the refined active bijection w.r.t. the reference reorientation (extending  Figure \ref{fig:K4-iso} to all reorientation activity classes).
All reorientations are depicted (up to opposite), and all subsets are associated to them (in brackets when they correspond to the opposite reorientation).

\begin{table}[h]
\begin{center}
\begin{tabular}{|c|c|c|c|}
\hline
Active filtration & Active partition & Rerientation activity class & {Basis}   \\
\hline
$\fcyc{\O}\subset 1\subset 123\subset E$ & $1+23+456$ & $123456$, $1\ovl{23}456$, $123\ovl{456}$, $1\ovl{23456}$, ... & 124 \\
$\fcyc{\O}\subset 1\subset E$ & $1+23456$ & $12345\ovl{6}$, $1\ovl{2345}6$, ... & 126 \\
$\fcyc{\O}\subset 145\subset E$ & $145+236$ & $1234\ovl{56}$,$1\ovl{23}4\ovl 56$, ... & 125 \\
$\fcyc{\O}\subset 123\subset E$ & $123+456$ & $12\ovl{3456}$, $12\ovl 3456$, ... & 134 \\
$\fcyc{\O}\subset E$ & $123456$ & $12\ovl 34\ovl{56}$, ... & 135 \\
$\fcyc{\O}\subset E$ & $123456$ & $12\ovl 34\ovl{5}6$,...& 136 \\
$\emptyset\subset \fcyc{123}\subset E$ & $123+456$ & $1\ovl 23456$, $1\ovl 23\ovl{456}$, ...& 234 \\
$\emptyset\subset \fcyc{145}\subset E$ & $145+236$ & $1\ovl{234}56$, $123\ovl 45\ovl 6$, ...& 245 \\
$\emptyset\subset \fcyc{246}\subset E$ & $246+135$ & $12\ovl{34}56$, $1\ovl{23}45\ovl 6$, ...& 146 \\
$\emptyset\subset \fcyc{356}\subset E$ & $356+124$ & $1234\ovl 56$, $12\ovl 345\ovl 6$, ... & 156 \\
$\emptyset\subset \fcyc{E}$ & $123456$ & $1\ovl 23\ovl 456$, ... & 235 \\
$\emptyset\subset \fcyc{E}$ & $123456$ & $1\ovl 23\ovl 45\ovl 6$, ... & 236 \\
$\emptyset\subset 246 \subset \fcyc{E}$ & $135+246$ & $1\ovl 2345\ovl 6$, $123\ovl 456$, ... & 346 \\
$\emptyset\subset 356 \subset \fcyc{E}$ & $124+356$ & $1\ovl{234}5\ovl 6$, $1\ovl 23\ovl{45}6$, ... & 256 \\
$\emptyset\subset 23456\subset \fcyc{E}$ & $1+23456$ & $1\ovl234\ovl{56}$, $12\ovl{34}56$, ... & 345 \\
$\emptyset\subset 356\subset 23456\subset \fcyc{E}$ & $1+24+356$ & $12\ovl{34}5\ovl 6$, $1\ovl{23}45\ovl 6$,
$1\ovl 234\ovl 56$, $123\ovl{45}6$, ...& 456 \\
\hline
\end{tabular}
\caption{Table of the canonical active bijection of $K_4$ (Theorem \ref{th:alpha}), where reorientations are written with a bar over reoriented elements w.r.t. the reference reorientation (the grey region in Figure \ref{fig:fig:K4bij-primal}), and where ``...'' means ``and opposites''. The cyclic-flat of each connected filtration is boxed in the first column. This table can be compared with \cite[Table \ref{a-fig:tabK4}]{AB2-a} for bases.
\EMEvder{comparer avec AB2a et ABG2 --- retrecir ?}%
}
\label{fig:tabK4ori}
\end{center}
\vspace{-0.5cm}
\end{table}

\EMEvder{voir taille des figures, adapter egaliser ou pas ....?}%

\EMEvder{dessous figures enlevees avec petits arrangements pour primal et dual avec uste base et act part}%

\begin{figure}[]
\centering	
\includegraphics[scale=1.4]{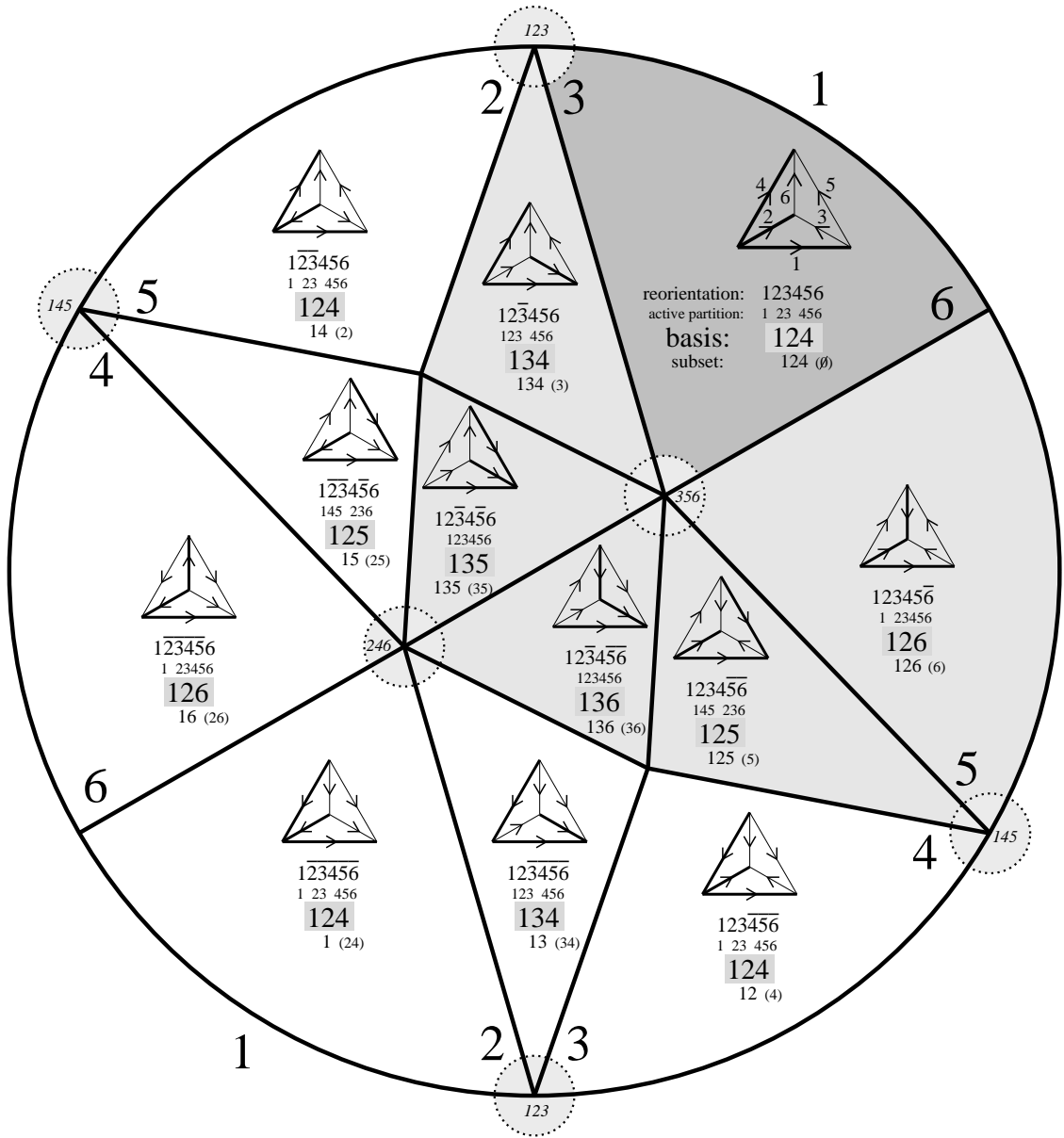}
\caption[]{Details for the primal part of Figure \ref{fig:exK4-complete-primal-dual} (connected filtrations involving cyclic flat $\emptyset$ in Table  \ref{fig:tabK4ori}).  The reference reorientation (signature of the arrangement) is given by the dark grey region. Reorientations corresponding to the regions are written w.r.t. to this reference reorientation, they correspond to the maximal covectors of the oriented matroid.
Corresponding acyclic graph orientations are also drawn in the regions.
The subset associated to the reorientation by the refined active bijection is written below the basis, the subset in brackets is associated to the opposite reorientation  (extending  Figure \ref{fig:K4-iso}). 
The dual-active-fixed representatives of activity classes of regions w.r.t. the reference orientation are shown in light grey.
\EMEvder{est-ce bien l'opposite ???}
}
\label{fig:fig:K4bij-primal}
\end{figure}

\begin{figure}[]
\centering	
\includegraphics[scale=1.2]{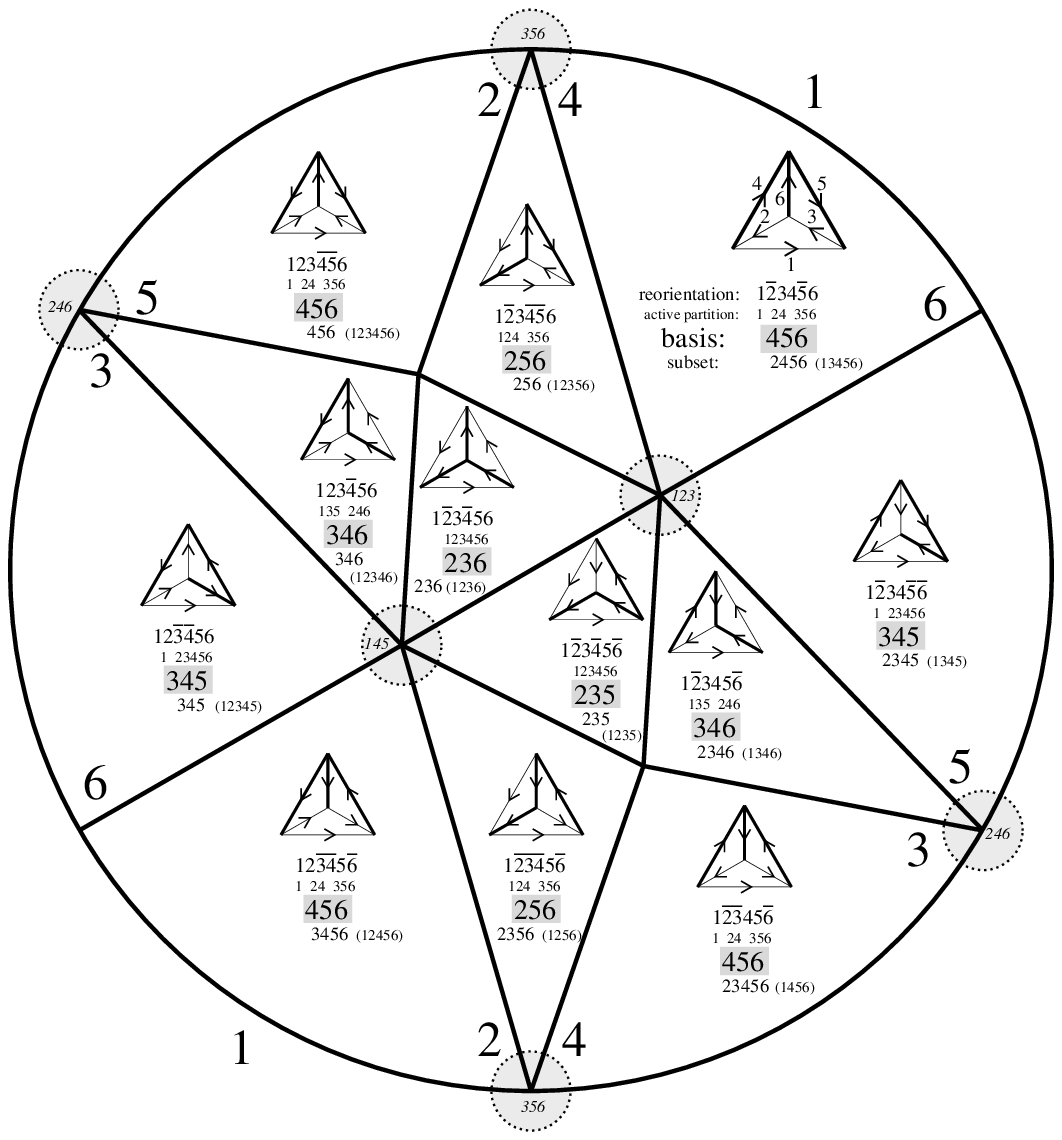}
\caption[]{Details for the dual part of Figure \ref{fig:exK4-complete-primal-dual} (connected filtrations involving cyclic flat $E$ in Table  \ref{fig:tabK4ori}).   The reference reorientation
\EMEvder{attention terme pas employ?, on dit avant reference oriented matroid, pas reference reorientation}%
is given by the dark grey region of Figure \ref{fig:fig:K4bij-primal}. Reorientations corresponding to the regions of the dual arrangement are written w.r.t. to this reference reorientation, they correspond to the maximal vectors of the oriented matroid.
Corresponding totally cyclic (or strongly connected) graph orientations are also drawn in the regions.
Subset associated to reorientations by the refined active bijection are indicated the same way as in
Figure~\ref{fig:fig:K4bij-primal}.
\EMEvder{mettre aussi representatives du dual?}%
}
\label{fig:fig:K4bij-dual}
\end{figure}

\begin{figure}[]
\centering	
\includegraphics[scale=1.1]{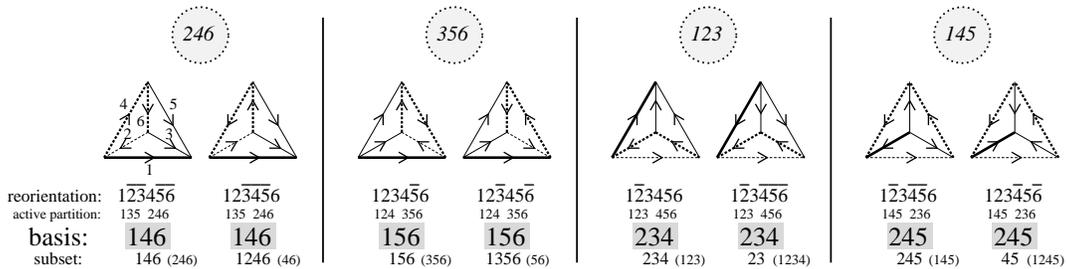}
\caption[]{Details and refined active bijection for the active filtrations involving non-trivial cyclic flats in Figure \ref{fig:exK4-complete-primal-dual} and Table \ref{fig:tabK4ori}, completing Figures \ref{fig:fig:K4bij-primal} and \ref{fig:fig:K4bij-dual}. 
Those cyclic flats are shown in dashed grey circles, and geometrically situated in Figures 
 \ref{fig:fig:K4bij-primal} and \ref{fig:fig:K4bij-dual}, following Figure \ref{fig:exK4-complete-primal-dual}. They appear with dashed edges on the graphs. 
 }
\label{fig:fig:K4bij-cyclic-flats}
\end{figure}

\subsection{Another similar  example of the canonical active bijection.}
\label{subsec:ex-another-graph}
In Figure \ref{fig:ex-graphe-autre-primal-dual}, we address another example of rank 3 with 6 elements, and we completely show the geometry of the canonical active bijection, involving all bases and all reorientations (up to opposite), in the primal and dual arrangements along with combinations of cyclic flats of the primal and the dual. Its Tutte polynomial is:
$$t(M;x,y)=x^3+2x^2+x+x^2y+3xy+xy^2+y+2y^2+y^3.$$
The ordering of the ground set is the natural ordering.
As for Figure \ref{fig:exK4-complete-primal-dual}, no reference reorientation or signature is specified. We focus on the geometry but this example is graphical and planar again, hence with a graphical dual, as shown on the figure.
The reader interested in the graph setting can easily draw which acyclic and totally cyclic orientations of the graph and its dual correspond to regions of the two arrangements.  The refined active bijection can also be easily deduced, as in Figures \ref{fig:K4-iso},  \ref{fig:fig:K4bij-primal},  \ref{fig:fig:K4bij-dual}, and \ref{fig:fig:K4bij-cyclic-flats}.


\begin{figure}[]
\centering
{\scalebox{2}
{\includegraphics[width=7.5cm]{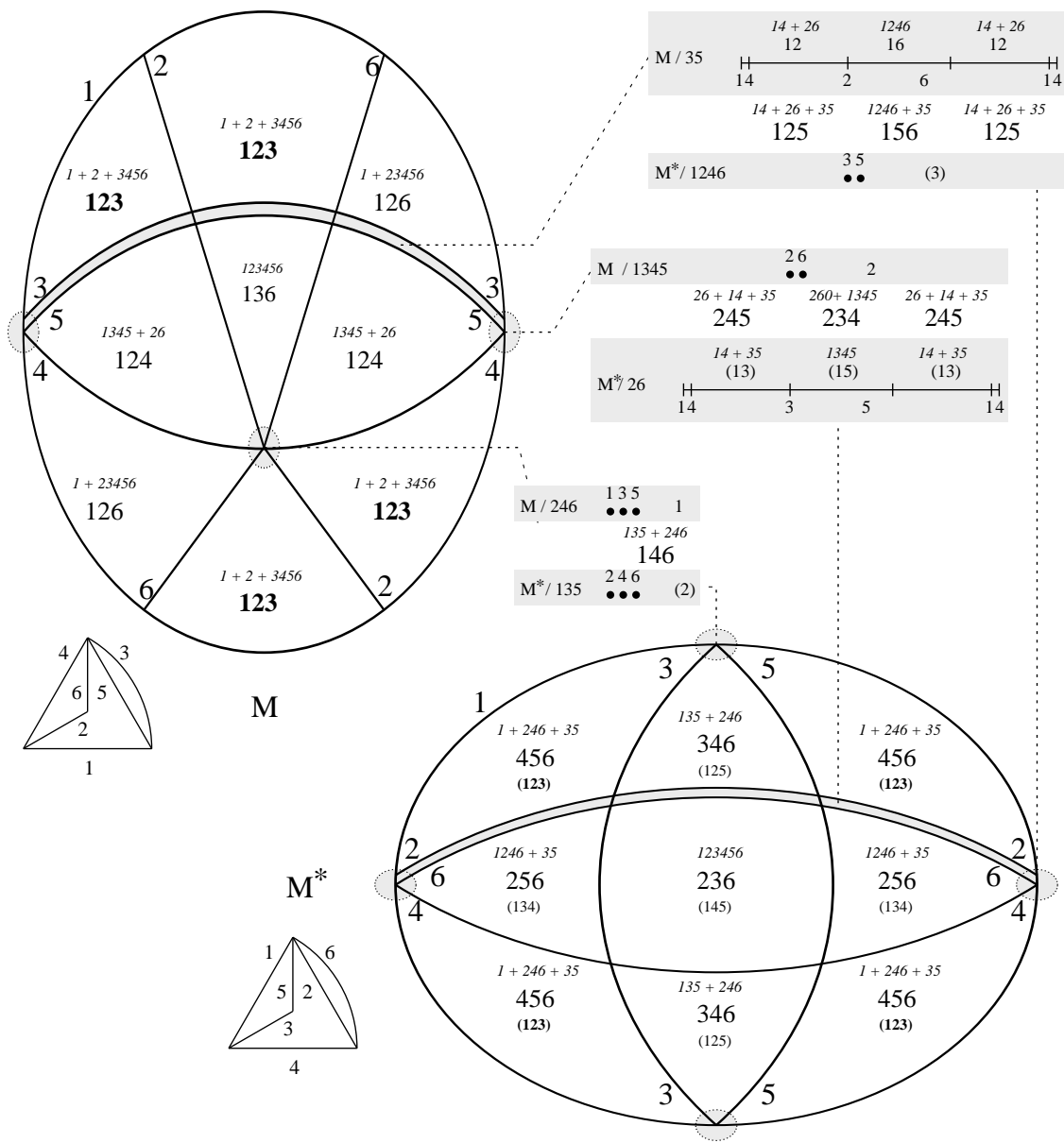}}
}
\caption{Complete primal/dual geometrical representation of the canonical active bijection on another example (see Section \ref{subsec:ex-another-graph}).  We follow the same caption as in Figure \ref{fig:exK4-complete-primal-dual}. 
Minors obtained from non-trivial cyclic flats  are represented and linked to their representations in the primal and dual arrangements, in order to show how regions of these minors are involved in the construction.
Namely: $35$, $1345$, and $246$ are those cyclic flats of $M$, corresponding to the cyclic flats $1246$, $26$ and $135$ of $M^*$.
Observe for instance 
that
the two bases $125$ and $245$ yield to consider the same partition $E=14+26+35$,
the same sequence of subsets $\emptyset\subset 35 \subset 1435 \subset E$, and the same minors $M/1345$, $M(1345)/35$ and $M(35)$, when one omits the associated cyclic flat, whereas on one hand
the active filtration of the basis $125$  is $(\emptyset, 35, 35, 35, 1435, E)$ with cyclic flat $35$, yielding $\Int(125)=12$ and $\Ext(125)=3$ and involving an acyclic reorientation of  $M(1345)/35$,
and on the other hand the active filtration of the basis 245 is $(\emptyset, 35, 1435, 1435, 1435, E)$ with cyclic flat  $1435$, yielding $\Int(245)=2$ and $\Ext(245)=13$ and involving a totally cyclic reorientation of  $M(1345)/35$ (that is an acyclic reorientation of  $M^*(1246)/26$).
%
\EMEvder{phrase precedente hyper longue... a scinder!}%
}
\label{fig:ex-graphe-autre-primal-dual}
\end{figure}

\EMEvder{desous en commentaires, figures avec orientations detaillees de l'exemple de figure precedente, alourdit inutilement, a metter eventeullement dans un article complementarie ?}%
%
%

\subsection{Example of the canonical active bijection on regions of a rank 3 (supersolvable) arrangement.}
\label{subsec:supersolv}

Figure \ref{fig:ex-arrgt} illustrates the canonical active bijection on regions of a rank 3 arrangement with 11 elements, highlighting its geometrical interpretation in terms of flags of faces.
This example is intended to be simple and pedagogic. Explanations are given in the figure caption.
The ordering is $1<2<\ldots<9<A<B$. 
The restriction of the Tutte polynomial in which we are interested is:
$$t(M;x,0)=x^3+8x^2+16x.$$

\begin{figure}[]
\centering
{
\scalebox{1.7}
{\includegraphics[width=7.5cm]{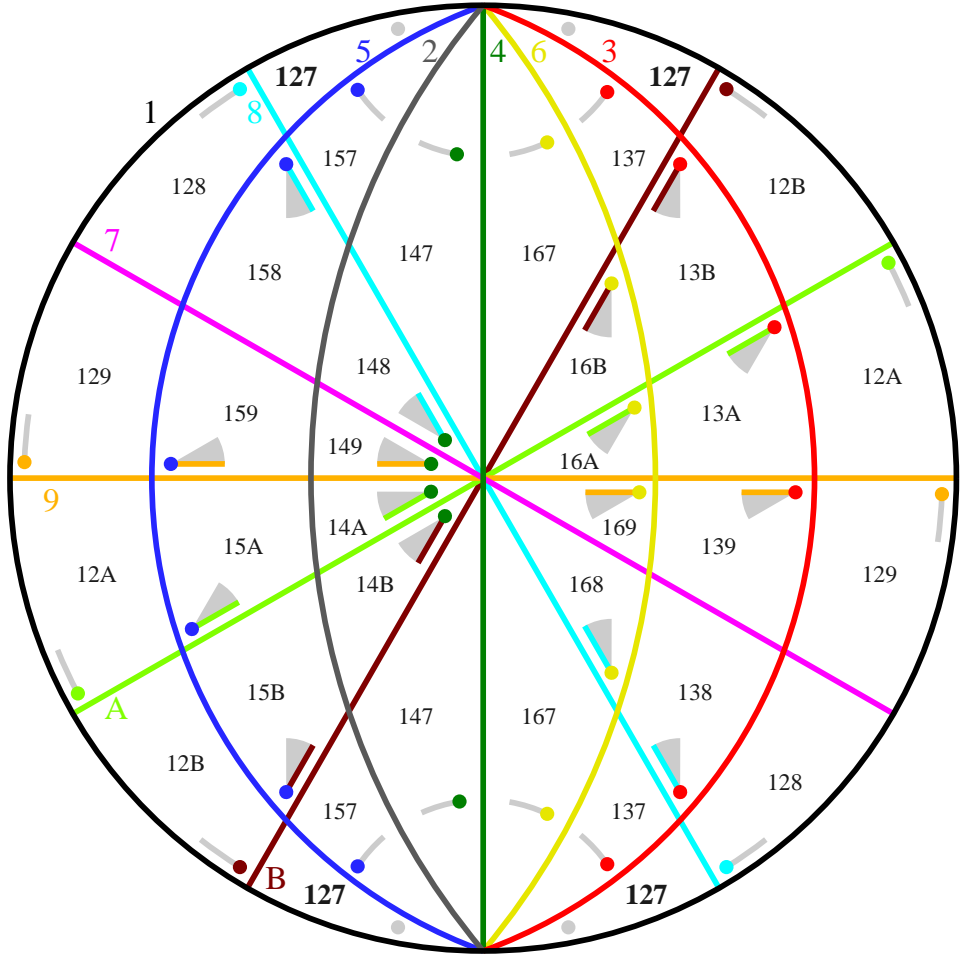}}
}
\caption{The ordering is $1<2<\dots<9<A<B<C<D$. In each region, we write the active basis, associated to the region by the canonical active bijection. We deal only with acyclic reorientations and internal bases 
(thus, we will not specify the cyclic flat of involved active filtrations as it is the empty one). 
The minimal basis (playing a crucial role) is $127$, it is associated with the four regions adjacent to $1$ and $1\cap 2$ (with dual-active elements $\{1,2,7\}$ and active filtration $\emptyset \subset 1\subset 123456\subset E$). In each bounded region (with dual-active element $1$), the fully optimal basis $B=1<a<b$ is  represented by a flag of faces: the region contains a segment of $b$ which contains a point of $a\cap b$. This sequence corresponds to the sequence of useful covectors: $C^*(B;1)\circ  C^*(B;a) \circ  C^*(B;b) \supset C^*(B;1)\circ  C^*(B;a)  \supset C^*(B;1)$.
Those flags of faces illustrate the Adjacency property from Definition \ref{def:acyc-alpha2},
and they are also intended to illustrate the optimality feature of these bases (intuitively, one can imagine mobile flags in the regions, moving  until they reach their fully optimal position: they are first ``pushed'' from $2$ towards $1$, and then from $7$ towards 1 when a face parallel to $2$ is reached, or from $7$ towards $4=\min(4789AB)$ in the case where the central point is reached,  see details in \cite{AB1, AB3}). 
Smaller flags in non-bounded regions are intended to illustrate the same features for minors involved in the active decomposition of regions and bases.
In each region adjacent to $1$ but not $1\cap 2$ (with dual-active elements $\{1,2\}$), the basis is of the type  $B=1<2<b$. Because of the active filtration $\emptyset\subset 1\subset E$, the basis is obtained from the fully optimal basis $2b$ for the bounded region induced on $M/1$ (added to $1$ which is the only internal basis of $M(1)$). This basis is represented by a similar flag as above, but in one dimension less for this minor (those smaller flags are intuitively ``pushed'' from 7 towards $2$ along $1$). 
In each region adjacent to $1\cap 2$ but not $1$ (with dual-active elements $\{1,7\}$), the basis is of the type  $B=1<a<7$. Because of the active filtration $\emptyset\subset 123456\subset E$, the basis is obtained from the fully optimal basis $1a$ for the bounded region induced on $M(123456)$ (added to $7$ which is the only internal basis of $M/123456$). This basis is represented by a similar flag as above, but in one dimension less for this minor  (those smaller flags are intuitively ``pushed'' from 2 towards $1$ around $1\cap 2$).
In addition, this arrangement has the specificity of being supersolvable, yielding specific constructive properties, see details in Subsection \ref{subsec:supersolv}. This figure is best viewed with colors.
\red{}}
\label{fig:ex-arrgt}
\end{figure}

In addition, independently, let us mention that this arrangement is supersolvable, a case studied into the details in \cite{GiLV06}.
Briefly, we have a sequence of three arrangements $M(1)\triangleleft M(123456)\triangleleft M$, such that these arrangement have increasing ranks and the intersection of  two pseudospehres in an arrangement is contained in a pseudoshere of the previous arrangement in the sequence.
Then regions of the arrangement are grouped in \emph{fibers} corresponding to regions in the previous arrangement and linearly ordered in these fibers.
Provided some consistency with the ordering,  the active bijection can be built recursively, in each fiber, using the ordering of regions in this fiber. It is related to the general deletion/contraction of the active bijection developed in \cite{AB4}.

In this particular example, one can see that all regions in the fiber delimited by $1$ and $5$, by $5$ and $2$, by $2$ and $4$, etc., are associated with bases that respectively contain $12$, $15$, $14$, etc.
Now, in each fiber, the basis associated to the two extreme regions  contains $7$ (the smallest of $M\s 123456$). For the other regions, the 
missing element of the associated basis is given by the element of $M\s 123456$ delimiting the region on the opposite side of $7$.

In general supersolvable arrangements, a similar construction holds at least for all bounded regions (coming from a general property of the active bijection that, for bounded regions, the greatest element of the fully optimal basis always borders the region \cite{AB4}).
The fact that, in this example, this construction directly holds for every region, including all non-boudned ones, is a further particulartiy of this example.  Here, in terms of \cite{GiLV06, AB4}, we have that the active mapping equals the \emph{weak active bijection} (which does not preserve active partitions in general).
Let us  mention that a more involved construction in rank 4 is given in \cite[Figure 4]{GiLV06}, which refines Figure \ref{fig:rank4-act-part} of the present paper, and shows how the active bijection is related to active partitions in a fiber of non-bounded regions.
Finally, beyond this, let us mention that the supersolvable structure of Coxeter arrangements can be used to show that the active bijections yields bijections between permutation and increasing trees (braid arrangement), and between signed permutations and signed increasing trees (hyperoctahedral arrangement). See details in \cite{GiLV06}.

\subsection{Example of the canonical and refined active bijections on regions of $D_{13}$}
\label{subsec:D13}

We end in Figure \ref{fig:D13refined} with regions of a more involved rank-3 example, from which the reader might foresee how the construction gets more complicated in higher dimensions.
This arrangement, which we call $D_{13}$, is obtained by adding 3 points $BCD$ to a Desargue configuration on $123456789A$,
see \cite[Example 4.1.1]{GiLV04} for a picture and more details.
This example has been detailed in our first paper \cite{GiLV04}, devoted to the uniform case and the rank-3 case%
. 
Beware that, as a marginal change w.r.t. \cite{GiLV04} (and \cite{GiLV07}), here we exchanged the roles $5$ and $6$ (so that $M(1A7B)$ and $M(156C)$ present different shapes).
The ordering is $1<2<\ldots<9<A<B<C<D$.
The restriction of the Tutte polynomial in which we are interested  is:
$$t(D_{13};x,0)=x^3+10x^2+24x.$$

\begin{figure}[]
\centering
{\scalebox{2}
{\includegraphics[width=7.5cm]{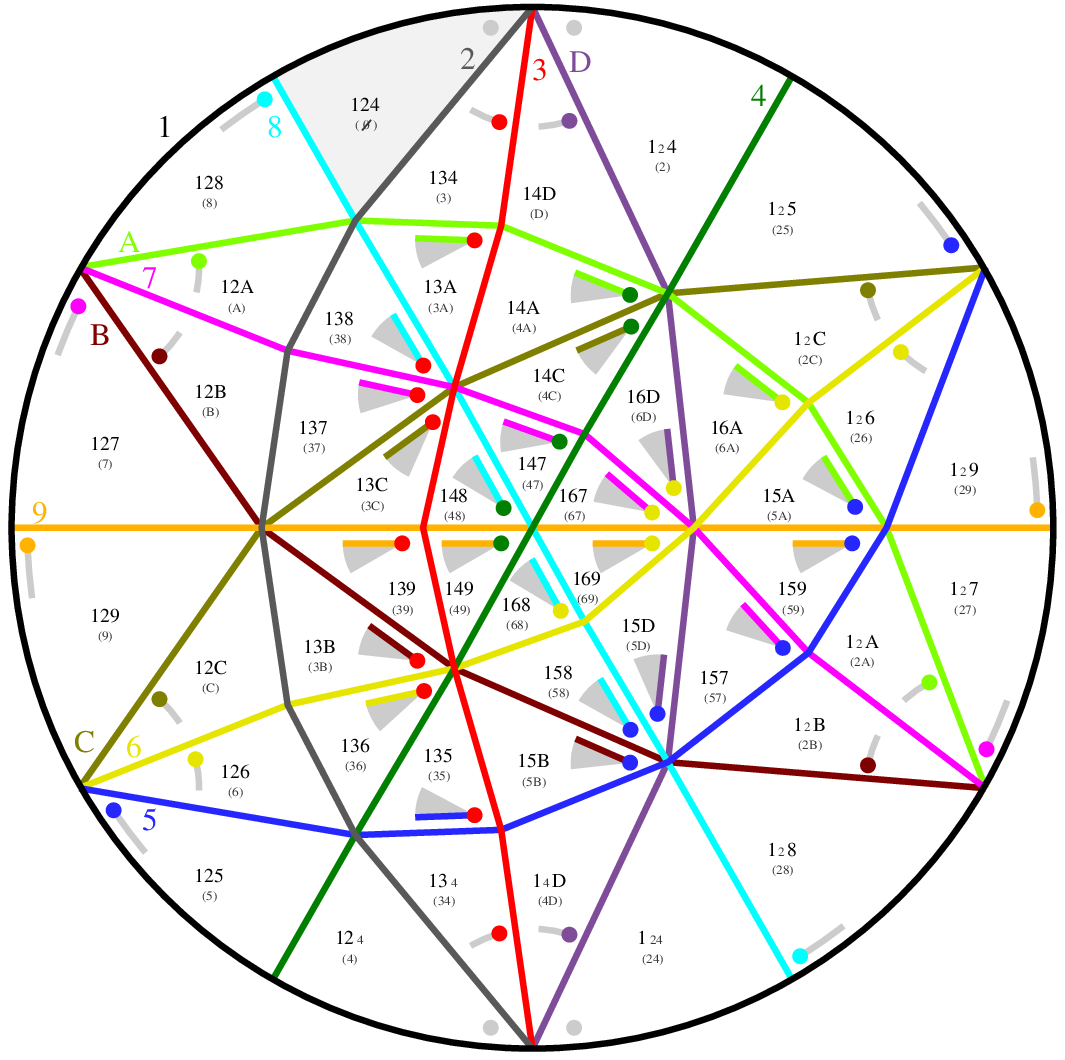}}
}
\caption{The ordering is $1<2<\dots<9<A<B$. Canonical and refined active bijections on regions of a more involved rank-3 example than Figure \ref{fig:ex-arrgt}, following the same caption. 
Here again, in each region, we write the basis associated by the canonical active bijection, but we write in smaller size the elements that should be removed to obtain the subset associated to the region by the refined active bijection w.r.t. the grey region as reference reorientation. Opposite   regions (on the other non-represented half of the sphere) are associated to the subsets in brackets (if a region with dual active elements $O^*$ is associated with the subset $A$, then the opposite region is associated with the subset $A \triangle O^*$, just as in Figure \ref{fig:K4-iso}). This yields the bijection between regions and no-broken-circuit subsets.
The pictures of flags of faces have the same meaning as in Figure \ref{fig:ex-arrgt}, but more types of situations occur. 
%
For bounded regions: the construction can be intuitively understood in a similar way  as in Figure \ref{fig:ex-arrgt}.
Let us detail non-bounded regions: the minimal basis is $124$; bases $125$, $127$, $128$, and $129$ with active elements $\{1,2\}$  are obtained from the active filtration $\emptyset\subset 1\subset E$, implying the minor $M/1$;   bases $12A$ and $12B$ with active elements $\{1,2\}$ are obtained from the active filtration $\emptyset\subset 17AB\subset E$, implying the minor $M(1A7B)$;
bases $126$ and $12C$ with active elements $\{1,2\}$ are obtained from the active filtration $\emptyset\subset 156C\subset E$, implying the minor $M(156C)$; bases $134$ and $14D$ with active elements $\{1,4\}$ are obtained from the active filtration $\emptyset\subset 123D\subset E$, implying the minor $M(123D)$.
See Section \ref{subsec:D13} for more information on this example.
This figure is best viewed with colors.
\EMEvder{attention gris de regiontrop clair ! a tester}
}
\label{fig:D13refined}
\end{figure}

\newpage
Geometrical explanations on the active bijection 
are provided in the caption of Figure  \ref{fig:D13refined}, in the continuation of 
Figure \ref{fig:ex-arrgt}.
Finally, let us mention that
other technical details are given in \cite{GiLV04}, where the possible siuations in rank-3 arrangements are differently but exhaustively listed, and also illustrated on this example%
\footnote{ 
Let us take opportunity of this paper to make two corrections  to \cite{GiLV04}:

\noindent - in \cite{GiLV04} page 231 line 2: instead of ``acting symmetrically on 1457 with three orbits 1457 23689A BCD'' read ``acting symmetrically on 1357 with three orbits 1357 24689A BCD''.

\noindent - in \cite{GiLV04} page 236 Figure 7: in the region corresponding to the basis 136, the dark angle should touch the pseudoline 6 instead of the pseudoline 3.
}%
.
The optimization feature is briefly explained in the caption of Figure~\ref{fig:ex-arrgt}.
This feature was  briefly addressed in \cite{GiLV04}, in order to show that, in rank-3 arrangements, the active bijection is the only bijection between bounded regions and uniactive internal basis satisfiying the Adjacency property of Definition \ref{def:acyc-alpha2} (illustrated here by flags of faces).
As shown in \cite{GiLV04}, this uniqueness property is also true in realizable uniform oriented matroids, but is false in non-euclidean oriented matroids, where the Dual-Adjacency property of Definition \ref{def:acyc-alpha2} cannot be replaced by the property of having a bijection (see also \cite{AB4} for a summary of properties implying the active bijection).
This uniqueness property is generalized to realizable oriented matroids in \cite{AB3}.
Also, the optimization features, roughly introduced in Figures  \ref{fig:ex-arrgt} and  \ref{fig:D13refined} using flags of faces,
are addressed into the details in \cite{AB3} (see also \cite{GiLV09} for a brief formal presentation), where one can also find 
an illustration on a rank-4 bounded region.

\EMEvder{dessous phrase beaucoup pus corute que tout ce qui rpecede pou resumer}
\EMEvder{For bounded regions: the construction can be understood from the Adjacency property of from the optimality viewpoint intuitively described in Figure \ref{fig:ex-arrgt}, we mention that the different rank-3 situations are addressed with a geometrical viewpoint and illustrated on the same example in \cite{GiLV04}. 
}
%









\vspace{-0.2cm}

\bibliographystyle{amsplain}




\input{AB2b-biblio.bbl}


\end{document}

%% file: AB2b-biblio.bbl
\providecommand{\bysame}{\leavevmode\hbox to3em{\hrulefill}\thinspace}
\providecommand{\MR}{\relax\ifhmode\unskip\space\fi MR }
\providecommand{\href}[2]{#2}